\documentclass[11pt]{amsart}
\usepackage{amsmath}
\usepackage{amssymb}
\usepackage{amsthm}
\usepackage{latexsym}
\usepackage{graphicx}
\usepackage{hyperref}
\usepackage{enumerate}
\usepackage[all]{xy}
\usepackage{chngcntr}

\newtheorem{thm}{Theorem}[section]
\newtheorem{cor}[thm]{Corollary}
\newtheorem{lem}[thm]{Lemma}
\newtheorem{prop}[thm]{Proposition}

\newtheorem{thmintro}{Theorem}

\theoremstyle{definition}

\newtheorem{rem}[thm]{Remark}

\newcommand{\N}{\mathbb N}
\newcommand{\Z}{\mathbb Z}
\newcommand{\Q}{\mathbb Q}
\newcommand{\R}{\mathbb R}
\newcommand{\C}{\mathbb C}

\newcommand{\mf}{\mathfrak}
\newcommand{\mc}{\mathcal}
\newcommand{\mb}{\mathbf}
\newcommand{\mh}{\mathbb}

\def\Irr{{\rm Irr}}
\newcommand{\mr}{\mathrm}
\newcommand{\ind}{\mathrm{ind}}
\newcommand{\enuma}[1]{\begin{enumerate}[\textup{(}a\textup{)}] {#1} \end{enumerate}}
\newcommand{\Fr}{\mathrm{Frob}}
\newcommand{\Sc}{\mathrm{Sc}}
\newcommand{\ad}{\mathrm{ad}}

\newcommand{\nr}{\mathrm{nr}}

\newcommand{\cpt}{\mathrm{cpt}}
\newcommand{\Rep}{\mathrm{Rep}}
\newcommand{\Res}{\mathrm{Res}}

\newcommand{\der}{\mathrm{der}}
\newcommand{\red}{\mathrm{red}}
\newcommand{\matje}[4]{\left(\begin{smallmatrix} #1 & #2 \\ 
#3 & #4 \end{smallmatrix}\right)}

\newcommand{\Mod}{\mathrm{Mod}}
\newcommand{\Hom}{\mathrm{Hom}}
\newcommand{\End}{\mathrm{End}}

\newcommand{\isom}{\xrightarrow{\sim}}

\newcommand{\IM}{\mathrm{IM}}

\newcommand{\he}{\dagger}

\begin{document}

\title[On principal series representations of quasi-split $p$-adic groups]{
On principal series representations\\ of quasi-split reductive $p$-adic groups}
\date{\today}
\maketitle
\vspace{2mm}

\begin{center}
{\Large Maarten Solleveld} \\[1mm]
IMAPP, Radboud Universiteit Nijmegen\\
Heyendaalseweg 135, 6525AJ Nijmegen, the Netherlands \\
email: m.solleveld@science.ru.nl
\end{center}
\vspace{3mm}

\begin{abstract}
Let $G$ be a quasi-split reductive group over a non-archimedean local field. We establish a
local Langlands correspondence for all irreducible smooth complex $G$-representations in the
principal series. The parametrization map is injective, and its image is an explicitly
described set of enhanced L-parameters. Our correspondence is determined by the choice of
a Whittaker datum for $G$, and it is canonical given that choice. 

We show that our parametrization satisfies many expected properties, among others with respect
to the enhanced L-parameters of generic representations, temperedness, cuspidal supports
and central characters. Our correspondence lifts to a categorical level, where it makes the
appropriate Bernstein blocks of $G$-representations naturally equivalent to module categories
of Hecke algebras coming from Langlands parameters. Along the way we characterize genericity 
of $G$-representations in terms of representations of an affine Hecke algebra.
\end{abstract}
\vspace{2mm}

\tableofcontents

\section*{Introduction}

Consider a quasi-split reductive group $G = \mc G (F)$ over a non-archimedean local field $F$.
Let $\Rep (G)$ be the category of smooth $G$-representations on complex vector spaces and let
$\Irr (G)$ be the set of equivalence classes of irreducible representations in $\Rep (G)$. 
The conjectural local Langlands correspondence (LLC) asserts that
$\Irr (G)$ is canonically partitioned into finite L-packets $\Pi_\phi (G)$, indexed by
L-parameters $\phi$. Some time after the initial formulation in \cite{Bor2}, it was realized 
that $\Pi_\phi (G)$ should be parametrized by the
set of irreducible representations of a finite component group $R_\phi$. These conjectures have
motivated a large part of the study of reductive groups over local fields in past decades,
see the survey papers \cite{ABPS3,Bor2,Kal2,Kud,Vog}.

This paper establishes a local Langlands correspondence for the most accessible class of 
$G$-representations, those in the principal series. To formulate the result precisely,
we quickly recall some relevant notions.

Let $T \subset G$ be the centralizer of a maximal $F$-split torus in
$G$, or equivalently a minimal Levi subgroup of $G$. Then $T$ is itself a torus because $G$
is quasi-split, and $T$ is unique up to conjugation. Any representation of $G$ that can be
obtained from a smooth representation of $T$ by parabolic induction and then taking a
subquotient, is called a principal series $G$-representation. These representations form a
product of Bernstein blocks in $\Rep (G)$. We denote the set of irreducible principal series
$G$-representations by $\Irr (G,T)$. We warn that some L-packets contain elements of 
$\Irr (G,T)$ and also other elements of $\Irr (G)$.

It has turned out that the representation $\rho_\pi$ of $R_\phi$ associated to a given 
$\pi \in \Irr (G)$ is not canonically determined. To specify it uniquely one needs additional
input, namely a Whittaker datum for $G$. Such a Whittaker datum can be used to normalize
relevant intertwining operators, which then determine exactly how $\rho_\pi \in \Irr (R_\phi)$
is related to $\pi$. For non-quasi-split groups $G$ such a normalization should also be
possible \cite[Conjecture 2.5]{Kal2}, but it is much more involved.

We fix a Borel subgroup $B = TU$ and a nondegenerate character $\xi$ of the unipotent radical
$U$ of $B$. Then $(U,\xi)$, or rather its $G$-conjugacy class, is a Whittaker datum for $G$.
Recall that $\pi \in \Irr (G)$ is called $(U,\xi)$-generic if $\Hom_U (\pi,\xi)$ is nonzero.

Let $\mb W_F$ be the Weil group of $F$, let $G^\vee$ be the complex dual group of $G$ and
let ${}^L G = G^\vee \rtimes \mb W_F$ be the Langlands dual group. In this introduction
(but not in the body of the paper) we realize L-parameters for $G$ as Weil--Deligne 
representations 
\[
\phi : \mb W_F \ltimes \C \to {}^L G .
\]
Here $\mb W_F$ acts on $\C$ by $w \cdot z = \| w \| z$, where $\| w \| \in p^\Z$ is determined 
by $w (x) = x^{\| w \|}$ for all $x$ in an algebraic closure of the residue field of $F$.
The appropriate component group of such an L-parameter is 
\[
R_\phi = \pi_0 \big( Z_{G^\vee} (\phi (\mb W_F \ltimes \C)) / Z(G^\vee)^{\mb W_F} \big) ,
\]
where $Z(G^\vee)$ denotes the centre of $G^\vee$. An enhancement of $\phi$ is an irreducible 
$R_\phi$-representation. Let $\Phi_e (G)$ be the set of enhanced L-parameters for $G$, considered 
up to $G^\vee$-conjugacy. An element $(\phi,\rho) \in \Phi_e (G)$ belongs to the principal 
series if its cuspidal support is an enhanced L-parameter for $T$. More explicitly, that means
\begin{itemize}
\item $\phi (\mb W_F) \subset T^\vee \rtimes \mb W_F$ (or for a $G^\vee$-conjugate of 
$T^\vee \rtimes \mb W_F$, because $\phi$ is only given up to $G^\vee$-conjugacy),
\item $\rho$ appears in the homology of a certain variety of Borel subgroups.
\end{itemize}
We denote the subset of $\Phi_e (G)$ associated to the principal series by $\Phi_e (G,T)$.
For a given $\phi$ it may happen that some enhancements yield elements of $\Phi_e (G,T)$, while 
other enhancements bring us outside $\Phi_e (G,T)$. 

Our main result is a canonical LLC for principal series representations:

\begin{thmintro}\label{thm:A} \textup{[see Section \ref{sec:5}]} \\
The Whittaker datum $(U,\xi)$ determines a canonical bijection
\[
\begin{array}{ccc}
\Irr (G,T) & \longleftrightarrow & \Phi_e (G,T) \\
\pi (\phi,\rho) & \text{\reflectbox{$\mapsto$}} & (\phi,\rho) \\
\pi & \mapsto & (\phi_\pi,\rho_\pi)
\end{array}
\]
with the following properties:
\enuma{
\item $\pi (\phi,\rho)$ is $(U,\xi)$-generic if and only if $\rho$ is trivial and
$u_\phi = \phi (1,1)$ lies in the dense $Z_{G^\vee}(\phi (\mb W_F))$-orbit in
\[
\{ v \in G^\vee: v \text{ is unipotent and } \phi (w) v \phi (w)^{-1} = v^{\| w \|}
\text{ for all } w \in \mb W_F \} .
\]
\item $\pi (\phi, \rho)$ is tempered (resp. essentially square-integrable) if and only if
$\phi$ is bounded (resp. discrete).
\item The bijection is compatible with the cuspidal support maps on both sides.
\item The bijection is equivariant for the canonical actions of 
$H^1 ( \mb W_F, Z(G^\vee))$.
\item The bijection is compatible with the Langlands classification and (for tempered 
representations) with parabolic induction.
}
All Borel's desiderata from \cite[\S 10]{Bor2} are satisfied. When $\pi$ is given, $\phi_\pi$ 
is uniquely determined by (a)--(e) and the local Langlands correspondence for tori.
\end{thmintro}

For non-split quasi-split groups, the vast majority of the groups under consideration here,
very little in this directon was previously known. On the other hand, for split groups many 
instances of Theorem \ref{thm:A} have been established before:
\begin{itemize}
\item Kazhdan and Lusztig \cite{KaLu} proved the bijection and (b,e) 
for Iwahori-spherical representations, assuming that $G$ is $F$-split and that $Z(G)$ is 
connected as algebraic group. Their starting point is Borel's description \cite{Bor1} of
those representations, in terms of Hecke algebras.
\item Reeder \cite{Ree} extended \cite{KaLu} to $\Irr (G,T)$ when $G$ is $F$-split, $Z(G)$
is connected and the residual characteristic $p$ of $F$ is not ``too small". This is based on 
work of Roche \cite{Roc1} and includes (a,b,e). We note that here the Whittaker datum is 
unique up to $G$-conjugacy because $Z(G)$ is connected.
\item In \cite{ABPSprin} a (noncanonical) bijection satisfying (b,d,e) was established for 
$\Irr (G,T)$, when $G$ is $F$-split and $p$ is not too small.
\item For quasi-split unitary groups with $p > 2$ a (noncanonical) bijection was 
constructed by the author's PhD student Badea \cite{Bad}.
\end{itemize}
In all cases, a study of affine Hecke algebras constitutes the largest part of the argument. 
Thanks to \cite{ABPS3,SolEnd}, that technique is now available in complete generality (even
outside the principal series). The main novelties of this paper are:
\begin{itemize}
\item The construction of the LLC is canonical and uniform, over all \\
non-archimedean local fields $F$ and all quasi-split reductive $F$-groups.
\item We can handle generic representations, even when not all Whittaker data for $G$ are
equivalent.
\item Our LLC lifts to a categorical level, as follows. For each involved Bernstein block
of $G$-representations, the LLC comes from a canonical equivalence between that block and the 
module category of a certain Hecke algebra defined entirely 
in terms of Langlands parameters.
\end{itemize}
We will now discuss the content of the paper in more detail, at the same time explaining
parts of the proof of the main theorem.

We start with a Bernstein block $\Rep (G)^{\mf s}$ in the principal series, and a progenerator
$\Pi_{\mf s}$ thereof. Via \cite{SolEnd} $\Rep (G)^{\mf s}$ is equivalent to the module 
category of some Hecke algebra $\End_G (\Pi_{\mf s})^{op}$, which we analyse in Section 
\ref{sec:1}. We show that $\End_G (\Pi_{\mf s})$ is isomorphic to an affine Hecke algebra 
$\mc H (\mf s)$ (extended with a twisted group algebra), and we determine its $q$-parameters. 

In Section \ref{sec:Whit} we involve the Whittaker datum, and that enables us to make the 
aforementioned isomorphism canonical. In the same way we show that the twist in the extension 
part of $\mc H (\mf s)$ is actually trivial, so that it is an extended affine Hecke
algebra $\mc H (\mf s)^\circ \rtimes \Gamma_{\mf s}$. Here the Bernstein group $W_{\mf s}$ 
associated to $\Rep (G)^{\mf s}$ appears as $W(R_{\mf s}^\vee) \rtimes \Gamma_{\mf s}$ for
a root system $R_{\mf s}^\vee$. 

A continuation of this analysis yields a useful criterion for genericity in terms of Hecke
algebra modules. Let $\mc H (W (R_{\mf s}^\vee), q_F^\lambda) \subset \mc H (\mf s)^\circ$
be the finite dimensional Iwahori--Hecke algebra from the Bernstein presentation of 
$\mc H (\mf s)^\circ$, where $q_F$ is the cardinality of the residue field of $F$ and 
$\lambda : R_{\mf s}^\vee / W_{\mf s} \to \Z_{\geq 0}$ is a label function. 
Recall that its Steinberg representation St is given by 
$T_{s_\alpha} \mapsto -1$ for every simple reflection $s_\alpha \in W(R_{\mf s}^\vee)$.
Let $\det : W_{\mf s} \to \{\pm 1\}$ be the determinant of the action of $W_{\mf s}$ on
the lattice of $F$-rational cocharacters of $T$. We extend St to a onedimensional 
representation (still denoted St) of $\mc H (W (R_{\mf s}^\vee), q_F^\lambda) \rtimes 
\Gamma_{\mf s}$ by making it $\det$ on $\Gamma_{\mf s}$.

\begin{thmintro}\label{thm:B} \textup{[see Theorem \ref{thm:6.4}]}\\
Suppose that $\pi \in \Rep (G)^{\mf s}$ has finite length. Then $\pi$ is $(U,\xi)$-generic
if and only if the associated $\mc H (\mf s)^{op}$-module $\Hom_G (\Pi_{\mf s},\pi)$
contains the Steinberg representation of 
$(\mc H (W (R_{\mf s}^\vee), q_F^\lambda) \rtimes \Gamma_{\mf s})^{op}$.
\end{thmintro}

The notion of principal series enhanced L-parameters is worked out in Section \ref{sec:2}. 
There we also recall the Hecke algebras on the Galois side of the LLC, from
\cite{AMS3}, and we compute their $q$-parameters. Via the LLC for tori we associate to
$\Rep (G)^{\mf s}$ a unique Bernstein component $\Phi_e (G)^{\mf s^\vee}$ of $\Phi_e (G,T)$.
That yields an extended affine Hecke algebra $\mc H (\mf s^\vee, q_F^{1/2})$. The crucial
step to pass from the $p$-adic side to the Galois side of the LLC is:

\begin{thmintro}\label{thm:C} \textup{[see Theorem \ref{thm:4.3}]}\\
There exists a canonical algebra isomorphism 
$\mc H (\mf s)^{op} \cong \mc H (\mf s^\vee,q_F^{1/2})$.
\end{thmintro}

The above steps make $\Rep (G)^{\mf s}$ canonically equivalent to the module category of
$\mc H (\mf s^\vee,q_F^{1/2})$. In \cite{AMS3}, $\Irr \big( \mc H (\mf s^\vee,q_F^{1/2}) \big)$ 
is parametrized by $\Phi_e (G)^{\mf s^\vee}$. We want to use that, but it does not quite suffice
because we also need to keep track of genericity of representations. Therefore we revisit
several constructions from \cite{AMS3}, in our setting of the principal series. The main
point of Section \ref{sec:red} is to show that through all those steps the one-dimensional
representation $\det$ of $(\mc H (W (R_{\mf s}^\vee), q_F^\lambda) \rtimes \Gamma_{\mf s})^{op}$
is transformed into an analogous representation $\det$ for an extended graded Hecke algebra.
That enables us to normalize the parametrization of $\Irr \big( \mc H (\mf s^\vee,q_F^{1/2}) \big)$, 
so that it matches generic representations with the desired kind of enhanced L-parameters.

With that settled the preparations are complete, and the bijection in Theorem \ref{thm:A}
is obtained as
\[
\Irr (G)^{\mf s} \; \leftrightarrow \; \Irr \big( \End_G (\Pi_{\mf s})^{op} \big) \;
\leftrightarrow \; \Irr \big( \mc H (\mf s)^{op} \big) \; \leftrightarrow\;
\Irr \big( \mc H (\mf s^\vee,q_F^{1/2}) \big) \; \leftrightarrow \; \Phi_e (G)^{\mf s^\vee}.
\]
The properties of the bijection $\Irr (G,T) \leftrightarrow \Phi_e (G,T)$, actually a few more
than mentioned already, are checked in the remainder of Section \ref{sec:5}.

Several further research topics are suggested by the above theorems.
\begin{itemize}
\item Like in \cite[\S 17]{ABPSprin}, one would like to show that the LLC is functorial with
respect to those homomorphisms of reductive $p$-adic groups that have commutative kernel and
commutative cokernel. That should be doable with the methods from \cite{SolFunct}. In particular
that can be applied to automorphisms of $G$ from conjugation with elements of 
$\mc G_\ad (F)$, then it will show how the LLC changes if one modifies the Whittaker datum.
\item Suppose that $\phi$ is discrete and $Z(G)$ is compact. It is conjectured in \cite{HII} 
that the formal degree of the square-integrable representation $\pi (\phi,\rho)$ equals 
$\dim (\rho)$ times the adjoint $\gamma$-factor of $\phi$ (with suitable normalizations on 
both sides). While this adjoint $\gamma$-factor can be computed as in
\cite[Appendix A]{FOS2}, it may be difficult to determine this formal degree. The reason is 
that one would like to use a type, but sometimes it is 
not known whether a type for the involved Bernstein block exists.
\item Every L-packet conjecturally supports a stable distribution on $G$. For L-packets that
are entirely contained in $\Irr (G,T)$, one could try to prove that the distribution
$\sum_{\rho \in \Irr (R_\phi)} \dim (\rho) \, \mr{tr} \, \pi (\phi,\rho)$ is stable.
\item A modern geometric approach to the Langlands correspondence \cite{FaSc,Hel,Zhu}  
predicts that the derived category of $\Rep (G)$ embeds in a derived category of coherent 
sheaves on a stack of Langlands parameters. It would be interesting to transfer the obtained 
natural equivalence 
\[
\Rep (G)^{\mf s} \cong \Mod \big( \mc H (\mf s^\vee,q_F^{1/2}) \big)
\] 
to a setting with such coherent sheaves, that would establish a part of the conjectures in 
\cite{FaSc,Hel,Zhu}. It is reasonable to expect that can be done, because 
$\mc H (\mf s^\vee,q_F^{1/2})$ is constructed from $\Phi_e (G)^{\mf s^\vee}$ and because on
the underlying cuspidal level the local Langlands correspondence for tori achieves it already.\\
\end{itemize}

\renewcommand{\theequation}{\arabic{section}.\arabic{equation}}
\counterwithin*{equation}{section}

\section{Hecke algebras for principal series representations}
\label{sec:1}

Let $F$ be a non-archimedean local field with ring of integers $\mf o_F$. Let $q_F$ be the
cardinality of the residue field. Let $| \cdot |_F : F \to \R_{\geq 0}$ be the norm and fix an 
element $\varpi_F \in \mf o_F$ with norm $q_F^{-1}$. We also fix a separable closure $F_s$ of $F$ 
and we let $\mb W_F \subset \mr{Gal}(F_s/F)$ be the Weil group. Field extensions of $F$ will by
default be contained in $F_s$.

Let $\mc G$ be a quasi-split reductive $F$-group, where we include con\-nec\-tedness in the 
definition of quasi-split. Let $\mc S$ be a maximal $F$-split torus in $\mc G$. 
We write $G = \mc G (F), S = \mc S (F)$ etcetera.
Since $\mc G$ is $F$-quasi-split, the centralizer $\mc T$ of $\mc S$ in $\mc G$ is a maximal
$F$-torus. It is also a minimal $F$-Levi subgroup, a Levi factor of a Borel subgroup 
$\mc B$ of $\mc G$. The Weyl group of $(\mc G,\mc S)$ and $(G,S)$ is 
\[
W(\mc G,\mc S) = N_{\mc G}(\mc S) / Z_{\mc G}(\mc S) = N_{\mc G}(S) / \mc T \cong
N_G (S) / T = N_G (T) / T .
\]
This is also the Weyl group of the root system $R(\mc G ,\mc S)$.

Let $T_\cpt$ be the unique maximal compact subgroup of $T$ and let $X_\nr (T)$ be the group of
unramified characters of $T$, that is, the characters that are trivial on $T_\cpt$.
Pick any smooth character $\chi_0 : T \to \C^\times$ and write $\chi_c = \chi_0 |_{T_\cpt}$.
Then $\chi_c$ determines $X_\nr (T) \chi_0$ and conversely. 

We denote the category of smooth $G$-representations on complex vector spaces by
$\Rep (G)$, and the set of equivalence classes of irreducible objects therein by $\Irr (G)$.
The set $X_\nr (T) \chi_0 = \Irr (T)^{\mf s_T}$ (with $\mf s_T = [T,\chi_0]_T$) is also known as 
a Bernstein component of $\Irr (T)$. From $\Irr (T)^{\mf s_T}$ one derives a Bernstein 
component $\Irr (G)^{\mf s}$, where $\mf s = [T,\chi_0 ]_G$. It consists of the irreducible 
subquotients of the normalized parabolic inductions 
\[
I_B^G (\chi) = \mr{ind}_B^G \big( \chi \otimes \delta_B^{1/2} \big) 
\text{ with } \chi \in X_\nr (T) \chi_0 .
\] 
We recall that the Bernstein block $\Rep (G)^{\mf s}$ is the full subcategory of $\Rep (G)$ made 
up by the representations $\pi$ such that every irreducible subquotient of $\pi$ belongs to
$\Irr (G)^{\mf s}$.
The standard way to classify $\Irr (G)^{\mf s}$ is by describing $\Rep (G)^{\mf s}$ as the 
module category of a Hecke algebra, and then using the representation theory of Hecke algebras.
We do so with the method that provides maximal generality, from \cite{Hei1,SolEnd}. 

We denote smooth induction with compact supports by ind. The
$T$-repre\-sentation $\mr{ind}_{T_\cpt}^T (\chi_c) \cong \chi_c \otimes \C [T / T_\cpt]$ is a
progenerator of $\Rep (T)^{\mf s_T}$. By the first and second adjointness theorems,
\[
\Pi_{\mf s} = I_B^G \big( \mr{ind}_{T_\cpt}^T (\chi_c) \big)
\]
is a progenerator of $\Rep (G)^{\mf s}$. Let $\End_G (\Pi_{\mf s})$ be the algebra of 
$G$-endomorphisms of $\Pi_{\mf s}$, acting from the left on $\Pi_{\mf s}$. Then the functors
\begin{equation}\label{eq:1.6}
\begin{array}{ccc}
\Rep (G)^{\mf s} & \longleftrightarrow & \End_G (\Pi_{\mf s})-\Mod \\
\rho & \mapsto & \Hom_G (\Pi_{\mf s},\rho) \\
V \otimes_{\End_G (\Pi_{\mf s})} \Pi_{\mf s} & \text{\reflectbox{$\mapsto$}} & V
\end{array}
\end{equation}
are equivalences of categories \cite[Theorem I.8.2.1]{Roc2}. This is compatible with parabolic
induction, in the following sense. Let $P = M R_u (P)$ be a parabolic subgroup of $G$, where
$B \subset P$, $M$ is a Levi factor of $P$ and $T \subset M$. The diagram
\begin{equation}\label{eq:1.7}
\begin{array}{ccc}
\Rep (G)^{\mf s} & \to & \End_G (\Pi_{\mf s})-\Mod \\
\uparrow I_P^G & & \uparrow \mr{ind}_{\End_M (\Pi_{\mf s_M})}^{\End_G (\Pi_{\mf s})} \\
\Rep (M)^{\mf s_M} & \to & \End_M (\Pi_{\mf s_M}) -\Mod
\end{array}
\end{equation}
commutes, see \cite[Condition 4.1 and Lemma 5.1]{SolComp}.

The algebra $\End_G (\Pi_{\mf s})$ was investigated in \cite{SolEnd}, in larger generality. We will 
make it more explicit in the current setting. Since $\dim \chi_0 = 1$, $\chi |_{T_\cpt}$ is irreducible 
and we may use \cite[\S 10]{SolEnd} with $E = E_1 = \C$ and $\sigma_1 = \sigma |_{T_\cpt} = \chi_c$. 
For comparison with \cite{SolEnd} we also note that the group 
\[
X_\nr (T,\chi_0) = \{ \chi \in X_\nr (T) : \chi \otimes \chi_0 \cong \chi_0 \}
\]
is trivial. We write
\[
W_{\mf s} = \mr{Stab}_{W(\mc G,\mc S)}(\mf s_T) = \mr{Stab}_{W(\mc G,\mc S)}(X_\nr (T) \chi_0) =
\mr{Stab}_{W(\mc G,\mc S)}(\chi_c) . 
\]
This group acts naturally on the complex variety 
\[
T_{\mf s} := \chi_0 X_\nr (T).
\]
by $(w \cdot \chi)(t) = \chi (w^{-1} t w)$.
The theory of the Bernstein centre \cite{BeDe} says that 
\[
Z(\Rep (G)^{\mf s}) \cong Z \big( \End_G (\Pi_{\mf s}) \big) \cong 
\mc O (X_\nr (T) \chi_0 )^{W_{\mf s}} = \mc O (X_\nr ( T) \chi_0 / W_{\mf s} ) .
\] 
The algebra $\End_G (\Pi_{\mf s})$ contains $\mc O (X_\nr (T) \chi_0) = \mc O (T_{\mf s})$ 
as maximal commutative subalgebra, and as module over that subalgebra it is free with a basis
$\{ N_w : w \in W_{\mf s} \}$ \cite[Theorem 10.9]{SolEnd}.

We note that the inertial equivalence class $\mf s$ for $G$ can arise from different inertial
equivalence classes for $T$. Namely, the possibilities are $w \mf s_T = [T,w \chi_0]_T$ with 
$w \in W(\mc G,\mc S)$. Thus $w \mf s = \mf s$ as inertial equivalence classes for $G$, but they 
are represented by different subsets of $\Irr (T)$. For any $w \in W(\mc G,\mc S)$, the 
$G$-representations $\Pi_{\mf s}$ and $\Pi_{w \mf s} = I_B^G (\ind_{T_\cpt}^T (w \chi_c))$ are 
isomorphic, see \cite[\S VI.10.1]{Ren}. That yields an algebra isomorphism
\begin{equation}\label{eq:1.8}
\End_G (\Pi_{\mf s}) \cong \End_G (\Pi_{w \mf s}) ,
\end{equation}
unique up to inner automorphisms of $\End_G (\Pi_{\mf s})$.
In principle that suffices to compare the functors $\Hom_G (\Pi_{\mf s},?)$ and
$\Hom_G (\Pi_{w \mf s},?)$ on the level of isomorphism classes of representations. Nevertheless,
we will have to make \eqref{eq:1.8} explicit later, and we prepare for that now.

\begin{prop}\label{prop:1.2}
Let $w \in W(\mc G,\mc S)$ be of minimal length in $w W_{\mf s}$. The isomorphism
$\Pi_{\mf s} \cong \Pi_{w \mf s}$ can be chosen so that the induced algebra isomorphism 
\eqref{eq:1.8} restricts to
\[
\mc O (T_{\mf s}) \to \mc O (T_{w \mf s}) : f \mapsto f \circ w^{-1} .
\]
\end{prop}
\begin{proof}
Let $w = s_r \cdots s_2 s_1$ be a reduced expression in the Weyl group $W(\mc G,\mc S)$. Then
each simple reflection $s_j$ has minimal length in $s_j W_{s_{j-1} \cdots s_1 \mf s}$. In this
way we reduce the proposition to the case $w = s_\alpha$ for a simple root $\alpha \in W
(\mc G,\mc S)$, with $s_\alpha \mf s_T \neq \mf s_T$.

Let $G_\alpha \subset G$ be the subgroup generated by $T \cup U_\alpha \cup U_{-\alpha}$. As 
\[
\Pi_{\mf s} = I_{B G_\alpha}^G I_{B \cap G_\alpha}^{G_\alpha} \ind_{T_\cpt}^T (\chi_c)
\] 
and similarly for $\Pi_{w \mf s}$, it suffices to work with the reductive group $G_\alpha$
and its Borel subgroup $B \cap G_\alpha$. Equivalently, we may (and will) assume that
$R(\mc G,\mc S)$ has rank one. In this rank one setting, an isomorphism 
\begin{equation}\label{eq:1.9}
\Pi_{\mf s} \cong \Pi_{s_\alpha \mf s}
\end{equation} 
is exhibited in \cite[Lemme VI.10.1]{Ren}. We analyse that construction. 

By \cite[Corollaire VII.1.3]{Ren}
\begin{equation}\label{eq:1.14}
I_B^G : \Rep (T)^{\mf s_T} \to \Rep (G)^{\mf s} \quad
\text{is an equivalence of categories.}
\end{equation}
Let $J_{\overline B}^G : \Rep (G) \to \Rep (T)$ be the 
normalized Jacquet restriction functor with respect to the opposite Borel subgroup 
$\overline B$. As $\mf s_T \neq s_\alpha \mf s_T$, Bernstein's geometric lemma 
\cite[Th\'eor\`eme VI.5.1]{Ren} entails that
\begin{equation}\label{eq:1.11}
J_{\overline B}^G I_B^G \pi \cong \pi \oplus s_\alpha^{-1} \cdot \pi
\qquad \text{for all } \pi \in \Rep (T)^{\mf s_T}.
\end{equation}
Let $\mr{pr}_{\mf s_T} : \Rep (T) \to \Rep (T)^{\mf s_T}$ be the projection provided by the 
Bernstein decomposition. From \eqref{eq:1.11} we see that
\begin{equation}\label{eq:1.15}
\mr{pr}_{\mf s_T} J_{\overline B}^G : \Rep (G)^{\mf s} \to \Rep (T)^{\mf s_T}
\quad \text{is the inverse of } \eqref{eq:1.14}.
\end{equation} 
It follows (slightly varying on the proof of \cite[Lemme VI.10.1]{Ren} by using $\overline B$
instead of $B$) that \eqref{eq:1.9} is determined by the choice of a $T$-isomorphism 
\begin{equation}\label{eq:1.10} 
\mr{pr}_{\mf s_T} J_{\overline B}^G \Pi_{s_\alpha \mf s} \cong \ind_{T_\cpt}^T (\chi_c) .
\end{equation} 
Pick a representative for $s_\alpha$ in $N_G (T)$. From \eqref{eq:1.11}  with $\pi = 
\ind_{T_\cpt}^T (s_\alpha \chi_c)$ we see that evaluation at $s_\alpha$ in 
$I_B^G \ind_{T_\cpt}^T (s_\alpha \chi_c)$ provides an isomorphism of $T$-representations
\begin{equation}\label{eq:1.12}
\mr{ev}_{s_\alpha} : \mr{pr}_{\mf s} J_{\overline B}^G \Pi_{s_\alpha \mf s} \to 
s_\alpha^{-1} \cdot \ind_{T_\cpt}^T (s_\alpha \chi_c).
\end{equation}
Further we have a canonical $T$-isomorphism
\begin{equation}\label{eq:1.13}
\begin{array}{ccc}
s_\alpha^{-1} \cdot \ind_{T_\cpt}^T (s_\alpha \chi_c) & \to & \ind_{T_\cpt}^T (\chi_c) \\
f & \mapsto & [ t \mapsto f (s_\alpha t s_\alpha^{-1})]
\end{array}.
\end{equation}
The composition of \eqref{eq:1.12} and \eqref{eq:1.13} gives us \eqref{eq:1.10}. 
Applying \eqref{eq:1.14}, we obtain \eqref{eq:1.9}.

The subalgebra $\mc O (T_{\mf s})$ of $\End_G (\Pi_{\mf s})$ arises as 
$I_B^G \big( \mc O (T_{\mf s}) \big)$, where $\mc O (T_{\mf s})$ acts on 
$\ind_{T_\cpt}^T (\chi_c) \cong \mc O (T_{\mf s})$ by multiplication, see \cite{SolEnd}.
From \eqref{eq:1.13} we see that\\ $s_\alpha^{-1} \cdot \ind_{T_\cpt}^T (s_\alpha \chi_c)$ is
naturally isomorphic to the regular representation of $\mc O (T_{\mf s})$. 

Similarly $\mc O (T_{w \mf s})$ acts on $\ind_{T_\cpt}^T (s_\alpha \chi_c) \cong \mc O 
(T_{w \mf s})$ by multiplication, and it becomes a subalgebra of $\End_G (\Pi_{s_\alpha \mf s})$ 
via $I_B^G$. From \eqref{eq:1.11} we see that its action on $s_\alpha^{-1} \cdot 
\ind_{T_\cpt}^T (s_\alpha \chi_c)$, obtained via $J_{\overline B}^G I_B^G$, is
\[
s_\alpha^{-1} : \mc O (T_{s_\alpha \mf s}) \to \mc O (T_{\mf s})
\]
followed by the regular representation. In other words, the action of $f \in \mc O (T_{\mf s})$
on $s_\alpha^{-1} \cdot \ind_{T_\cpt}^T (s_\alpha \chi_c)$ via \eqref{eq:1.10} coincides
with the action of $s_\alpha (f) = f \circ s_\alpha^{-1} \in \mc O(T_{\mf s})$ on
$s_\alpha^{-1} \cdot \ind_{T_\cpt}^T (s_\alpha \chi_c)$ via \eqref{eq:1.13}. Applying the 
normalized parabolic induction functor $I_B^G$, we find that $I_B^G (f) \in \End_G (\Pi_{\mf s})$
is transformed into $I_B^G (f \circ s_\alpha^{-1}) \in \End_G (\Pi_{s_\alpha \mf s})$ 
by \eqref{eq:1.9}.
\end{proof}

We resume the analysis of $\End (\Pi_{\mf s})$ with $\mf s = [T,\chi_0]_G$.
Let $R_{\mf s,\mu}$ be the set of roots $\alpha \in R (\mc G,\mc S)$ for which Harish-Chandra's
function $\mu_\alpha$ is not constant on $X_\nr (T) \chi_0$. Then $R_{\mf s,\mu}$ is a root
system and $W(R_{\mf s,\mu})$ is a normal subgroup of $W_{\mf s}$ \cite[Proposition 1.3]{Hei1}. 
As explained in \cite[\S 3]{SolEnd}, we can modify $\chi_0$ inside $X_\nr (T) \chi_0$ so that 
$W(R_{\mf s,\mu})$ fixes $\chi_0$. Let $R_{\mf s,\mu}^+$ be the positive system determined by
the chosen Borel subgroup $\mc B$ of $\mc G$. Then 
\[
W_{\mf s} = W(R_{\mf s,\mu}) \rtimes \Gamma_{\mf s}
\]
where $\Gamma_{\mf s}$ denotes the stabilizer of $R_{\mf s,\mu}^+$ in $W_{\mf s}$.
Following \cite[\S 3]{SolEnd}, we use the lattice $T / T_\cpt
\cong X^* (X_\nr (T))$, and the dual lattice $(T / T_\cpt)^\vee \cong X_* (X_\nr (T))$. For
$\alpha \in R_{\mf s,\mu}$ let $h_\alpha^\vee$ be the unique generator of $T / T_\cpt \cap
\Q \alpha^\vee$ such that $| \alpha (h_\alpha^\vee) |_F > 1$. We put
\[
R_{\mf s}^\vee = \{ h_\alpha^\vee : \alpha \in R_{\mf s,\mu} \} \qquad \subset T / T_\cpt
\]
and we let $R_{\mf s} \subset (T / T_\cpt )^\vee$ be the dual root system. 
By \cite[Proposition 3.1]{SolEnd}
\[
\mc R_{\mf s} = \big( R_{\mf s}^\vee, T / T_\cpt, R_{\mf s}, (T / T_\cpt)^\vee \big) 
\]
is a root datum with Weyl group $W(R_{\mf s}^\vee) = W(R_{\mf s,\mu})$. Moreover $W_{\mf s}$
acts naturally on $\mc R_{\mf s}$ and $\Gamma_{\mf s}$ is the $W_{\mf s}$-stabilizer of the
basis of $\mc R_{\mf s}$ determined by~$\mc B$. 

The complex variety $T_{\mf s}$ is isomorphic to $X_\nr (T)$ via multiplication with $\chi_0$. 
Let $\mc H (\mf s)^\circ$ be the vector space $\mc O (T_{\mf s}) \otimes \C [W (R_{\mf s}^\vee)]$, 
identified with $\mc O (X_\nr (T)) \otimes \C [W(R_{\mf s}^\vee)]$
via $X_\nr (T) \to T_{\mf s}$. Given label functions $\lambda,\lambda^*$ and $q \in \C^\times$,
we build the affine Hecke algebra $\mc H (\mc R_{\mf s},\lambda,\lambda^*,q)$ (see for instance
\cite[Proposition 2.2]{AMS1} with $\mb z_j$ specialized to $q$).
Via the above isomorphism of vector spaces we make $\mc H (\mf s)^\circ$ into an algebra which is
isomorphic to $\mc H (\mc R_{\mf s},\lambda,\lambda^*,q)$. The group $\Gamma_{\mf s}$ acts on 
$\mc H (\mf s)^\circ$ by algebra isomorphisms:
\begin{equation}\label{eq:1.3}
\gamma (f \otimes w) = f \circ \gamma^{-1} \otimes \gamma w \gamma^{-1} \qquad
f \in \mc O (T_{\mf s}), w \in W(R_{\mf s}^\vee) .
\end{equation}
That gives rise to a crossed product algebra
\[
\mc H (\mf s) := \mc H (\mf s)^\circ \rtimes \Gamma_{\mf s} ,
\]
which we would like to be isomorphic with $\End_G (\Pi_{\mf s})$.

For $s_\alpha$ with $\alpha \in R_{\mf s,\mu}$ simple, and more generally for any $w \in 
W(R_{\mf s}^\vee)$, an element $N_w \in \End_G (\Pi_{\mf s})$ is constructed in
\cite[Lemma 10.8 and remarks]{SolEnd}, it is called $q_F^{-\lambda (\alpha) /2} T'_w$ over there.
It can be determined uniquely by the choice of a good maximal compact subgroup $K$ of $G$,
associated to a special vertex in apartement for $\mc T$ in the Bruhat--Tits building of $(\mc G,F)$.

For $\gamma \in \Gamma_{\mf s}$ we have to be more careful, mainly because it need not fix $\chi_0$.
(The group $W_{\mf s}$ fixes $\chi_0$ when $\mc G$ is $F$-split, but the argument in that case
does not generalize to $\mc G$ that only split over a ramified extension of $F$.) Since 
$X_\nr (T,\chi_0) = 1$, there exists a unique $\chi_\gamma \in X_\nr (T)$ such that 
$\gamma \cdot \chi_0 = \chi_0 \otimes \chi_\gamma$. Then $\chi_\gamma$ is fixed by 
$W(R_{\mf s})$ \cite[Lemma 3.5]{SolEnd}. The element $J_\gamma$ from 
\cite[Theorem 10.9]{SolEnd} comes from $A_\gamma$ in 
\cite[\S 5]{SolEnd}. From \cite[start of \S 5.1]{SolEnd} we see that $A_\gamma$ depends on 
$\chi_\gamma$ (which is unique) and on some 
\[
\rho_\gamma \in \Hom_T (\gamma \chi_0, \chi_0 \otimes \chi_\gamma) .
\] 
For the latter we have a canonical choice, namely the identity on $\C$. Apart from that 
$A_\gamma$ depends only on the choice of $K$.

\begin{thm}\label{thm:1.1}
The above intertwining operators $N_w J_\gamma \in \End_G (\Pi_{\mf s})$ give rise to an
algebra isomorphism 
\[
\End_G (\Pi_{\mf s}) \cong \mc H (\mf s)^\circ \rtimes \C [\Gamma_{\mf s},\natural_{\mf s}] ,
\]
for a 2-cocycle $\natural_{\mf s} : \Gamma_{\mf s}^2 \to \C^\times$, suitable $W_{\mf s}$-invariant 
label functions $\lambda : R_{\mf s}^\vee \to \Z_{>0}, \lambda^* : R_{\mf s}^\vee \to \Z_{\geq 0}$ 
and $q$-base $q_F^{1/2}$. 
This isomorphism is determined by the choice of a maximal compact subgroup $K$ of $G$.
\end{thm}
\begin{proof}
The isomorphism between $\mc H (\mf s)^\circ$ and the subalgebra of $\End_G (\Pi_{\mf s})$
generated by $\mc O (T_{\mf s})$ and the $N_w$ with $w \in W(R_{\mf s}^\vee)$
is given in \cite[Theorem 10.9]{SolEnd}. The operators $J_\gamma \; (\gamma \in \Gamma_{\mf s})$
in \cite[Theorem 10.9]{SolEnd} coincide with the $A_\gamma \in 
\End_G (\Pi_{\mf s})$ from \cite[\S 5.1]{SolEnd}. The multiplication rules for the $A_\gamma$
are given in \cite[Proposition 5.2.a]{SolEnd}. As $X_\nr (T,\chi_0) = 1$, we get
\[
A_{\gamma} A_{\gamma'} = \natural_{\mf s} (\gamma,\gamma') A_{\gamma \gamma'} 
\qquad \gamma,\gamma' \in \Gamma_{\mf s},
\]
for some $\natural_{\mf s} (\gamma,\gamma') \in \C^\times$. By the associativity of the 
multiplication, $\natural_{\mf s}$ is a 2-cocycle. The other parts of \cite[Proposition 5.2]{SolEnd} 
also simplify, because $\chi_\gamma$ is fixed by $W(R_{\mf s}^\vee)$. They show that 
\[
A_\gamma A_w = A_{\gamma w} \quad \text{and} \quad A_w A_\gamma = A_{w\gamma} \quad
\text{for } \gamma \in \Gamma_{\mf s}, w \in W(R_{\mf s}^\vee) .
\]
This implies
\begin{equation}\label{eq:1.2}
A_\gamma^{-1} A_w A_\gamma = A_{\gamma^{-1} w \gamma} = A_{\gamma^{-1}} A_w A_\gamma .
\end{equation}
In view of how $N_w$ is constructed from $A_w$ \cite[\S 10]{SolEnd}, the relation \eqref{eq:1.2}
entails $A_\gamma^{-1} N_w A_\gamma = N_{\gamma^{-1} w \gamma}$. That and \cite[(5.2)]{SolEnd} show 
that $\Gamma_{\mf s}$ acts on the image of $\mc H (\mf s)^\circ$ in $\End_G (\Pi_{\mf s})$ as in 
\eqref{eq:1.3}. Combining that with \cite[Theorem 10.9]{SolEnd} yields the required algebra isomorphism.
\end{proof}
 
An important part of the structure of $\End_G (\Pi_{\mf s})$ consists of the labels 
$\lambda (h_\alpha^\vee)$, $\lambda^* (h_\alpha^\vee)$ with $\alpha \in R_{\mf s,\mu}$. Here the 
eigenvalues of $N_{s_\alpha}$ are $q_F^{\lambda (h_\alpha^\vee)/2}$ and 
$-q_F^{-\lambda (h_\alpha^\vee)/2}$. When we  recall the known formulas for these labels, it will 
be convenient to consider all $\alpha \in R (\mc G, \mc S)$ such that $s_\alpha \in W_{\mf s}$.

Suppose first that $\mc G$ is $F$-split. By \cite[Proposition 4.3]{SolParam}, $\alpha \in 
R_{\mf s,\mu}$ if and only if $\chi \circ \alpha^\vee : F^\times \to \C^\times$ is unramified.
Further, by \cite[Theorem 4.4]{SolParam}
\begin{equation}\label{eq:1.1}
\lambda \big( \alpha^\vee (\varpi_F^{-1}) \big) = 
\lambda^* \big( \alpha^\vee (\varpi_F^{-1}) \big) = 1.
\end{equation}
Now we suppose that $\mc G$ quasi-split but not necessarily split. A special role is played by 
pairs of roots in type ${}^2 A_{2n}$, such that the diagram automorphism permutes the pair. 
We settle the other cases before we turn to those exceptional roots. 

Let $\alpha_{\mc T} \in R(\mc G,\mc T)$ be a preimage of $\alpha \in R (\mc G,\mc S)$ and let 
$\mb W_{F,\alpha_{\mc T}}$ be its stabilizer. The splitting field $F_\alpha = 
F_s^{\mb W_{F,\alpha_\mc T}}$ of $\alpha$ is unique up to isomorphism. Let $f (F_\alpha/F)$ be 
the residual degree of $F_\alpha / F$. 

Assuming that $\alpha$ is not exceptional, the issue can be reduced to \eqref{eq:1.1}. Indeed,
by \cite[\S 4.2]{SolParam}, $\alpha \in R_{\mf s,\mu}$ if and only if $\chi \circ \alpha^\vee : 
F_\alpha^\times \to \C^\times$ is unramified. Moreover, by \cite[Corollary 4.5]{SolParam}
\begin{equation}\label{eq:1.5}
\lambda \big( \alpha^\vee (\varpi_{F_\alpha}^{-1}) \big) = 
\lambda^* \big( \alpha^\vee (\varpi_{F_\alpha}^{-1}) \big) = f(F_\alpha / F).
\end{equation}
In most cases $h_\alpha^\vee = \alpha^\vee (\varpi_{F_\alpha}^{-1})$ in $T / T_\cpt$, and sometimes 
$\alpha^\vee (\varpi_{F_\alpha}^{-1}) = (h_\alpha^\vee)^2$ in $T / T_\cpt$. 
In the latter cases, for instance $PGL_2 (F)$,
\begin{equation}\label{eq:1.4}
\lambda (h_\alpha^\vee) = f(F_\alpha / F) \quad \text{and} \quad \lambda^* (h_\alpha^\vee) = 0 .
\end{equation}
The exceptional roots occur only when $R_{\mf s,\mu}$ has a component of type $BC_n$ which comes
from a component of type ${}^2 A_{2n}$ in $R (\mc G,\mc S)$. Consider an indivisible root $\alpha
\in R_{\mf s,\mu}$ which comes from two adjacent roots in ${}^2 A_{2n}$. As explained in
\cite[\S 4.2]{SolParam}, the computation of the parameters for this $\alpha$ can be reduced to a
quasi-split group $SU_3 (F_\alpha / F_{2\alpha})$ Moreover, since the groups of unramified characters 
of $SU_3 (F_\alpha / F_{2\alpha}),\, U_3 (F_\alpha / F_{2\alpha})$ and $PU_3 (F_\alpha / F_{2\alpha})$ 
are naturally identified, the reductions from \cite[\S 2]{SolParam} apply to these groups in the strong 
sense that in these instances of \cite[Proposition 2.4]{SolParam} no doubling or halving of roots can 
occur. Consequently the labels for $\alpha \in R (\mc G,\mc S)$ are precisely $f(F_{2\alpha}/F)$ times
the labels for $\alpha$ as root for $U_3 (F_\alpha / F_{2\alpha})$.

For  $U_3 (F_\alpha / F_{2\alpha})$ all $q$-parameters for principal series representations were computed
via types by the author's PhD student Badea \cite{Bad}. The outcome can be summarized as follows.
\begin{itemize}
\item If $F_\alpha / F_{2\alpha}$ is unramified and $\chi_c$ is trivial on $T_\cpt \cap SU_3 (F_\alpha / 
F_{2\alpha})$, then $\alpha \in R_{\mf s,\mu}$ and $\lambda (h_\alpha^\vee) = 3, \lambda^* 
(h_\alpha^\vee) = 1$.
\item If $F_\alpha / F_{2\alpha}$ is unramified and $\chi_c$ is nontrivial on $T_\cpt \cap SU_3 (F_\alpha / 
F_{2\alpha})$, then $\alpha \in R_{\mf s,\mu}$ and $\lambda (h_\alpha^\vee) = \lambda^* 
(h_\alpha^\vee) = 1$.
\item If $F_\alpha / F_{2\alpha}$ is ramified, then $\alpha \in R_{\mf s,\mu}$ if and only if
$\chi \circ \alpha^\vee : F_\alpha^\times \to \C^\times$ is nontrivial on $\mf o_{F_{2\alpha}}^\times$.
(We note that $\chi^2 \circ \alpha^\vee |_{\mf o_{F_{2\alpha}}^\times} = 1$ because $s_\alpha \chi_c = 
\chi_c$.) When this condition is fulfilled, we have $\lambda \big( \alpha^\vee (\varpi_{F_{2\alpha}}^{-1}) 
\big) = \lambda^* \big( \alpha^\vee (\varpi_{F_{2\alpha}}^{-1}) \big) = 1$ and $\lambda (h_\alpha^\vee) 
= 1, \lambda^* (h_\alpha^\vee) = 0$.
\end{itemize}
We warn that in \cite{Bad} it is assumed throughout that the residual characteristic of $F$ is not 2.
For unramified characters $\chi$ this restriction is not necessary, because in those cases the Hecke
algebras and the parameters were already known from \cite{Bor1}. However, for other $\chi$ the 
tricky calculations in \cite[\S 2.7 and \S 5.2.1]{Bad} do not work in residual characteristic 2.

For $F$ of arbitrary characteristic, the Hecke algebra parameters for $U_3 (F_\alpha / F_{2\alpha})$ can 
also be determined via the endoscopic methods from \cite{Moe}, see \cite[Theorem 4.9]{SolParam}. That 
shows that the above formulas also apply when the residual characteristic of $F$ is 2.

\section{Whittaker normalization}
\label{sec:Whit}

Unfortunately the isomorphism from Theorem \ref{thm:1.1} is not entirely canonical, because it
depends on a good maximal compact subgroup $K$ of $G$, and often $G$ has more than one conjugacy
class of such subgroups. Further, it may be expected that the 2-cocycle $\natural_{\mf s}$ of
$\Gamma_{\mf s}$ is trivial, because $G$ is quasi-split. We will fix both issues by using
a Whittaker datum. Let $U$ be the unipotent radical of $B$ (since all Borel subgroups of $G$
are conjugate, the choice of $B$ is inessential.) Let $\xi : U \to \C^\times$ be a nondegenerate
smooth character, which means that it is nontrivial on every root subgroup $U_{\alpha}$ with 
$\alpha \in R (\mc G,\mc S)$ simple. Then the $G$-conjugacy class of $(U,\xi)$ is Whittaker 
datum for $G$.

Recall that a Whittaker functional for $\pi \in \Rep (G)$ is an element of 
\[
\Hom_{U} (\pi,\xi) \cong \Hom_G \big( \pi, \mr{Ind}_{U}^G (\xi) \big) ,
\]
where Ind denotes smooth induction. We say that $\pi$ is generic, or more precisely $(U,\xi)$-generic, 
if it admits a nonzero Whittaker functional. It is well-known \cite{Rod,Sha} that every 
representation $I_B^G (\chi)$ with $\chi \in \Irr (T)$ is generic, and that its space of Whittaker
functionals has dimension one. For the upcoming arguments we need a larger but modest supply  
of generic representations.

\begin{prop}\label{prop:3.6}
Suppose that $R(\mc G,\mc S)$ and $R_{\mf s,\mu}$ have rank one. Then $|W_{\mf s}| = 2$ and
by Theorem \ref{thm:1.1} $\mc H (\mf s)$ is an affine Hecke algebra with a unique positive root 
$h_\alpha^\vee$. Let $\mr{St}_{\mc H (\mf s)}$ be the Steinberg representation of $\mc H (\mf s)$,
the unique essentially discrete series representation with an $\mc O (T_{\mf s})$-weight of the
form $\chi_0 |\alpha |_F^s$ with $s \in \R$. 
\enuma{
\item The $G$-representation $\mr{St}_{\mf s} := \mr{St}_{\mc H (\mf s)} 
\otimes_{\End_G (\Pi_{\mf s})} \Pi_{\mf s}$ is generic.
\item Suppose that $\lambda (h_\alpha^\vee) \neq \lambda^* (h_\alpha^\vee)$. In that case 
$\mc H (\mf s)$ has a unique essentially discrete series representation $\mr{St}_{\mc H (\mf s)-}$ 
with an $\mc O (T_{\mf s})$-weight of the form $\chi_0 |\alpha |_F^{ia + s}$ where $s,a \in \R$ and
$|\alpha (h_\alpha^\vee) |_F^{ia} = -1$, see \cite[\S 2.2]{SolHecke}. 
Then the $G$-representation $\mr{St}_{\mf s -} := \mr{St}_{\mc H (\mf s)-} 
\otimes_{\End_G (\Pi_{\mf s})} \Pi_{\mf s}$ is generic.
\item Suppose that the coroot of $h_\alpha^\vee$ lies in $2 (T/T_\cpt)^\vee$ and that
$\lambda (h_\alpha^\vee) = \lambda^* (h_\alpha^\vee)$. Choose $a \in \R$ as in part (b). Then 
$I_B^G (\chi_0 |\alpha |_F^{ia})$ is a direct sum of two irreducible subrepresentations. One of 
them, say $\pi^g_{\mf s -}$, is $(U,\xi)$-generic and the other, say $\pi^n_{\mf s -}$, is not. 
\item The irreducible $G$-representations in parts (a-c) are unitary.
}
\end{prop}
\begin{proof}
(a) As $R(\mc G,\mc S)$ and $R_{\mf s,\mu}$ have the same rank, the equivalence of categories 
\eqref{eq:1.6} translates ``essentially square-integrable" into ``essentially discrete series" 
\cite[Theorem 9.6.c]{SolEnd}. In particular $\mr{St}_{\mf s}$ is an essentially square-integrable
$G$-representation. The assumptions of the proposition amount to the assumptions for 
\cite[Theorem 8.1]{Shah}. Part (b) of that result provides the desired conclusion, at least
when char$(F) = 0$. The version of \cite[Theorem 8.1]{Shah} with char$(F) > 0$ was established
in \cite[Theorem 5.5]{Lom}.\\
(b) This is analogous to part (a).\\
(c) It is well-known (see for instance \cite[\S 2.2]{SolHecke}) that 
$\ind_{\mc O (T_{\mf s})}^{\mc H (\mf s)} (\chi_0 |\alpha |_F^{ia})$ 
is a direct sum of two onedimensional representations, say $\pi_{\mc H (\mf s)-}^g$ and 
$\pi_{\mc H (\mf s)-}^n$. Writing $\pi_{\mf s -}^{g/n} = \pi_{\mc H (\mf s)-}^{g/n}
\otimes_{\End_G (\Pi_{\mf s})} \Pi_{\mf s}$, we obtain
\[
I_B^G (\chi_0 |\alpha |_F^{ia}) = \pi_{\mf s -}^g \oplus \pi_{\mf s -}^n .
\]
Since $\dim \Hom_U \big( I_B^G (\chi_0 |\alpha |_F^{ia}) , \xi \big) = 1$, exactly one these direct
summands is generic (which one depends on $\xi$). By renaming if necessary, we can make $\pi_{\mf s-}^g$
generic.\\
(d) This holds because these representations are tempered and irreducible \cite[Corollaire VII.2.6]{Ren}.
\end{proof}

\subsection{Modules of Whittaker functionals} \

For our purposes it is more convenient to analyse a perspective on generic representations 
which is dual to the traditional view. For $(\pi,V) \in \Rep (G)$ let $V^\he$ be the smooth 
Hermitian dual space, that is, the vector space of all conjugate-linear maps $\ell : V \to \C$ 
which factor through the projection $V \to V^K$ for some compact open
subgroup $K$ of $G$. The Hermitian dual representation $\pi^\he$ on $V^\he$ is defined by 
\[
(\pi^\he (g) \ell) (v) = \ell (\pi (g^{-1}) v) \qquad \forall v \in V .
\]
Equivalently, $\pi^\he$ is the smooth contragredient of the complex conjugate of $\pi$. If $\pi$ is 
unitary and admissible, then $\pi^\he$ is isomorphic to $\pi$ via the $G$-invariant inner product.

\begin{lem}\label{lem:3.7}
If $\pi \in \Rep (G)^{\mf s}$, then also $\pi^\he \in \Rep (G)^{\mf s}$.
\end{lem}
\begin{proof}
Let $\mf s' = [M,\sigma ]_G$ be any inertial equivalence class different from $\mf s$. We may assume that 
$\sigma$ is unitary, so $\sigma^\he \cong \sigma$. Let $P \subset G$ be a parabolic subgroup with Levi 
factor $M$ and let $M_1 \subset M$ be the subgroup generated by all compact subgroups of $M$. Then 
$I_P^G (\ind_{M_1}^M (\sigma))$ is a progenerator of $\Rep (G)^{\mf s'}$, see
\cite[Th\'eor\`eme VI.10.1]{Ren}. With Bernstein's second adjointness we compute
\begin{multline}\label{eq:3.23}
\Hom_G \big( I_P^G (\ind_{M_1}^M (\sigma)), \pi^\he \big) \cong \Hom_M \big( \ind_{M_1}^M (\sigma),
J^G_{\overline P} (\pi^\he) \big) \cong \\
\Hom_M \big( \ind_{M_1}^M (\sigma), (J^G_P \pi)^\he \big) 
\cong \Hom_{M_1} \big( \sigma, (J^G_P \pi)^\he \big) \cong \Hom_{M_1} (J^G_P (\pi), \sigma) .
\end{multline}
Since $[M,\sigma]_G \neq \mf s$, $J^G_P (\pi)$ does not have any irreducible subquotient isomorphic
with $\sigma$ or an unramified twist of $\sigma$. Hence \eqref{eq:3.23} is zero. This means that the 
component of $\pi^\he$ in $\Rep (G)^{\mf s'}$ is zero for any $\mf s' \neq \mf s$.
\end{proof}

From \cite[(2.1.1)]{BuHe} one sees that the 
Hermitian dual of $\ind_{U}^G (\xi)$ is $\mr{Ind}_{U}^G (\xi)$, with respect to the pairing
\[
\begin{array}{ccc}
\mr{Ind}_{U}^G (\xi) \times \ind_{U}^G (\xi) & \to & \C \\
\langle f_1, f_2 \rangle & = & \int_{U \backslash G} f_1 (g) \overline{f_2 (g)} \textup{d} g
\end{array}.
\]
Hence there is a natural isomorphism 
\begin{equation}\label{eq:3.24}
\Hom_G \big( \pi, \mr{Ind}_{U}^G (\xi) \big) \cong \Hom_G (\ind_{U}^G (\xi), \pi^\he) .
\end{equation}
By Lemma \ref{lem:3.7} and \eqref{eq:1.6}, the right hand side is isomorphic with
\begin{equation}\label{eq:3.25}
\Hom_{\End_G (\Pi_{\mf s})} \big( \Hom_G (\Pi_{\mf s}, \ind_{U}^G (\xi)), 
\Hom_G (\Pi_{\mf s},\pi^\he) \big) .
\end{equation}
Thus any nonzero Whittaker functional for $\pi$ yields a nonzero element of \eqref{eq:3.25}. This 
prompts us to analyse $\Hom_G (\Pi_{\mf s}, \ind_{U}^G (\xi))$ as $\End_G (\Pi_{\mf s})^{op}$-module. 
By \cite[Theorem 2.2]{BuHe} there are canonical isomorphisms of $T$-representations
\begin{equation}\label{eq:3.21}
J_{\overline B}^G \ind_U^G (\xi) \cong \ind_{U \cap T}^T (\xi) = \ind_{\{e\}}^T (\mr{triv}) .
\end{equation}
From that we compute
\begin{multline}\label{eq:3.2}
\Hom_G \big( \Pi_{\mf s},\ind_U^G (\xi) \big) = \Hom_G \big( I_B^G (\ind_{T_\cpt}^T (\chi_c)), 
\ind_U^G (\xi) \big) \cong \\
\Hom_T (\ind_{T_\cpt}^T (\chi_c), J^G_{\overline B} \ind_U^G (\xi) \big) \cong
\Hom_T (\ind_{T_\cpt}^T (\chi_c), \ind_{\{e\}}^T (\mr{triv}) \big) .
\end{multline}
The Bernstein decomposition of $\Rep (T)$ entails that only the part of 
$\ind_{\{e\}}^T (\mr{triv})$ on which $T_\cpt$ acts according to $\chi_c$ 
contributes to the right hand side. Hence \eqref{eq:3.2} is naturally isomorphic with 
\begin{equation}\label{eq:3.3}
\Hom_T \big( \mr{ind}_{T_\cpt}^T (\chi_c), \ind_{T_\cpt}^T (\chi_c) \big) \cong
\Hom_{T_\cpt} (\chi_c, \ind_{T_\cpt}^T (\chi_c) \big) \cong \ind_{T_\cpt}^T (\chi_c) .   
\end{equation}
This vector space contains a canonical unit vector, namely $\chi_c \in \mr{ind}_{T_\cpt}^T (\chi_c)$ 
or equivalently ${\bf 1} \in \mc O (T_{\mf s} )$. 
We use the boldface to indicate that it is an element of \eqref{eq:3.3}, not of $\End_G (\Pi_{\mf s})$. 

We want to normalize our intertwining operators $N_w$ so that they act on $\bf 1$ in an easy way. 
Any $f \in \mc O (T_{\mf s}) \cong \mr{ind}_{T_\cpt}^T (\chi_c)$ can be regarded as element 
of $\End_G (\Pi_{\mf s})$, namely $I_B^G$ applied to multiplication by $f$. The action of that on 
\eqref{eq:3.3} is again multiplication by $f$. Thus \eqref{eq:3.13} is free 
of rank one as $\mc O (T_{\mf s})$-module, and $\bf 1$ forms a canonical basis. 

Let $\C (T_{\mf s})$ be the field of rational functions on $T_{\mf s}$, the quotient field of 
$\mc O (T_{\mf s})$. It follows from Bernstein's geometric lemma \cite[Th\'eor\`eme VI.5.1]{Ren} that
\begin{equation}\label{eq:3.4}
\End_G \big( I_B^G \C (T_{\mf s}) \big) \cong \End_G \big( I_B^G \mc O 
(T_{\mf s}) \big) \otimes_{\mc O (T_{\mf s})} \C (T_{\mf s}) ,
\end{equation}
see \cite[Lemma 5.3]{SolEnd}. The natural isomorphisms \eqref{eq:3.2} and \eqref{eq:3.3} extend to
\begin{equation}\label{eq:3.5}
\Hom_G \big( I_B^G \mc O (T_{\mf s}), \mr{ind}_U^G ( \xi) \big) \otimes_{\mc O (T_{\mf s})}
\C (T_{\mf s}) \cong \C (T_{\mf s}) ,
\end{equation} 
and as module over \eqref{eq:3.4} this is an extension of scalars of \eqref{eq:3.3}.
The advantage of this setup is:

\begin{prop}\label{prop:3.1}
Theorem \ref{thm:1.1} extends to an algebra isomorphism
\[
\End_G \big( I_B^G \C (T_{\mf s}) \big) \cong \big( \C (T_{\mf s}) \rtimes
W(R_{\mf s}^\vee) \big) \rtimes \C [\Gamma_{\mf s}, \natural_{\mf s}] .
\]
\end{prop}
\begin{proof}
This is a direct consequence of Theorem \ref{thm:1.1} and \S 5.1 (in particular Corollary 5.8) of  
\cite{SolEnd}.
\end{proof}

In Proposition \ref{prop:3.1} the basis elements of $\C [\Gamma_{\mf s}, \natural_{\mf s}]$ are the
same $J_\gamma = A_\gamma$ as in Theorem \ref{thm:1.1}. The basis elements of 
\[
\C [W(R_{\mf s}^\vee)] \subset \End_G \big( I_B^G \C (T_{\mf s}) \big)
\]
are the $\mc T_w$ from \cite[Proposition 5.5]{SolEnd}, which are expressed in terms of
the $N_w$ in Lemma 10.8 and the preceding remarks of \cite{SolEnd}. Proposition \ref{prop:3.1} 
enables us to analyse the actions on \eqref{eq:3.3} and on \eqref{eq:3.5} more explicitly.

For $w \in W(R_{\mf s}^\vee ), \gamma \in \Gamma_{\mf s}$ and $f \in \mc O (T_{\mf s})$:
\begin{equation}\label{eq:3.6}
\begin{aligned}
f \cdot \mc T_w J_\gamma & = {\bf 1} \cdot f \mc T_w J_\gamma =
({\bf 1} \cdot \mc T_w J_\gamma) \cdot (J_\gamma^{-1} \mc T_w^{-1} f \mc T_w J_\gamma) \\
& = ({\bf 1} \cdot \mc T_w J_\gamma ) \cdot (f \circ w \gamma) = 
(f \circ w \gamma) ({\bf 1} \cdot \mc T_w J_\gamma) 
\end{aligned}
\end{equation}
in \eqref{eq:3.5}. Notice that ${\bf 1} \cdot J_\gamma$ must be invertible in 
$\mc O (T_{\mf s})$, because $J_\gamma$ is invertible in $\End_G (\Pi_{\mf s})$.
  
We write $\theta_{n\alpha}$ for $\theta_{n h_\alpha^\vee}$, where $n \in \Z$ and $\alpha \in 
R_{\mf s}^\vee$. We also abbreviate 
\[
q_\alpha = q_F^{(\lambda (h_\alpha^\vee) + \lambda^* (h_\alpha^\vee))/2} \qquad \text{and}
\qquad q_{\alpha *} = q_F^{(\lambda (h_\alpha^\vee) - \lambda^* (h_\alpha^\vee))/2}.
\] 

\begin{prop}\label{prop:3.2}
For each simple root $h_\alpha^\vee \in R_{\mf s}^\vee$ there exists $n_\alpha \in \Z$
such that ${\bf 1} \cdot \mc T_{s_\alpha} = -\theta_{n_\alpha \alpha}$ in \eqref{eq:3.5}.
\end{prop}
\begin{proof}
The operators $N_{s_\alpha} \in \End_G (\Pi_{\mf s})$ and $\mc T_{s_\alpha}$ arise by parabolic 
induction from the analogous elements for the Levi subgroup $G_\alpha$ of $G$ generated by $T \cup
U_\alpha \cup U_{-\alpha}$. Hence it suffices to work in $G_\alpha$, which means that we may assume
that $R(\mc G,\mc S)$ and $R_{\mf s,\mu}$ have rank one. 

First we consider the cases where $q_{\alpha *} \neq 1$, or equivalently
$\lambda (h_\alpha^\vee) \neq \lambda^* (h_\alpha^\vee)$.
From \cite[(5.19)]{SolEnd} we know that $\mc T_{s_\alpha} (q_\alpha - \theta_{-\alpha}) 
(q_{\alpha *} + \theta_{-\alpha}) \in \End_G (\Pi_{\mf s})$. Hence we can write 
\[
{\bf 1} \cdot \mc T_{s_\alpha} = f_1 (q_\alpha - \theta_{-\alpha})^{-1} 
(q_{\alpha *} + \theta_{-\alpha})^{-1} \quad \text{with } f_1 \in \mc O (T_{\mf s}) .
\]
The relations $\mc T_{s_\alpha}^2 = 1$ and \eqref{eq:3.6} imply that
\[
1 = ({\bf 1} \cdot \mc T_{s_\alpha}) \: s_\alpha ({\bf 1} \cdot \mc T_{s_\alpha}) = 
\frac{f_1 \; s_\alpha (f_1)}{(q_\alpha - \theta_\alpha) (q_{\alpha *} + \theta_\alpha)
(q_\alpha - \theta_{-\alpha}) (q_{\alpha *} + \theta_{-\alpha})} .
\]
It follows that there exist $\epsilon \in \{\pm 1\}$ and $n_\alpha \in \Z$ such that
\[
f_1 = \epsilon \theta_{n_\alpha \alpha} 
(q_\alpha - \theta_{\pm \alpha})(q_{\alpha *} + \theta_{\pm' \alpha}) ,
\]
for suitable signs $\pm, \pm'$. Equivalently
\begin{equation}\label{eq:3.7}
{\bf 1} \cdot \mc T_{s_\alpha} = \epsilon \theta_{n_\alpha \alpha} 
\Big( \frac{q_\alpha - \theta_{\alpha}}{q_\alpha - \theta_{-\alpha}} \Big)^\eta 
\Big( \frac{q_{\alpha *} + \theta_{\alpha}}{q_{\alpha *} + \theta_{-\alpha}} \Big)^{\eta'} 
=: \epsilon \theta_{n_\alpha \alpha} f_2 ,
\end{equation}
where $\eta, \eta' \in \{0,1\}$.  

Under our assumption $\alpha^\sharp \in 2 (T / T_\cpt)^\vee$ and $s_\alpha$ fixes any 
$\chi \in X_\nr (T)$ with $\chi (h_\alpha^\vee) = -1$. Notice that $f_2 (\chi) = 1$ whenever 
$\theta_\alpha (\chi)\in \{\pm 1\}$. As in \cite[10.7.b]{SolEnd} define
\begin{equation}\label{eq:3.20}
\epsilon_\alpha = \left\{ 
\begin{array}{ll}
1 & \text{if } I_B^G (\mr{ev}_\chi) \mc T_{s_\alpha} = - I_B^G (\mr{ev}_\chi) , \\
0 & \text{otherwise.}
\end{array} \right.
\end{equation}
By \cite[Lemma 10.8]{SolEnd}
\begin{equation}\label{eq:3.8}
q_F^{\lambda (h_\alpha^\vee) / 2} N_{s_\alpha} + 1 =
(\mc T_{s_\alpha} \theta_{-\epsilon_\alpha \alpha}+ 1) (\theta_\alpha q_\alpha - 1)
(\theta_\alpha q_{\alpha *} + 1) (\theta_{2 \alpha} - 1)^{-1}
\end{equation}
belongs to $\End_G (\Pi_{\mf s})$. In particular 
\begin{equation}\label{eq:3.28}
{\bf 1} \cdot (q_F^{\lambda (h_\alpha^\vee) / 2} N_{s_\alpha} + 1) = 
\frac{(\epsilon \theta_{(n_\alpha - \epsilon_\alpha) \alpha} f_2 + 1) (\theta_\alpha q_\alpha - 1)
(\theta_\alpha q_{\alpha *} + 1)}{ \theta_{2 \alpha} - 1}
\end{equation}
lies in $\Hom_G \big( \Pi_{\mf s}, \mr{ind}_U^G (\xi) \big) \cong \mc O (T_{\mf s})$.
Specializing the numerator of \eqref{eq:3.28} at $\chi'$ with $\theta_\alpha (\chi') = 1$ gives
$(\epsilon + 1) (q_\alpha - 1) (q_{\alpha *} + 1)$.
Since $q_\alpha > 1$ and \eqref{eq:3.28} has no poles, this implies $\epsilon = -1$. 

Let $\mc G_\der$ be the derived group of $\mc G$ and write $\mf s_\der = [\chi |_{T \cap G_\der},
T \cap G_\der]_{G_\der}$. By construction $\mc H (\mf s_\der)$ is the subalgebra of 
$\mc H (\mf s)$ generated by $\C [T \cap G_\der / T_\cpt \cap G_\der]$ and $N_{s_\alpha}$.
From \cite[\S 2.2]{SolHecke} we recall that
$\mr{St}_{\mc H (\mf s)} : \mc H (\mf s) \to \C$ is given on $\mc H (\mf s_\der)$ by
\[
\mr{St}_{\mc H (\mf s)} (N_{s_\alpha}) = - q_F^{-\lambda (h_\alpha^\vee)/2}, \qquad
\mr{St}_{\mc H (\mf s)} (\theta_{n \alpha}) = q_\alpha^{-n} .
\]
From Proposition \ref{prop:3.6}.a,d we know that 
\[
\Hom_G (\mr{St}_{\mf s}, \mr{Ind}_U^G (\xi)) \cong \Hom_G (\ind_U^G (\xi), \mr{St}_{\mf s})
\quad \text{has dimension } 1.
\]
As in \eqref{eq:3.25}, any nonzero Whittaker functional yields a surjection 
\[
\Hom_G (\Pi_{\mf s}, \ind_U^G (\xi)) \cong \mc O (T_{\mf s}) \to \mr{St}_{\mc H (\mf s)} .
\]
This is an $\mc O (T_{\mf s})$-module homomorphism, so up to rescaling it must be evaluation at
$\chi_{\mr{St}}$, the unique $\mc O (T_{\mf s})$-weight of $\mr{St}_{\mc H (\mf s)}$. Since
$\mc H (W (R_\mf s^\vee),q_F^\lambda)$ acts on $\mr{St}_{\mc H (\mf s)}$ via the sign representation,
\begin{equation}\label{eq:3.14}
\theta_x \cdot (q_F^{\lambda (h_\alpha^\vee)/2} N_{s_\alpha} + 1) \in \ker (\mr{ev}_{\chi_{\mr{St}}}) 
\quad \forall x \in T / T_\cpt.
\end{equation}
For $x = 0$ we can make that more explicit with \eqref{eq:3.8}:
\begin{multline}\label{eq:3.9}
{\bf 1} \cdot (q_F^{\lambda (h_\alpha^\vee)/2} N_{s_\alpha} + 1) = (\theta_0 - \theta_{(n_\alpha - 
\epsilon_\alpha) \alpha} f_2) (\theta_\alpha q_\alpha - 1)(\theta_\alpha q_{\alpha*} + 1)
(\theta_{2 \alpha} - 1)^{-1} \\ 
\scalebox{1.2}{$=
\frac{ (q_\alpha - \theta_{-\alpha} )^\eta (q_{\alpha *} + \theta_{-\alpha})^{\eta'} -
\theta_{(n_\alpha - \epsilon_\alpha) \alpha} (q_\alpha - \theta_\alpha)^\eta (q_{\alpha *} +
\theta_\alpha )^{\eta'}}{(q_\alpha - \theta_{-\alpha} )^\eta (q_{\alpha *} + \theta_{-\alpha})^{\eta'}}
\frac{(\theta_\alpha q_\alpha - 1)(\theta_\alpha q_{\alpha*} + 1)}{(\theta_{2 \alpha} - 1)}$}
\end{multline}
When $\eta = 1$, this reduces to 
\[
\frac{ (q_\alpha - \theta_{-\alpha} ) (q_{\alpha *} + \theta_{-\alpha})^{\eta'} -
\theta_{(n_\alpha - \epsilon_\alpha) \alpha} (q_\alpha - \theta_\alpha) (q_{\alpha *} +
\theta_\alpha )^{\eta'}}{ (q_{\alpha *} + \theta_{-\alpha})^{\eta'}}
\frac{\theta_\alpha (\theta_\alpha q_{\alpha*} + 1)}{(\theta_{2 \alpha} - 1)} .
\]
Evaluation at $\chi_{\mr{St}}$ sends this element to
\[
\frac{- q_\alpha^{\epsilon_\alpha - n_\alpha} (q_\alpha - q_\alpha^{-1}) (q_{\alpha*} + 
q_\alpha^{-1})^{\eta'}
q_\alpha^{-1} (q_\alpha^{-1} q_{\alpha*} +1)}{(q_{\alpha*} + q_\alpha) (q_\alpha^{-2} -1)} \neq 0.
\]
That contradicts \eqref{eq:3.14}, so that $\eta$ must be 0.

We recall from \cite[proof of Theorem 2.4.c]{SolHecke} that 
$\mr{St}_{\mc H (\mf s)-}$ is given on $\mc H (\mf s_\der)$ by
\[
\mr{St}_{\mc H (\mf s)-} (N_{s_\alpha}) = - q_F^{-\lambda (h_\alpha^\vee)/2}, \qquad
\mr{St}_{\mc H (\mf s)-} (\theta_{n \alpha}) = (-q_{\alpha*}^{-1})^n .
\]
By Proposition \ref{prop:3.6}.b,d, 
\[
\Hom_G (\mr{St}_{\mf s-}, \mr{Ind}_U^G (\xi)) \cong \Hom_G (\ind_U^G (\xi), \mr{St}_{\mf s-})
\quad \text{has dimension } 1.
\]
As above, this gives a surjection
\[
\Hom_G (\Pi_{\mf s}, \ind_U^G (\xi)) \cong \mc O (T_{\mf s}) \to \mr{St}_{\mc H (\mf s)-} ,
\]
which (up to rescaling) is evaluation at the $\mc O (T_{\mf s})$-weight $\chi_{\mr{St}-}$ of 
$\mr{St}_{\mc H (\mf s)-}$. Then $\ker (\mr{ev}_{\chi_{\mr{St}-}})$ contains
\begin{multline}\label{eq:3.15}
{\bf 1} \cdot (q_F^{\lambda (h_\alpha^\vee) / 2} N_{s_\alpha} + 1) =
{\bf 1} \cdot (\mc T_{s_\alpha} \theta_{-\epsilon_\alpha \alpha}+ 1) (\theta_\alpha q_\alpha - 1)
(\theta_\alpha q_{\alpha *} + 1) (\theta_{2 \alpha} - 1)^{-1} \\
= \big( \theta_0 - \theta_{(n_\alpha - \epsilon_\alpha) \alpha} (q_{\alpha *} + \theta_\alpha / 
q_{\alpha *} + \theta_{-\alpha})^{\eta'} \big) (\theta_\alpha q_\alpha - 1)
(\theta_\alpha q_{\alpha *} + 1) (\theta_{2 \alpha} - 1)^{-1} \\
= \frac{ (q_{\alpha *} + \theta_{-\alpha})^{\eta'} -
\theta_{(n_\alpha - \epsilon_\alpha) \alpha} (q_{\alpha *} +
\theta_\alpha )^{\eta'}}{(q_{\alpha *} + \theta_{-\alpha})^{\eta'}}
\frac{(\theta_\alpha q_\alpha - 1)(\theta_\alpha q_{\alpha*} + 1)}{(\theta_{2 \alpha} - 1)} .
\end{multline}
When $\eta' = 1$, \eqref{eq:3.15} simplifies to
\[
(q_{\alpha *} + \theta_{-\alpha} - \theta_{(n_\alpha - \epsilon_\alpha) \alpha} (q_{\alpha*} + 
\theta_\alpha )) \frac{(\theta_\alpha q_\alpha - 1) \theta_\alpha}{(\theta_{2 \alpha} - 1)} .
\]
Evaluation at $\chi_{\mr{St}-}$ results in
\[
\frac{(-q_{\alpha *}^{-1})^{n_\alpha - \epsilon_\alpha} (q_\alpha - q_{\alpha *}^{-1})
(-q_{\alpha*}^{-1} q_\alpha - 1) q_{\alpha*}^{-1}}{q_{\alpha*}^{-2} - 1} . 
\]
This is nonzero because $q_\alpha \geq q_{\alpha*} > 1$. But then \eqref{eq:3.15} does not lie 
in the kernel of $\mr{ev}_{\chi_{\mr{St}-}}$, a contradiction. Therefore $\eta'$ must be 0. 

Now we consider the cases with $q_{\alpha *} = 1$, or equivalently $\lambda (h_\alpha^\vee) = 
\lambda^* (h_\alpha^\vee)$. Then we can omit all factors $q_{\alpha*} + \theta_{\pm \alpha}$, and we
can replace $\theta_{2\alpha} - 1$ by $\theta_\alpha - 1$. 
The above argument with $\mr{St}_{\mc H (\mf s)}$ still applies, and shows that $\eta = 0$.
\end{proof}

For the moment we continue to work in $G_\alpha$. Assume that $\alpha^\sharp \in 2 (T / T_\cpt )^\vee$
and $\lambda (h_\alpha^\vee) = \lambda^* (h_\alpha^\vee)$. The onedimensional $\mc H (\mf s)$-representation
$\pi_{\mc H (\mf s)-}^g$ from Proposition \ref{prop:3.6}.c extends canonically to a representation of
$\mc H (\mf s) + \mc T_{s_\alpha} \mc H (\mf s)$, because $\mc T_{s_\alpha}$ does not have a pole at
$|\alpha |_F^{ia}$. In particular $\pi_{\mc H (\mf s)-}^g$ determines a character of the
order two group $\langle \mc T_{s_\alpha} \rangle$. We define 
\begin{equation}\label{eq:3.26}
\epsilon_\alpha = \left\{ 
\begin{array}{llll}
1 & \pi_{\mc H (\mf s)-}^g |_{\langle \mc T_{s_\alpha} \rangle} & = & \mr{triv},\\
0 & \pi_{\mc H (\mf s)-}^g |_{\langle \mc T_{s_\alpha} \rangle} & = & \mr{sign}.
\end{array}\right.
\end{equation}
This complements the definition of $\epsilon_\alpha$ when $\alpha^\sharp \in 2 (T / T_\cpt )^\vee$
and $\lambda (h_\alpha^\vee) \neq \lambda^* (h_\alpha^\vee)$, see \eqref{eq:3.20}. Together these 
provide a function
\begin{equation}\label{eq:3.27}
\epsilon_? : \{ h_\alpha^\vee \in R_{\mf s}^\vee \text{ simple, } 
\alpha^\sharp \in 2 (T / T_\cpt )^\vee \} \to \{0,1\} .
\end{equation}

\begin{lem}\label{lem:3.5}
\enuma{
\item The function \eqref{eq:3.27} is $\Gamma_{\mf s}$-invariant.
\item Take $\alpha$ in the domain of $\epsilon_?$ and let $n_\alpha$ be as in Proposition 
\ref{prop:3.2}. Then $n_\alpha - \epsilon_\alpha$ is even. 
}
\end{lem}
\begin{proof}
(a) For $\gamma \in \Gamma_{\mf s}$, represented in $N_G (T)$, we have $\gamma G_\alpha \gamma^{-1} = 
G_{\gamma (\alpha)}$ and\\ 
$J_\gamma \mc T_{s_\alpha} J_\gamma^{-1} = \mc T_{s_{\gamma (\alpha)}}$. Further
\[
\mr{Ad}(\gamma) I_{B \cap G_\alpha}^{G_\alpha} (\chi_0 |\alpha |_F^z ) \cong
I_{B \cap G_{\gamma (\alpha)}}^{G_{\gamma (\alpha)}} (\chi_0 |\alpha |_F^z ) 
\text{ for any } z \in \C .
\] 
When $\lambda (h_\alpha^\vee) = \lambda^* (h_\alpha^\vee)$, we apply this with $z = ia$.
We note that $I_{G_\alpha B}^G (\pi_{\mf s-}^g)$ is generic while $I_{G_\alpha B}^G (\pi_{\mf s-}^n)$ is
not. As $I_{G_\alpha B}^G (\pi_{\mf s-}^g) = I_{G_{\gamma (\alpha) B}}^G \mr{Ad}(\gamma) (\pi_{\mf s-}^g)$,
we conclude that $\pi_{\mf s-}^g$ for $G_{\gamma (\alpha)}$ is obtained from $\pi_{\mf s-}^g$ for 
$G_\alpha$ by $\mr{Ad}(\gamma)$. Hence $\epsilon_{\gamma (\alpha)} = \epsilon_\alpha$.

When $\lambda (h_\alpha^\vee) \neq \lambda^* (h_\alpha^\vee)$, the same argument works with the 
irreducible representation $I_{B \cap G_\alpha}^{G_\alpha} (\chi_0 |\alpha |_F^{ia})$. \\
(b) Suppose that $\lambda (h_\alpha^\vee) \neq \lambda^* (h_\alpha^\vee)$.
Recall from the proof of Proposition \ref{prop:3.2} that $\epsilon f_2 = -1$.
Specializing the numerator of \eqref{eq:3.28} at $\chi$ with $\theta_\alpha (\chi) = -1$ gives
\[
(- (-1)^{n_\alpha - \epsilon_\alpha} + 1)(-q_\alpha - 1)(-q_{\alpha *} + 1) =
((-1)^{n_\alpha - \epsilon_\alpha} - 1) (q_\alpha + 1) (1 - q_{\alpha*}) .
\]
Again this must be 0 by \eqref{eq:3.28}. Using $q_{\alpha *} \neq 1$ we find that $n_\alpha - 
\epsilon_\alpha$ is even.

Suppose that $\lambda (h_\alpha^\vee) = \lambda^* (h_\alpha^\vee)$ and $\pi_{\mc H (\mf s)-}^g
|_{\langle \mc T_{s_\alpha}\rangle} = \mr{triv}$. By Proposition \ref{prop:3.6}.d, any Whittaker
functional for $\pi_{\mc H (\mf s)-}^g$ gives a surjection
\[
\Hom_G (\Pi_{\mf s},\ind_U^G (\xi)) \cong \mc O (T_{\mf s}) \to \pi_{\mc H (\mf s)-}^g .
\]
As $\mc O (T_{\mf s})$-module homomorphism it is (up to scaling) evaluation at
$\chi_- := \chi_0 |\alpha |_F^{ia}$, a character such that $\theta_\alpha (\chi_-) = -1$. Then
$\ker (\mr{ev}_{\chi_-})$ contains 
\[
{\bf 1} \cdot (\mc T_{s_\alpha} - 1) = - \theta_{n_\alpha \alpha} - \theta_0 ,
\]
so $n_\alpha$ is odd. Recall that $\epsilon_\alpha = 1$ in this case.

Suppose that $\lambda (h_\alpha^\vee) = \lambda^* (h_\alpha^\vee)$ and $\pi_{\mc H (\mf s)-}^g
|_{\langle \mc T_{s_\alpha}\rangle} = \mr{sign}$. Then $\ker (\mr{ev}_{\chi_-})$ contains
\[
{\bf 1} \cdot (\mc T_{s_\alpha} + 1) = -\theta_{n_\alpha \alpha} + \theta_0 ,
\] 
so $n_\alpha$ is even. Here $\epsilon_\alpha = 0$, so again $n_\alpha - \epsilon_\alpha$ is even.
\end{proof}

\subsection{Normalization of intertwining operators} \

With Lemma \ref{lem:3.5}.a, we can extend $\epsilon_?$ to a $W_{\mf s}$-invariant function on 
$\{ h_\alpha^\vee \in R_{\mf s}^\vee : \alpha^\sharp \in 2 (T / T_\cpt )^\vee \}$.
In \cite{SolEnd}, $\epsilon_\alpha$ was only defined when $\lambda (h_\alpha^\vee) \neq \lambda^*
(h_\alpha^\vee)$, implicitly saying that it is 0 otherwise. We can just as well use $\epsilon_\alpha$
for any simple $h_\alpha^\vee$ with $\alpha^\sharp \in 2 (T / T_\cpt)^\vee$, Lemma \ref{lem:3.5}.a
ensures that all the computations from \cite{SolEnd} remain valid. In particular we can now (re)define
$N_{s_\alpha} \in \End_G (\Pi_{\mf s})$ by
\begin{equation}\label{eq:3.29}
q_F^{\lambda (h_\alpha^\vee) / 2} N_{s_\alpha} + 1 = (\mc T_{s_\alpha} 
\theta_{-\epsilon_\alpha \alpha}+ 1)\
(\theta_\alpha q_\alpha - 1) (\theta_\alpha q_{\alpha *} + 1) (\theta_{2 \alpha} - 1)^{-1} 
\end{equation}
for any simple $h_\alpha^\vee$ with $\alpha^\sharp \in 2 (T / T_\cpt)^\vee$. The analogous formula when
$\alpha^\sharp \not\in 2 (T / T_\cpt)^\vee$ is slightly simpler:
\begin{equation}\label{eq:3.30}
q_F^{\lambda (h_\alpha^\vee) / 2} N_{s_\alpha} + 1 = (\mc T_{s_\alpha} + 1)\
(\theta_\alpha q_\alpha - 1) (\theta_{\alpha} - 1)^{-1} .
\end{equation}

Recall that the isomorphism in Theorem \ref{thm:1.1} was determined by the choice of a good maximal 
compact subgroup $K$ of $G$, associated to a special vertex in apartement for $\mc T$ in the 
Bruhat--Tits building of $(\mc G,F)$.

\begin{lem}\label{lem:3.3}
The good maximal compact subgroup $K$ can be replaced by a $G$-conjugate, such that the isomorphism 
in Theorem \ref{thm:1.1} satisfies, for all simple roots $h_\alpha^\vee \in R_{\mf s}^\vee$:
\[
{\bf 1} \cdot \mc T_{s_\alpha} \theta_{-\epsilon_\alpha \alpha} = 
- {\bf 1} \quad \text{and} \quad 
{\bf 1} \cdot N_{s_\alpha} = -q_F^{-\lambda( h_\alpha^\vee )/2} {\bf 1} .
\]
\end{lem}
\begin{proof}
Recall the integers $n_\alpha$ from Proposition \ref{prop:3.2}. We will tacitly put $\epsilon_\alpha = 0$ 
when $\alpha^\sharp \not\in 2 (T / T_\cpt)^\vee$. Select $y$ in
$\Hom_\Z (\Z R_{\mf s}, \Z)$ so that $\langle y, \alpha^\sharp \rangle = n_\alpha - \epsilon_\alpha$ 
for every simple root $\alpha^\sharp \in R_{\mf s}$.
By Lemma \ref{lem:3.5}.b $y$ can be extended to an element of $\Hom_\Z ((T/T_\cpt)^\vee ,\Z) =
T/T_\cpt$, which we still denote by $y$. The automorphism $\mr{Ad}(\theta_y)$ of $\mc H (\mf s)^\circ$
extends uniquely to an automorphism of 
$\mc H (\mf s)^\circ \otimes_{\mc O (T_{\mf s})} \C (T_{\mf s})$, which satisfies
\begin{equation}\label{eq:3.11}
\mr{Ad}(\theta_{y}) (\mc T_{s_\alpha} \theta_{-\epsilon_\alpha}) = 
\mc T_{s_\alpha} \theta_{s_\alpha (y) - y} \theta_{-\epsilon_\alpha} = 
\mc T_{s_\alpha} \theta_{(- n_\alpha) \alpha} .
\end{equation}
By Proposition \ref{prop:3.2}
\begin{equation}\label{eq:3.12}
{\bf 1} \cdot \mr{Ad}(\theta_y) (\mc T_{s_\alpha} \theta_{-\epsilon_\alpha \alpha}) = - {\bf 1}.
\end{equation}
For any representative $y_G$ of $y$ in $T$, $K' = \mr{Ad}(y_G) K$ is another good maximal 
compact subgroup of $G$. If we replace $K$ by $K'$, then we must replace the representatives 
$\tilde w \in K$ for $w \in W (\mc G,\mc S)$,
which are used in the constructions behind Theorem \ref{thm:1.1}, by representatives in $K'$.
Which choice in $K'$ does not matter, we take
\[
\mr{Ad}(y_G^{-1}) \tilde w = \tilde w w^{-1}(y_G^{-1}) y_G \in K' .
\]
For a simple root, that means 
\[
\mr{Ad}(y_G) \tilde{s_\alpha} T_\cpt = \tilde{s_\alpha} \; 
(h_\alpha^\vee )^{\langle y,\alpha^\sharp \rangle} \in N_G (T) / T_\cpt.
\]
According to \cite[Proposition 3.1]{Hei1}, the effect of this replacement on $\mc T_{s_\alpha}$ is
left composition with $s_\alpha (\theta_\alpha^{-\langle y,\alpha^\sharp \rangle}) =
\theta_\alpha^{\langle y,\alpha^\sharp \rangle}$ or equivalently right multiplication with 
$\theta_\alpha^{-\langle y,\alpha^\sharp \rangle}$. In view of \eqref{eq:3.11}, the effect of 
Ad$(y_G^{-1})$ on $\mc H (\mf s)^\circ$ is precisely $\mr{Ad}(\theta_{y})$.

By \eqref{eq:3.12}, the new element $\mc T'_{s_\alpha} = \mc T_{s_\alpha} 
\theta_{\langle y,\alpha^\sharp \rangle \alpha}$ has the same $\epsilon_\alpha$ as before, and 
$n_\alpha$ has become $\epsilon_\alpha$. Now \eqref{eq:3.12} says that
${\bf 1} \cdot \mc T'_{s_\alpha} \theta_{-\epsilon_\alpha \alpha} = -{\bf 1}$.
The equations \eqref{eq:3.29} and \eqref{eq:3.30} for the new elements $N'_{s_\alpha}$ become 
\[
q_F^{\lambda (h_\alpha^\vee) / 2} N'_{s_\alpha} + 1 =
(\mc T'_{s_\alpha} \theta_{-\epsilon_\alpha \alpha} + 1) (\theta_\alpha q_\alpha - 1)
(\theta_\alpha q_{\alpha *} + 1) (\theta_{2 \alpha} - 1)^{-1}
\]
In view of \eqref{eq:3.12}, this implies
\[
{\bf 1} \cdot (q_F^{\lambda (h_\alpha^\vee) /2} N'_{s_\alpha} + 1 ) = 0 
\in \C (T_{\mf s}). 
\]
Equivalently, we obtain ${\bf 1} \cdot N_{s_\alpha} = -q_F^{-\lambda( h_\alpha^\vee )/2} {\bf 1}$.
\end{proof}

From now on we choose $K$ as in the statement of Lemma \ref{lem:3.3}. For $w \in W_{\mf s}$ let 
det$(w)$ be the determinant of the action of $w$ on $(T / T_\cpt) \otimes_\Z \R$. Then $\det :
W_{\mf s} \to \R^\times$ is a quadratic character extending the sign character of $W(R_{\mf s}^\vee)$.

For $\gamma \in \Gamma_{\mf s}$ we write ${\bf 1} \cdot J_\gamma = z_\gamma \theta_{x_\gamma}$ with 
$z_\gamma \in \C^\times$ and $x_\gamma \in T / T_\cpt$. Consider the operators
\[
N_\gamma = \det (\gamma) z_\gamma^{-1} \theta_{-x_\gamma} J_\gamma \in \End_G (\Pi_{\mf s}) .
\]
From \eqref{eq:3.6} we see that
\begin{equation}\label{eq:3.13}
{\bf 1} \cdot N_\gamma = \det (\gamma) {\bf 1} \qquad \gamma \in \Gamma_{\mf s} .
\end{equation}

\begin{thm}\label{thm:3.4}
The operators $N_w N_\gamma$, for $w \in W(R_{\mf s}^\vee)$ with $K$ as in Lemma \ref{lem:3.3} and 
$N_\gamma$ with $\gamma \in \Gamma_{\mf s}$ as above, provide an algebra isomorphism
\[
\End_G (\Pi_{\mf s}) \cong \mc H (\mf s)^\circ \rtimes \Gamma_{\mf s} = \mc H (\mf s) .
\]
Given the Whittaker datum $(U,\xi)$, this isomorphism is canonical.
\end{thm}
\begin{proof}
By direct computation, using Lemma \ref{lem:3.3}:
\begin{equation}\label{eq:3.18}
{\bf 1} \cdot \mc T_{s_\alpha} \theta_{-\epsilon_\alpha \alpha} = 
- \det (\gamma) z_\gamma \theta_{s_\alpha (x_\gamma)} .
\end{equation}
A similar computation shows that
\begin{equation}\label{eq:3.19}
{\bf 1} \cdot \mc T_{s_{\gamma (\alpha)}} \theta_{\epsilon_\alpha \gamma 
(\alpha)} J_\gamma = -{\bf 1} \cdot J_\gamma = - \det (\gamma) z_\gamma \theta_{x_\gamma} .
\end{equation}
As $J_\gamma \mc T_{s_\alpha} \theta_{-\epsilon_\alpha \alpha}$ equals 
$\mc T_{s_{\gamma (\alpha)}} \theta_{\epsilon_\alpha \gamma (\alpha)} J_\gamma$, \eqref{eq:3.18} 
and \eqref{eq:3.19} are equal, and we deduce that $s_\alpha (x_\gamma) = x_\gamma$. 
Hence $x_\gamma$ is fixed by each such $s_\alpha$, and by the entire group $W(R_{\mf s}^\vee)$. 
Now we can easily check that the $N_\gamma$ satisfy the desired relations: 
\begin{align*}
& {\bf 1} \cdot N_{\gamma \tilde \gamma} = \det (\gamma \tilde \gamma) {\bf 1} = 
\det(\gamma) \det (\tilde \gamma) {\bf 1} = {\bf 1} \cdot N_\gamma N_{\tilde \gamma} ,\\
& N_\gamma f N_\gamma^{-1} = J_\gamma f J_\gamma^{-1} = f \circ \gamma^{-1} \qquad 
f \in \mc O (T_{\mf s}) ,\\
& N_\gamma \mc T_{s_\alpha} N_\gamma^{-1} = z_\gamma^{-1} \theta_{-x_\gamma} J_\gamma  
\mc T_{s_\alpha} J_\gamma^{-1} \theta_{x_\gamma} z_\gamma = 
\theta_{-x_\gamma} \mc T_{s_{\gamma \alpha}} \theta_{x_\gamma} = \mc T_{s_{\gamma \alpha}} .
\end{align*}
The first two of these relations imply that
\[
N_{\gamma \tilde \gamma} = N_\gamma N_{\tilde \gamma} \quad \text{for all } \gamma, 
\tilde \gamma \in \Gamma_{\mf s}.
\]
We deduce that, with respect to the given $\mc O (T_{\mf s})$-basis, $\End_G (\Pi_{\mf s})$
becomes $\mc H (\mf s)^\circ \rtimes \Gamma$.

Any two isomorphisms of this kind differ by an automorphism $\psi$ of 
$\mc H (\mf s)^\circ \rtimes \Gamma_{\mf s}$. Since the subalgebra $\mc O (T_{\mf s})$ is
mapped naturally to $\End_G (\Pi_{\mf s})$, $\psi$ is the identity on that subalgebra. 
Hence $\psi$ extends to an automorphism of the version of $\mc H (\mf s)^\circ \rtimes \Gamma_{\mf s}$
with $\C (T_{\mf s})$. Then \eqref{eq:3.6} entails that $\psi$ multiplies each basis element 
$\mc T_w N_\gamma$ by an element of $\mc O (T_{\mf s})$. Combining that with 
Lemma \ref{lem:3.3} and \eqref{eq:3.13}, we find that $\psi$ is the identity.
\end{proof}

Theorem \ref{thm:3.4} shows in particular that the 2-cocycle $\natural_{\mf s}$
from Theorem \ref{thm:1.1} becomes trivial in $H^2 (\Gamma_{\mf s},\C^\times)$.

Recall from page \pageref{eq:1.8} that $\mf s$ can also arise from $w \mf s_T = [T,w \chi_0]_T$
for any $w \in W(\mc G,\mc S)$. To compare all these cases, it suffices to consider one $w$
from every left coset of $W_{\mf s} = \mr{Stab}_{W(\mc G,\mc S)}(\mf s_T)$.

\begin{prop}\label{prop:3.9}
Let $w \in W(\mc G,\mc S)$ be of minimal length in $w W_{\mf s}$. 
\enuma{
\item The isomorphism $\Pi_{\mf s} \cong \Pi_{w \mf s}$ from \cite[\S VI.10.1]{Ren} can be
normalized so that it sends ${\bf 1} \in \Hom_G (\Pi_{\mf s}, \ind_U^G (\xi))$ to
${\bf 1} \in \Hom_G (\Pi_{w \mf s}, \ind_U^G (\xi))$.
\item In that situation the induced algebra isomorphism $\End_G (\Pi_{\mf s}) \cong
\End_G (\Pi_{w \mf s})$ is given by $f \mapsto f \circ w^{-1}$ for $f \in \mc O (T_{\mf s})$ and
$N_v \mapsto N_{w v w^{-1}}$ for $v \in W_{\mf s}$.
}
\end{prop}
\begin{proof}
(a) The isomorphism of $G$-representations
\[
\phi_w : \Pi_{w \mf s} \isom \Pi_{\mf s}
\]
from Proposition \ref{prop:1.2} induces a map 
\begin{equation}\label{eq:3.31}
\mc O (T_{\mf s}) \cong \Hom_G (\Pi_{\mf s}, \ind_U^G (\xi)) \longrightarrow
\Hom_G (\Pi_{w \mf s}, \ind_U^G (\xi)) \cong \mc O (T_{w \mf s})
\end{equation}
and a compatible algebra isomorphism 
\[
\mr{Ad} (\phi_w^{-1} ) : \End_G (\Pi_{\mf s}) \to \End_G (\Pi_{w \mf s}).
\]
In view of \eqref{eq:3.21}--\eqref{eq:3.3} and Proposition \ref{prop:1.2}, \eqref{eq:3.31} must be
$f \mapsto f \circ w^{-1}$ followed by multiplication with some element of $\mc O (T_{w \mf s})$.

Like in the proof of Proposition \ref{prop:1.2}, we reduce to the case where $R(\mc G,\mc S)$ has
rank one, $w = s_\alpha$ is a simple reflection and $s_\alpha \mf s_T \neq \mf s_T$. We represent
$s_\alpha$ in the maximal compact subgroup $K$ from Lemma \ref{lem:3.3}.
Consider $\chi_c \in \ind_{T_\cpt}^T (\chi_c)$ and
\[
{\bf 1} \in \Hom_G (\Pi_{\mf s}, \ind_U^G (\xi)) \cong
\Hom_T \big( \ind_{T_\cpt}^T (\chi_c),\ind_{\{e\}}^T (\mr{triv}) \big) .
\]
Recall from \eqref{eq:3.2} that here the isomorphism is given by $J_{\overline B}^G$ and the 
natural transformation $\mr{id} \to J^G_{\overline B} I_B^G$. By definition 
$J_{\overline B}^G ({\bf 1}) \chi_c = \chi_c$. We want to determine 
\begin{equation}\label{eq:3.32}
J_{\overline B}^G ({\bf 1}) J_{\overline B}^G (\phi_{s_\alpha}) (s_\alpha \chi_c) \in 
J_{\overline B}^G \ind_U^G (\xi) ,
\end{equation}
where $s_\alpha \chi$ is considered as element of $\ind_{T_\cpt}^T (s_\alpha \chi_c) \subset
J_{\overline B}^G \Pi_{s_\alpha \mf s}$. From \eqref{eq:1.11} we see that 
$J_{\overline B}^G (\phi_{s_\alpha})$ on $\ind_{T_\cpt}^T (s_\alpha \chi_c)$ equals
\[
s_\alpha \circ \big[ J_{\overline B}^G (\phi_{s_\alpha}) \text{ on } s_\alpha^{-1} \cdot 
\ind_{T_\cpt}^T (s_\alpha \chi_c) \big] \circ s_\alpha^{-1} .
\]
Similarly $J_{\overline B}^G ({\bf 1})$ on $s_\alpha \cdot \ind_{T_\cpt}^T (\chi_c)$ equals
\[
s_\alpha \circ \big[ J_{\overline B}^G ({\bf 1}) \text{ on }
\ind_{T_\cpt}^T (\chi_c) \big] \circ s_\alpha^{-1} .
\]
Now we can compute \eqref{eq:3.32}. First $s_\alpha \chi$ is mapped to $\chi_c$ by 
$s_\alpha^{-1}$, then \eqref{eq:1.10}--\eqref{eq:1.13} show that $J_{\overline B}^G 
(\phi_{s_\alpha})$ sends that to $\chi_c \in \ind_{T_\cpt}^T (\chi_c)$. Applying 
$J_{\overline B}^G ({\bf 1})$ returns $\chi_c \in \ind_{\{e\}}^T (\mr{triv})$ and finally the 
action of $s_\alpha$ yields $s_\alpha \chi_c$. Hence 
\[
J_{\overline B}^G ({\bf 1}) J_{\overline B}^G (\phi_{s_\alpha}) (s_\alpha \chi_c) = 
s_\alpha \chi_c = J_{\overline B}^G ({\bf 1}) (s_\alpha \chi_c) ,
\]
where the second ${\bf 1}$ comes from $\Pi_{s_\alpha \mf s}$. In view of \eqref{eq:3.2} and
\eqref{eq:3.3}, that implies 
\[
J_{\overline B}^G ({\bf 1}) J_{\overline B}^G (\phi_{s_\alpha}) = J_{\overline B}^G ({\bf 1}) 
\quad \text{and} \quad {\bf 1} \circ \phi_{s_\alpha} = {\bf 1} . 
\]
(b) Since $w$ has minimal length in $w W_{\mf s}$, it also has minimal length in 
$w W(R_{\mf s,\mu})$. According to \cite[Lemma 2.4.a]{AMS3}, $w(R_{\mf s,\mu}^+) \subset 
R(\mc B,\mc S)$. By the $W(\mc G,\mc S)$-equivariance of Harish-Chandra $\mu$-functions also
$w (R_{\mf s,\mu}^+ ) \subset R_{w \mf s,\mu}$, and therefore $w(R_{\mf s,\mu}^+) =
R_{w \mf s,\mu}^+$.

The proof of part (a) shows that \eqref{eq:3.31} equals $f \mapsto f \circ w^{-1}$. With the
rough description of the $\mc H (\mf s)$-action on $\Hom_G (\Pi_{\mf s}, \ind_U^G (\xi))$ from
\eqref{eq:3.6} we deduce that Ad$(\phi_w^{-1})$ sends each $N_v \in \mc H (\mf s)$ to 
$N_{w v w^{-1}}$ times an element of $\mc O (T_{\mf s^\vee})$. The more precise descriptions 
from Lemma \ref{lem:3.3} and \eqref{eq:3.13} show that in fact $\mr{Ad}(\phi_w^{-1}) N_v$
equals $N_{w v w^{-1}}$ for any $v \in W_{\mf s}$. 
\end{proof}

\section{Characterization of generic representations}
\label{sec:charact}

We want to parametrize $\Irr (G,T)$ so that the generic representation correspond to the 
expected kind of enhanced L-parameters. To that end a simple characterization of genericity in
terms of Hecke algebras will be indispensable.

We start with a complete description of $\Hom_G \big( \Pi_{\mf s}, \ind_U^G (\xi) 
\big)$ as right $\mc H (\mf s)$-module.
Let $\mc H (W(R_{\mf s}^\vee, q_F^\lambda) \subset \mc H (\mf s)^\circ$ be the finite
dimensional Iwahori--Hecke algebra spanned by the $N_w$ with $w \in W(R_{\mf s}^\vee)$.
The Steinberg representation of $\mc H (W(R_{\mf s}^\vee, q_F^\lambda))$ is defined by 
\begin{equation}\label{eq:3.1}
\mr{St}(N_{s_\alpha}) = -q_F^{\lambda (h_\alpha^\vee)/2} \quad \text{for simple} \quad 
h_\alpha^\vee \in R_{\mf s}^\vee .
\end{equation}
We extend this to a representation St of 
\[
\mc H (W_{\mf s},q_F^\lambda) := \mc H (W(R_{\mf s}^\vee), q_F^\lambda) \rtimes \Gamma_{\mf s}
\]
by $\mr{St} (N_w N_\gamma) = \mr{St}(N_w) \det (\gamma)$. Notice that this formula equally well 
defines a representation of the opposite algebra $\mc H (W_{\mf s},q_F^\lambda)^{op}$. 

Special cases of the next result were established in \cite{ChSa} (for the Iwahori-spherical  
Bernstein component of a split group) and in \cite{MiPa} (for principal representations of
split reductive $p$-adic groups, with some extra conditions).

\begin{lem}\label{lem:3.8}
There is an isomorphism of $\mc H (\mf s)^{op}$-representations
\[
\Hom_G \big( \Pi_{\mf s}, \ind_U^G (\xi) \big) \cong \ind_{\mc H (W_{\mf s},q_F^\lambda)^{op}}^{
\mc H (\mf s)^{op}} (\mr{St}) .
\]
\end{lem}
\begin{proof}
Let $w \in W(R_{\mf s}^\vee)$ and $\gamma \in \Gamma_{\mf s}$. By Lemma \ref{lem:3.3} and 
\eqref{eq:3.13}
\[
{\bf 1} \cdot N_w N_\gamma = {\bf 1} \cdot \mr{St} (N_w N_\gamma)
\in \Hom_G \big( \Pi_{\mf s}, \ind_U^G (\xi) \big) .
\]
As vector spaces
\[
\mc H (\mf s) = \mc O (T_{\mf s}) \otimes \mc H (W_{\mf s},q_F^\lambda) .
\]
Further $\Hom_G \big( \Pi_{\mf s}, \ind_U^G (\xi) \big)$ is isomorphic
to $\mc O (T_{\mf s})$ as $\mc O (T_{\mf s})$-module, with basis vector ${\bf 1}$. Hence
\[
\begin{array}{ccc}
\ind_{\mc H (W_{\mf s},q_F^\lambda)^{op}}^{\mc H (\mf s)^{op}} (\mr{St}) & \to & 
\Hom_G \big( \Pi_{\mf s}, \ind_U^G (\xi) \big) \\
h \otimes 1 & \mapsto & h \cdot {\bf 1} 
\end{array}
\]
is an isomorphism of $\mc H (\mf s)^{op}$-modules.
\end{proof}

The criterium \eqref{eq:3.24}--\eqref{eq:3.25} for genericity can be put in a more manageable 
form with Lemma \ref{lem:3.8}. For any $\pi \in \Rep (G)^{\mf s}$:
\begin{align}
\nonumber \Hom_G (\pi, \mr{Ind}_U^G (\xi)) & \cong \Hom_G (\ind_U^G (\xi), \pi^\he) \\
\nonumber & \cong \Hom_{\End_G (\Pi_{\mf s})^{op}} \big( \Hom_G (\Pi_{\mf s}, \ind_U^G (\xi)), 
\Hom_G (\Pi_{\mf s},\pi^\he) \big) \\
\label{eq:6.1} & \cong \Hom_{\mc H (\mf s)^{op}} \big( \big( \ind_{\mc H 
(W_{\mf s},q_F^\lambda)^{op}}^{\mc H (\mf s)^{op}} (\mr{St}) \big), 
\Hom_G (\Pi_{\mf s},\pi^\he) \big) \\
\nonumber & \cong \Hom_{\mc H (W_{\mf s},q_F^\lambda )^{op}} \big( \mr{St}, 
\Hom_G (\Pi_{\mf s},\pi^\he) \big) .
\end{align}

\begin{cor}\label{cor:6.1}
A representation $\pi \in \Rep (G)^{\mf s}$ is $(U,\xi)$-generic if and only if the 
$\mc H (W_{\mf s},q_F^\lambda )^{op}$-module $\Hom_G (\Pi_{\mf s},\pi^\he)$ contains St.
\end{cor}

In this corollary the effect of $\pi \mapsto \pi^\he$ on $\mc H (\mf s)^{op}$-modules is not
obvious, we analyse that in several steps.
With the *-structure and the trace from \cite[\S 3.1]{SolAHA},
\begin{equation}\label{eq:6.7}
\mc H (W_{\mf s},q_F^\lambda ) = \mc H (W(R_{\mf s}^\vee), q_F^\lambda) \rtimes \Gamma_{\mf s}
\end{equation}
is a finite dimensional Hilbert algebra, so in particular semisimple. 

Recall that a standard $G$-representation is of the form $I_P^G (\tau \otimes \chi)$, where
$P = M N$ is a parabolic subgroup of $G$, $\tau$ is an irreducible tempered $M$-representation
and $\chi : M \to \R_{>0}$ is an unramified character in positive position with respect to $P$.
By conjugating $P$ and $M$, we may assume that $T \subset M$. 

\begin{lem}\label{lem:6.3}
Suppose that $\tau \in \Rep (M)^{[T,\chi_0]_M}$. Then $I_P^G (\tau \otimes \chi) \in 
\Rep (G)^{\mf s}$ and 
\[
\Hom_G \big( \Pi_{\mf s}, I_P^G (\tau \otimes \chi)^\he \big) \cong 
\Hom_G \big( \Pi_{\mf s}, I_P^G (\tau \otimes \chi) \big) \text{ as } 
\mc H (W_{\mf s},q_F^\lambda )^{op}\text{-modules.}
\]
\end{lem}
\begin{proof}
The representation $I_P^G (\tau \otimes \chi)$ has cuspidal support $\mr{Sc}(\tau) \otimes 
\chi |_T \in [T,\chi_0]_T$, so it belongs to $\Rep (G)^{[T,\chi_0]_G}$. By \cite[IV.2.1.2]{Ren}
\[
I_P^G (\tau \otimes \chi)^\he \cong I_P^G ((\tau \otimes \chi)^\he) \cong
I_P^G (\tau^\he \otimes \chi^\he) .
\]
Since $\chi$ is real-valued and $\tau$ is unitary \cite[Corollaire VII.2.6]{Ren}, the right
hand side is isomorphic with $I_P^G (\tau \otimes \chi^{-1})$. Consider the continuous path
\[
[-1,1] \to \Rep (G)^{\mf s} : t \mapsto I_P^G (\tau \otimes \chi^t) .
\]
Via the equivalence of categories \eqref{eq:1.6}, we obtain a continuous path in\\
$\Mod (\mc H (\mf s)^{op})$. Modules of a finite dimensional semisimple algebra are stable 
under continuous deformations, so
\[
\Hom_G \big( \Pi_{\mf s}, I_P^G (\tau \otimes \chi)^\he \big) \cong
\Hom_G \big( \Pi_{\mf s}, I_P^G (\tau \otimes \chi^{-1}) \big) \cong
\Hom_G \big( \Pi_{\mf s}, I_P^G (\tau \otimes \chi) \big)
\]
as $\mc H (W_{\mf s},q_F^\lambda )^{op}$-modules.
\end{proof}

We are ready to establish a useful characterization of genericity, without Hermi\-tian duals.
The next result is formulated for finite length representations, but we believe it is also
valid without that condition. To study it for arbitrary representations in $\Rep (G)^{\mf s}$
one probably needs Hermitian duals of modules over affine Hecke algebras.

\begin{thm}\label{thm:6.4}
Suppose that $\pi \in \Rep (G)^{\mf s}$ has finite length. Then $\pi$ is $(U,\xi)$-generic
if and only if $\Hom_{\mc H (W_{\mf s},q_F^\lambda )^{op}} \big( \Hom_G (\Pi_{\mf s},\pi) , 
\mr{St} \big)$ is nonzero.
\end{thm}
\begin{proof}
Since $\pi$ has finite length, we can form it semisimplification $\pi_{ss}$. Then
$\pi_{ss}^\he$ is the semisimplification of $\pi^\he$. By \eqref{eq:6.7} the module category of 
$\mc H (W_{\mf s},q_F^\lambda )^{op}$ is semisimple. In particular 
\begin{equation}\label{eq:6.2}
\Hom_{\mc H (W_{\mf s},q_F^\lambda )^{op}} \big( \mr{St}, \Hom_G (\Pi_{\mf s},\pi^\he) \big)
\end{equation}
does not change if we replace $\pi^\he$ by $\pi_{ss}^\he$. Since we only need 
semisimplifications of modules here, we may pass to the Grothendieck group of finite length
representations in $\Rep (G)^{\mf s}$. The standard modules in $\Rep (G)^{\mf s}$ form a
$\Z$-basis of that Grothendieck group. Indeed, that is a consequence of the Langlands
classification \cite[Th\'eor\`eme VII.4.2]{Ren} and the property that the irreducible quotient of 
a standard module is the unique maximal constituent in a certain sense \cite[\S XI.2]{BoWa}. 

For each such standard module we have Lemma \ref{lem:6.3}, and hence the conclusion of
Lemma \ref{lem:6.3} extends to the entire Grothendieck group of the finite length part
of $\Rep (G)^{\mf s}$. In particular 
\[
\Hom_G (\Pi_{\mf s},\pi^\he) \cong  \Hom_G (\Pi_{\mf s},\pi_{ss}^\he) \cong
\Hom_G (\Pi_{\mf s},\pi_{ss}) \cong  \Hom_G (\Pi_{\mf s},\pi) 
\]
as $\mc H (W_{\mf s},q_F^\lambda )^{op}$-modules. Hence 
the vector space \eqref{eq:6.2} is isomorphic with 
\[
\Hom_{\mc H (W_{\mf s},q_F^\lambda )^{op}} \big( \mr{St}, \Hom_G (\Pi_{\mf s},\pi) \big) .
\]
By the semisimplicity of the involved algebra, this has the same dimension as
\begin{equation}\label{eq:6.3}
\Hom_{\mc H (W_{\mf s},q_F^\lambda )^{op}} \big( 
\Hom_G (\Pi_{\mf s},\pi), \mr{St} \big) .
\end{equation}
We conclude by applying Corollary \ref{cor:6.1} to \eqref{eq:6.2} and \eqref{eq:6.3}.
\end{proof}

From Theorem \ref{thm:6.4} it is easy to prove an analogue of the uniqueness (up to scalars) 
of Whittaker functionals \cite{Rod,Sha} in the context of Hecke algebras.
Let $M$ be a standard Levi subgroup of $G$ and write $\mf s_M = [T,\chi_0]_M$. Via parabolic
induction $\mc H (\mf s_M) \cong \End_M (\Pi_{\mf s_M})$ becomes a subalgebra of 
$\mc H (\mf s) \cong \End_G (\Pi_{\mf s})$. In fact the constructions in \cite[\S 10.2]{SolEnd}
show that $\mc H (\mf s_M)$ is a parabolic subalgebra of $\mc H (\mf s)$, in the sense of
\cite[p. 216]{SolComp}. The functor $\ind_{\mc H (\mf s_M)^{op}}^{\mc H (\mf s)^{op}}$ 
corresponds to parabolic induction from $M$ to $G$, see \cite[Proposition 1.8.5.1]{Roc2}.

\begin{lem}\label{lem:6.12}
Let $V$ be an irreducible $\mc H (\mf s_M)^{op}$-module. Then
\[
\dim \Hom_{\mc H (W_{\mf s}, q_F^\lambda)^{op}} 
\big( \ind_{\mc H (\mf s_M)^{op}}^{\mc H (\mf s)^{op}} V, \mr{St} \big) \leq 1.
\]
\end{lem}
\begin{proof}
By the Bernstein presentation of $\mc H (\mf s)^{op}$ we can simplify the module:
\[
\Res^{\mc H (\mf s)^{op}}_{\mc H (W_{\mf s}, q_F^\lambda)^{op}} 
\big( \ind_{\mc H (\mf s_M)^{op}}^{\mc H (\mf s)^{op}} V \big) = 
\ind^{\mc H (W_{\mf s}, q_F^\lambda)^{op}}_{\mc H (W_{\mf s_M}, q_F^\lambda)^{op}} 
\big( \Res^{\mc H (\mf s_M)^{op}}_{\mc H (W_{\mf s_M}, q_F^\lambda)^{op}} V \big). 
\]
With Frobenius reciprocity it follows that
\begin{equation}\label{eq:6.22}
\Hom_{\mc H (W_{\mf s}, q_F^\lambda)^{op}} 
\big( \ind_{\mc H (\mf s_M)^{op}}^{\mc H (\mf s)^{op}} V, \mr{St} \big) \cong
\Hom_{\mc H (W_{\mf s_M}, q_F^\lambda)^{op}} ( V, \mr{St} ) .
\end{equation}
This reduces the lemma to the case $M = G$, which investigate next.

As $\mc H (\mf s)$ has finite rank as module over its centre, $V$ has finite dimension. 
Hence $V$ contains an eigenvector for $\mc O (T_{\mf s})$, say with character $t$. Then
\[
0 \neq \Hom_{\mc O (T_{\mf s})} (t,V) \cong \Hom_{\mc H (\mf s)^{op}} \big( \ind_{\mc O (T_{\mf s})}^{\mc
H (\mf s)^{op}} (t), V \big), 
\]
so $V$ is a quotient of $\ind_{\mc O (T_{\mf s})}^{\mc H (\mf s)^{op}} (t)$. For multiplicities upon
restriction to the finite dimensional semisimple subalgebra $\mc H (W_{\mf s},q_F^\lambda)^{op}$, 
that means
\begin{equation}\label{eq:6.8}
\dim \Hom_{\mc H (W_{\mf s}, q_F^\lambda)^{op}} (V,\mr{St}) \leq
\dim \Hom_{\mc H (W_{\mf s}, q_F^\lambda)^{op}} \big( 
\ind_{\mc O (T_{\mf s})}^{\mc H (\mf s)^{op}} (t),\mr{St} \big) .
\end{equation}
By the presentation of $\mc H (\mf s)$, $\ind_{\mc O (T_{\mf s})}^{\mc H (\mf s)^{op}} (t) \cong
\mc H (W_{\mf s}, q_F^\lambda)^{op}$ as $\mc H (W_{\mf s}, q_F^\lambda)^{op}$-modules. 
Hence the right hand side of \eqref{eq:6.8} is 1.
\end{proof}

\section{Hecke algebras for principal series L-parameters}
\label{sec:2}

Recall that we fixed a separable closure $F_s$ of $F$. Let $\mb I_F \subset \mb W_F \subset 
\mr{Gal}(F_s/F)$ be the inertia subgroup of the Weil group and pick a geometric Frobenius element 
$\Fr_F$ of $\mb W_F$. Let $G^\vee$ be the complex dual group of $G$ and let 
${}^L G = G^\vee \rtimes \mb W_F$ be the Langlands dual group. Let $\Phi (G)$ be the set of 
L-parameters $\phi : \mb W_F \times SL_2 (\C) \to {}^L G$, considered modulo $G^\vee$-conjugacy. 

For an L-parameter $\phi$ we have the component group 
\[
R_\phi = \pi_0 (Z_{G^\vee}(\phi) / Z(G^\vee)^{\mb W_F}) ,
\] 
this is the appropriate version because $G$ is quasi-split. We define a ($G$-relevant) enhancement of 
$\phi$ to be an irreducible representation of the finite group $R_\phi$. Compared to \cite{AMS1}, the 
quasi-splitness of $G$ allows us to focus on the enhancements whose $Z({G^\vee}_{\mr{sc}})$-character
is trivial, and that eliminates the need to consider the centralizer of $\phi$ in ${G^\vee}_{\mr{sc}}$.
We denote the set of $G^\vee$-conjugacy classes of enhanced L-parameters for $G$ by $\Phi_e (G)$.

Recall \cite{AMS1} that there exists a notion of cuspidality and a cuspidal support map Sc for 
enhanced L-parameters. The map Sc associates to each $(\phi,\rho) \in \Phi_e (G)$ a $F$-Levi
subgroup $L$ of $G$ and a cuspidal enhanced L-parameter for $L$ (unique up to $G^\vee$-conjugation).
We say that $(\phi,\rho)$ is a principal series L-parameter if $\mr{Sc}(\phi,\rho)$ is an enhanced
L-parameter for $T$ (or a $G$-conjugate of $T$). In that case $\mr{Sc}(\phi,\rho)$ is unique up
to $N_{G^\vee}(T^\vee \rtimes \mb W_F)$-conjugacy. In other words, $\mr{Sc}(\phi,\rho)$ as element of 
$\Phi_e (T)$ is unique up to conjugacy by $N_{G^\vee}(T^\vee \rtimes \mb W_F) / T^\vee$. 

For the maximal torus $T$, the dual group $T^\vee$ is a complex torus. In particular any L-parameter 
for $T$ is trivial on $SL_2 (\C)$ and has trivial component group. Hence an element of $\Phi_e (T)$
is just the $T^\vee$-conjugacy class of a homomorphism $\hat \chi : \mb W_F \to {}^L T$. 
Every element of $\Phi_e (T)$ is cuspidal, because $T$ has no proper Levi subgroups. 

To describe principal series (enhanced) L-parameters more explicitly, we consider an arbitrary 
$(\phi,\rho) \in \Phi_e (G)$. We want to determine $\mr{Sc} (\phi,\rho) = (L,\psi,\epsilon)$. 
By construction 
\begin{equation}\label{eq:2.1}
\psi |_{\mb I_F} = \phi |_{\mb I_F} \; \text{and} \;
\psi \Big( \Fr_F, \matje{q_F^{-1/2}}{0}{0}{q_F^{1/2}} \Big) =
\phi \Big( \Fr_F, \matje{q_F^{-1/2}}{0}{0}{q_F^{1/2}} \Big) . 
\end{equation}
In order that $L = T$, it is necessary that $\phi \Big( \Fr_F, \matje{q_F^{-1/2}}{0}{0}{q_F^{1/2}} 
\Big) \in T^\vee \Fr_F$ and $\phi (i) \in T^\vee i$ for all $i \in \mb I_F$. The group 
\[
H^\vee := Z_{G^\vee}(\phi (\mb W_F))
\] 
is reductive \cite[(4.4.f)]{Vog} and 
\[
R_\phi = \pi_0 \big( Z_{G^\vee}(\phi) / Z(G^\vee)^{\mb W_F}) \quad \text{equals} \quad 
\pi_0 \big( Z_{H^\vee} \big( \phi (SL_2 (\C)) \big) / Z(G^\vee)^{\mb W_F} \big) .
\]
This group is a quotient of 
\[
\pi_0 \big( Z_{H^\vee} \big( \phi (SL_2 (\C)) \big) \big) \cong
\pi_0 \big( Z_{H^\vee} (u_\phi ) \big) ,
\]
where $u_\phi = \phi \big( 1, \matje{1}{1}{0}{1} \big)$. Thus we can regard $\rho$ as an irreducible
representation of $\pi_0 \big( Z_{H^\vee} (u_\phi ) \big)$. Let $(M^\vee, u_\psi, \epsilon)$ be
the cuspidal quasi-support of $(u_\phi,\rho)$ for $H^\vee$, as in \cite[\S 5]{AMS1}. Then $\psi$ is 
the L-parameter determined (up to conjugacy) by \eqref{eq:2.1} and $u_\psi$, while $\epsilon$ is as 
above and $L^\vee = Z_{H^\vee} (Z(M^\vee)^\circ)$.

For $L^\vee = T^\vee$ we need $M^\vee = T^\vee$, which implies that $u_\psi = 1$. There is an explicit
criterium for $\mr{Sc}(u_\phi,\rho) = (T^\vee,1,\epsilon)$ with arbitrary $\epsilon$, as follows. 
Let $\mc B_{H^\vee}^{u_\phi}$ be the variety of Borel subgroups of $H^{\vee, \circ}$ that contain
$u_\phi$, it carries a natural action of $Z_{H^\vee}(u_\phi)$. As $Z_{H^{\vee,\circ}} (u_\phi)$ is
a union of connected components of $Z_{H^\vee} (u_\phi)$, its component group 
$\pi_0 (Z_{H^{\vee,\circ}} (u_\phi))$ is a subgroup of $\pi_0 (Z_{H^\vee} (u_\phi))$. Let $\rho^\circ$ 
be any irreducible constituent of $\rho |_{\pi_0 (Z_{H^{\vee,\circ}} (u_\phi))}$. Then the criterium 
says: $\rho^\circ$ appears in the action of $\pi_0 (Z_{H^{\vee,\circ}} (u_\phi))$ on the (top degree) 
homology of $\mc B_{H^\vee}^{u_\phi}$.
 
Summarising, we found the following necessary and sufficient conditions for $(\phi,\rho) \in \Phi_e (G)$
to be a principal series enhanced L-parameter:
\begin{itemize}
\item[(i)] $\phi \Big( \Fr_F, \matje{q_F^{-1/2}}{0}{0}{q_F^{1/2}} \Big), \phi (i) \in T^\vee 
\rtimes \mb W_F$ for any $i \in \mb I_F$;
\item[(ii)] $\rho^\circ$ appears in $H_* \big( \mc B_{H^\vee}^{u_\phi} \big)$ where 
$H^\vee = Z_{G^\vee}(\phi (\mb W_F))$. 
\end{itemize}
We note that under these conditions $\mr{Sc}(\phi,\rho)$ does not depend on $u_\phi$ or $\rho$. 
Moreover it equals $\mr{Sc}(\phi,\mr{triv})$, because $H^{top} \big( \mc B_{H^\vee}^{u_\phi} \big)$
is a permutation representation of $R_\phi$ (with as permuted objects the irreducible components of
$\mc B_{H^\vee}^{u_\phi}$), and that always contains the trivial representation. With this in mind, 
we call $\phi \in \Phi (G)$ a principal series L-parameter if (i) holds.

Recall from \cite[\S 3.3.1]{Hai} that there is a natural isomorphism
\begin{equation}\label{eq:2.8}
X_\nr (G) \cong \big( Z (G^\vee)^{\mb I_K,\circ} \big)_\Fr , 
\end{equation}
that the group of unramified characters $X_\nr (T)$ is naturally isomorphic to
$(T^{\vee,\mb I_F})^\circ_{\mb W_F}$. We will sometimes identify these groups and write simply
$X_\nr (T)$. We note that $(T^{\vee,\mb I_F})^\circ_{\mb W_F}$ acts on $\Phi (T)$ by
\[
(z \hat \chi) |_{\mb I_F} = \hat \chi |_{\mb I_F},\; (z \hat \chi)(\Fr_F) = z (\hat \chi (\Fr_F))
\]
for $z \in (T^{\vee,\mb I_F})^\circ$ and $\hat \chi \in \Phi (T)$. A Bernstein component of
$\Phi_e (T) = \Phi (T)$ is by definition one $X_\nr (T)$-orbit in $\Phi (T)$. We will usually 
write this as $\mf s^\vee_T = X_\nr (T) \hat \chi$ for one $\hat \chi \in \Phi (T)$. It gives rise
to a Bernstein component $\Phi_e (G)^{\mf s^\vee} := \mr{Sc}^{-1} (T, \mf s_T^\vee)$ in the
principal series part of $\Phi_e (G)$.

Next we make the extended affine Hecke algebra $\mc H (\mf s^\vee, \mb z)$ from \cite{AMS3}
explicit. The maximal commutative subalgebra of $\mc H (\mf s^\vee, \mb z)$ is $\mc O (\mf s^\vee) 
\otimes \C [\mb z, \mb z^{-1}]$, where $\mb z$ is a formal variable. In this context we prefer to 
write $T_{\mf s^\vee}$ for $\mf s^\vee$, to emphasize that it is a complex torus (as a variety, 
in general it does not have a canonical multiplication).

The group $N_{G^\vee}(T^\vee \rtimes \mb W_F) / T^\vee$ acts naturally on $\Phi (T)$, by conjugation.
Let $W_{\mf s^\vee}$ denote the stabilizer of $\mf s^\vee$ in $N_{G^\vee}(T^\vee \rtimes \mb W_F) / 
T^\vee$. Although $W_{\mf s^\vee}$ need not be a Weyl group, it always contains the Weyl group
of a root system. Namely, consider the group $J = Z_{G^\vee}(\hat \chi (\mb I_F))$, with the torus
$(T^{\vee,\mb W_F})^\circ$ and the maximal torus $T^\vee$. According to \cite[Proposition 3.9.a and
(3.22)]{AMS3}, $R \big( J^\circ, (T^{\vee,\mb W_F})^\circ \big)$ is a root system and $W_{\mf s^\vee}$ 
acts naturally on it. Moreover \cite[Proposition 3.9.b]{AMS3} says that for a suitable choice of 
$\hat \chi$ in $T_{\mf s^\vee}$ the set of indivisible roots 
\[
R \big( J^\circ, (T^{\vee,\mb W_F})^\circ \big)_\red \quad \text{equals} \quad
R \big( Z_{G^\vee}(\hat \chi (\mb W_F))^\circ, (T^{\vee,\mb W_F})^\circ \big)_\red .
\]
For such a choice of a basepoint $\hat \chi$ of $T_{\mf s^\vee}$, 
\[
W_{\mf s^\vee}^\circ := W \big( R (J^\circ,(T^{\vee,\mb W_F})^\circ ) \big)
\]
is a normal subgroup of $W_{\mf s}$. Let $R^+ \big( J^\circ, (T^{\vee,\mb W_F})^\circ \big)$ be
the positive root system determined by the Borel subgroup $B^\vee$ of $G^\vee$. Then
$W_{\mf s^\vee} = W_{\mf s^\vee}^\circ \rtimes \Gamma_{\mf s^\vee}$, where $\Gamma_{\mf s^\vee}$
denotes the stabilizer of $R^+ \big( J^\circ, (T^{\vee,\mb W_F})^\circ \big)$ in $W_{\mf s^\vee}$.

The root system for $\mc H (\mf s^\vee, \mb z)$ will essentially be $R \big( J^\circ,(T^{\vee,\mb W_F}
)^\circ \big)$, but we still need to rescale the elements \cite[\S 3.2]{AMS3}. We note that the 
inclusion $(T^{\vee,\mb W_F})^\circ \to T^\vee$ induces a surjection
\[
\mr{pr} : R (J^\circ, T^\vee) \cup \{0\} \to R \big( J^\circ,(T^{\vee,\mb W_F})^\circ \big) \cup \{0\}.
\]
In \cite[Definition 3.11]{AMS3}, positive integers $m_\alpha$ for $\alpha^\vee \in R \big(J^\circ, 
(T^{\vee,\mb W_F})^\circ \big)_\red$ are defined, as follows.
\begin{itemize}
\item Suppose that $\mr{pr}^{-1} (\{ \alpha^\vee \})$ meets $k > 1$ connected components of\\
$R(J^\circ,T^\vee)$. These $k$ components are permuted transitively by $\Fr_F$. Then $m_\alpha$ 
equals $k$ times the analogous number $m'_\alpha$ obtained by replacing $F$ by its degree $k$ 
unramified extension (or equivalently replacing $\Fr_F$ by $\Fr_F^k$).
\item Suppose that $\mr{pr}^{-1}(\{ \alpha^\vee \})$ lies in a single connected component of\\
$R(J^\circ,T^\vee)$. Then $m_\alpha$ is the smallest natural number such that $\ker (m_\alpha 
\alpha^\vee)$ contains the kernel of the canonical surjection
\begin{equation}\label{eq:2.2}
(T^{\vee,\mb W_F})^\circ \to (T^{\vee,\mb I_F})^\circ_{\mb W_F} \cong X_\nr (T) .
\end{equation}
\end{itemize}
In fact it is easy to identify the kernel of \eqref{eq:2.2} as 
\[
(T^{\vee,\mb W_F})^\circ_{\Fr_F} := (T^{\vee,\mb W_F})^\circ \cap (1 - \Fr_F) T^{\vee,\mb I_F}.
\]

\begin{lem}\label{lem:2.1}
The number $m_\alpha$ equals $f(F_\alpha / F)$, where $\mb W_{F_\alpha}$ is the $\mb W_F$-stabilizer 
of a lift of $\alpha^\vee$ to $R(G^\vee,T^\vee)$. 
\end{lem}
\begin{proof}
The $m_\alpha$ can be related to the structure of the $F$-group $\mc G$. Let $\mc G_\alpha$ be
the $F$-simple almost direct factor of $\mc G$ such that $\mr{pr}^{-1}(\{\alpha^\vee\})$ consists
of roots coming from $G_\alpha^\vee$. Write $\mc G_\alpha = \mr{Res}_{E_\alpha / F} \mc H_\alpha$,
where $\mc H_\alpha$ is absolutely simple. The injection $\mc G_\alpha \to \mc G$ induces a
surjection ${}^L G \to {}^L G_\alpha$ which does not change $\alpha^\vee \big( (T^{\vee,\mb W_F}
)^\circ_{\Fr_F} \big)$. Knowing that, the first bullet above says that $m_\alpha$ equals 
$f(E_\alpha/F)$ times the number $m'_\alpha$ for $\mc H_\alpha (E_\alpha)$. 

Let $\mc T_\alpha$ be the maximal torus of $\mc H_\alpha$ with $\mr{Res}_{E_\alpha / F} 
\mc T_\alpha = \mc T \cap \mc G_\alpha$. The Weil group $\mb W_{E_\alpha}$ acts on the irreducible 
root system $R(\mc H^\vee_\alpha,\mc T^\vee_\alpha)$, and the set of orbits is in bijection with
the irreducible component of $R(J^\circ,(T^{\vee,\mb W_F})^\circ)$ containing $\alpha^\vee$.
Let $\mb W_{F_\alpha}$ be the $\mb W_{E_\alpha}$-stabilizer of an element 
$\alpha'^{\vee} \in R(\mc H^\vee_\alpha,\mc T^\vee_\alpha)$ that corresponds to $\alpha^\vee$.
Then $\alpha^\vee = \alpha'^{\vee} |_{T_\alpha^{\mb W_{F_\alpha}}}$.

Suppose that the elements of $\mb W_{E_\alpha} \alpha'^{\vee}$ are mutually orthogonal, which
happens in most cases. From the definitions we see that 
\[
\big| \alpha^\vee \big( (T^{\vee,\mb W_F})^\circ_{\Fr_F} \big) \big| = f (F_\alpha / E_\alpha) .
\]
Here the relevant elements of $(T^{\vee,\mb W_F})^\circ_{\Fr_F}$ are the powers of 
\[
(1 - \Fr_F) t \quad \text{where} \quad
(\Fr^n \alpha'^\vee)(t) = \exp (2 \pi i n / f(F_\alpha / E_\alpha)).  
\]
We find $m'_\alpha = f(F_\alpha / E_\alpha)$ and $m_\alpha = f (F_\alpha / F)$.

Next we consider the cases where the elements of $\mb W_{E_\alpha} \alpha'^{\vee}$
are not mutually orthogonal. Classification shows that $R(\mc H^\vee, \mc T^\vee_\alpha)$ has
type ${}^2 A_{2n}$ and 
\[
|\mb W_{E_\alpha} \alpha'^{\vee}| = [F_\alpha : E_\alpha] = 2,
\] 
so that $H_{\alpha,\ad} \cong PU_{2n+1}(F_\alpha / E_\alpha)$. Direct computations show that,
when $F_\alpha / E_\alpha$ is ramified, $m'_\alpha = 1$ and $m_\alpha = f (E_\alpha / F) =
f(F_\alpha / F)$. Similarly, when $F_\alpha / E_\alpha$ is unramified, $m'_\alpha = 2$ and 
$m_\alpha = 2 f (E_\alpha / F) = f(F_\alpha / F)$.  
\end{proof}

Lemma \ref{lem:2.1} and \cite[Lemma 3.12]{AMS3} yield the precise definition of the root system 
for $\mc H (\mf s^\vee,\mb z)$:
\[
R_{\mf s^\vee} = \big\{ m_\alpha \alpha^\vee : \alpha^\vee \in R \big( J^\circ, 
(T^{\vee,\mb W_F})^\circ \big)_\red \big\} .
\]
This root system is endowed with an action of $W_{\mf s^\vee}$. Hence $W_{\mf s^\vee}$ 
also acts on the resulting root datum from \cite[\S 3.2]{AMS3}:
\begin{align*}
\mc R_{\mf s^\vee} & = \big( R_{\mf s^\vee}, X^* \big( (T^{\vee,\mb I_F})^\circ_{\mb W_F} \big),
R_{\mf s^\vee}^\vee, X_* \big( (T^{\vee,\mb I_F})^\circ_{\mb W_F} \big) \big) \\
& = \big( R_{\mf s^\vee}, T / T_\cpt, R_{\mf s^\vee}^\vee, (T / T_\cpt)^\vee \big) .
\end{align*}
The label functions $\lambda, \lambda^*$ for $\mc H (\mf s^\vee,\mb z)$ are determined in
\cite[Proposition 3.14]{AMS3}. Suppose first that the elements of $\mb W_{E_\alpha} \alpha'^\vee$
are mutually orthogonal (in the notation from the proof of Lemma \ref{lem:2.1}), and that
the same holds for $\alpha^\vee / 2$ whenever $\alpha^\vee / 2$ can be lifted to 
$R(G^\vee,(T^{\vee,\mb W_F})^\circ)$. In these non-exceptional cases
\begin{equation}\label{eq:2.3}
\lambda (m_\alpha \alpha^\vee) = \lambda^* (m_\alpha \alpha^\vee) = m_\alpha = f(F_\alpha / F) .
\end{equation}
If in addition $m_\alpha \alpha^\vee \in 2 X^* \big( (T^{\vee,\mb I_F})^\circ_{\mb W_F} \big)
= 2 (T / T_\cpt)$, then we can get the same Hecke algebras with $m_\alpha \alpha^\vee / 2$
instead of $m_\alpha \alpha^\vee$, and 
\begin{equation}\label{eq:2.7}
\lambda ( m_\alpha \alpha^\vee / 2) = m_\alpha = f(F_\alpha / F) ,\; 
\lambda^* ( m_\alpha \alpha^\vee / 2) = 0.
\end{equation}
We call the remaining cases exceptional, these occur only when $R(G^\vee,T^\vee)$ has a component
of type ${}^2 A_{2n}$ and $\alpha^\vee$ or $\alpha^\vee / 2$ comes from two non-orthogonal
roots that are exchanged by the diagram automorphism. As noted in the proof of Lemma \ref{lem:2.1},
$H_{\alpha,\ad} \cong PU_{2n+1}(F_\alpha / E_\alpha)$. The groups $PU_{2n+1}(F_\alpha / E_\alpha)$,
$SU_{2n+1}(F_\alpha / E_\alpha)$ and $U_{2n+1}(F_\alpha / E_\alpha)$ give the same root system,
the same unramified characters and the same groups $(T^{\vee,\mb W_F})^\circ$. Hence the relevant
data for $H_\alpha$ can be reduced (via its derived group) to those for 
$U_{2n+1}(F_\alpha / E_\alpha)$, and it suffices to continue the analysis in the latter group.

For $U_{2n+1}(F_\alpha / E_\alpha)$ all the labels were computed in \cite[\S 5]{AMS4}. For 
convenience we provide an overview, where we remark that the labels from \cite{AMS4} still have to
be multiplied by $f(E_\alpha/F)$ to account for the restriction of scalars $\mc G_\alpha (F) = 
\mc H_\alpha (E_\alpha)$, as in the proof of Lemma \ref{lem:2.1}. We write $\alpha^\vee = 
\alpha'^\vee + \alpha''^\vee$, where $\alpha'^\vee$ and $\alpha''^\vee$ are non-orthogonal roots 
in $A_{2n}$ exchanged by the diagram automorphism. 
\begin{itemize}
\item $F_\alpha / E_\alpha$ unramified and $\hat \chi (\mb W_{E_\alpha}) \subset 
Z(GL_{2n+1}(\C)) \rtimes \mb W_{E_\alpha}$. Then $\lambda (\alpha^\vee) = 3$ 
and $\lambda^* (\alpha^\vee) = 1$.
\item $F_\alpha / E_\alpha$ unramified and $\hat \chi (\mb W_{E_\alpha}) \not\subset Z(GL_{2n+1}(\C)) 
\rtimes \mb W_{E_\alpha}$. Here we need $\hat \chi (\mb W_{E_\alpha})$ to fix $U_{\alpha^\vee}$ 
pointwise for $\alpha^\vee \in R_{\mf s^\vee}$.
Under that condition $\lambda (\alpha^\vee) = \lambda^* (\alpha^\vee) = 1$.
\item $F_\alpha / E_\alpha$ is ramified. When $\hat \chi \circ \alpha^\vee : F_\alpha \to
\C^\times$ is conjugate-orthogonal, $\alpha^\vee \notin R_{\mf s^\vee}$. Otherwise 
$\hat \chi \circ \alpha^\vee$ is conjugate-symplectic, then $\alpha^\vee \in R_{\mf s^\vee}$ and
$\lambda (\alpha^\vee) = \lambda^* (\alpha^\vee) = 1$. Equivalently, using $\alpha^\vee / 2$ as root:
\begin{equation}\label{eq:2.4}
\lambda (\alpha^\vee / 2) = 1 ,\quad \lambda^* (\alpha^\vee / 2) = 0.
\end{equation}
\end{itemize}
The algebra $\mc H (\mf s^\vee, \mb z)$ has a subalgebra $\mc H (\mf s^\vee,\mb z)^\circ$,
whose underlying vector space is 
\[
\mc O (T_{\mf s^\vee}) \otimes \C [\mb z, \mb z^{-1}] \otimes \C [ W_{\mf s^\vee}^\circ ] . 
\]
It is isomorphic to the affine Hecke algebra $\mc H (\mc R_{\mf s^\vee}, \lambda,\lambda^*,\mb z)$,
for suitable label functions $\lambda,\lambda^*$. The identification of the vector spaces comes 
from the elements $N_w \in \mc H (\mc R_{\mf s^\vee}, \lambda,\lambda^*,\mb z)$ and the
bijection 
\begin{equation}\label{eq:2.6}
(T^{\mb I_F})^\circ_{\mb W_F} \to T_{\mf s^\vee} : t \mapsto t \hat \chi .
\end{equation}

\begin{thm}\label{thm:2.2}
There is a canonical algebra isomorphism 
\[
\mc H (\mf s^\vee,\mb z) \cong \mc H (\mf s^\vee,\mb z)^\circ \rtimes \Gamma_{\mf s^\vee}.
\]
\end{thm}
\begin{proof}
By design $\mc H (\mf s^\vee,\mb z)$ is free as $\mc H (\mf s^\vee,\mb z)^\circ$-module, with a
basis indexed by $\Gamma_{\mf s^\vee}$. More precisely, by \cite[Proposition 3.15.a]{AMS3} the 
actions of $\Gamma_{\mf s^\vee}$ on $R_{\mf s^\vee}, T_{\mf s^\vee}$ and $\mc O (T_{\mf s^\vee})$
naturally induce an action of $\Gamma_{\mf s^\vee}$ on $\mc H (\mf s^\vee,\mb z)^\circ$. 
For every $\gamma \in \Gamma_{\mf s^\vee}$, that yields an element of $\mc H (\mf s^\vee,\mb z)$,
unique up to scaling. 

For the Langlands parameters under consideration, the sheaf $q\mc E$ from \cite{AMS1,AMS3} is
just the constant sheaf with stalk $\C$ on the point $1 \in T^\vee$. It follows that there is
canonical choice for the map $qb_\gamma$ from \cite[(3.39)]{AMS3}, namely the identity. Then
$\gamma \mapsto qb_\gamma$ is multiplicative, the scalars $\lambda_{\gamma,\gamma'}$ in the 
proof of \cite[Proposition 3.15.b]{AMS3} reduce to 1 and $\C [\Gamma_{\mf s^\vee}]$ embeds in 
$\mc H (\mf s^\vee,\mb z)$ as the span of these $qb_\gamma$. With this in place, 
\cite[Proposition 3.15.a]{AMS3} provides the desired statement.
\end{proof}

Next we specialize $\mb z$ to $q_F^{1/2}$, that yields the algebra
\begin{equation}\label{eq:2.5}
\mc H (\mf s^\vee ,q_F^{1/2}) = \mc H (\mf s^\vee ,q_F^{1/2} )^\circ \rtimes \Gamma_{\mf s^\vee}
\cong \mc H (\mc R_{\mf s^\vee}, \lambda,\lambda^* ,q_F^{1/2}) \rtimes \Gamma_{\mf s^\vee}.
\end{equation}
We note that here the isomorphism depends on the choice of the basepoint $\hat \chi$ of 
$T_{\mf s^\vee}$. From \eqref{eq:2.5} we see that the centre of $\mc H (\mf s^\vee,q_F^{1/2})$ is 
\[
Z \big( \mc H (\mf s^\vee,q_F^{1/2}) \big) =
\mc O (T_{\mf s^\vee})^{W_{\mf s}} = \mc O (T_{\mf s^\vee} / W_{\mf s^\vee})
\]
The main use of the algebras \eqref{eq:2.5} lies in the following result.

\begin{thm}\textup{\cite[Theorem 3.18]{AMS3}}
\label{thm:2.3}
There exists a canonical bijection
\[
\begin{array}{ccc}
\Phi_e (G)^{\mf s^\vee} & \to & \Irr \big( \mc H (\mf s^\vee, q_F^{1/2}) \big) \\
(\phi,\rho) & \mapsto & \bar M (\phi, \rho, q_F^{1/2}) 
\end{array}
\]
such that:
\enuma{
\item $\bar M (\phi, \rho, q_F^{1/2})$ admits the central character $W_{\mf s^\vee} \tilde \phi
\in T_{\mf s^\vee} / W_{\mf s^\vee}$, where\\ $\tilde \phi |_{\mb I_F} = \phi |_{\mb I_F}$ 
and $\tilde \phi (\Fr_F) = \phi \big( \Fr_F, \matje{q_F^{-1/2}}{0}{0}{q_F^{1/2}} \big)$.
\item $\phi$ is bounded if and only if $\bar M (\phi, \rho, q_F^{1/2})$ is tempered.
\item $\phi$ is discrete if and only if $\bar M (\phi, \rho, q_F^{1/2})$ is essentially discrete
series and the rank of $R_{\mf s^\vee}$ equals the $F$-split rank of $\mc T / Z(\mc G)$.
\item The bijection is equivariant for the canonical actions of $Z(G^\vee)^{\mb I_F} \cap 
(T^{\vee,\mb I_F})^\circ$. 
}
\end{thm}

We note that in \cite{AMS3} the canonicity is obtained in a slightly weaker sense, by 
interpreting the subalgebra of $\mc H (\mf s^\vee, q_F^{1/2})$ spanned by the $N_\gamma$ with
$\gamma \in \Gamma_{\mf s^\vee}$ as the endomorphism algebra of a certain perverse sheaf
\cite[(2.5)]{AMS3}. We got rid of that subtlety in the proof of Theorem \ref{thm:2.2}.

For part (d) we recall that any element $t \in Z(G^\vee )^{\mb I_F}$ determines a weakly 
unramified character of $G$ \cite[\S 3.3.1]{Hai}, and that character is trivial on $T_\cpt$ if 
and only if $t \in (T^{\vee,\mb I_F})^\circ$. To $t \in Z(G^\vee)^{\mb I_F} \cap 
(T^{\vee,\mb I_F})^\circ$ we associate the automorphism   
\[
x N_w \mapsto x(t) x N_w \qquad x \in T / T_\cpt, w \in W_{\mf s^\vee} 
\]
of $\mc H (\mf s^\vee, q_F^{1/2})$, where $x$ is regarded as function on $T_{\mf s^\vee}$ via 
\eqref{eq:2.6}. The action of $t$ on $\Irr \big( \mc H (\mf s^\vee, q_F^{1/2}) \big)$ is 
composition with the above automorphism.

\section{Comparison of Hecke algebras}
\label{sec:4}

We start with a Bernstein component $\mf s_T$ for $T$. Recall that this is just a $X_\nr (T)$-coset
in $\Irr (T)$. The local Langlands correspondence for tori \cite{Lan,Yu} associates to $\mf s_T$ a
$X_\nr (T)$-orbit in $\Phi (T)$, that is, one Bernstein component $\mf s_T^\vee$ in $\Phi_e (T)$.
Let $W_{\mf s}$ be the stabilizer of $\mf s_T$ in $N_G (T) / T$ and let $W_{\mf s^\vee}$ be
the stabilizer of $\mf s^\vee_T$ in $N_{G^\vee} (T^\vee \rtimes \mb W_F) / T^\vee$.

\begin{lem}\label{lem:4.4}
There is a natural isomorphism $W_{\mf s} \cong W_{\mf s^\vee}$.
\end{lem}
\begin{proof}
The root datum 
\[
(R(\mc G,\mc T), X^* (\mc T), R(G^\vee,T^\vee), X^* (T^\vee))
\]
comes with a natural group isomorphism 
\begin{equation}\label{eq:4.3}
W(\mc G,\mc T) \cong W(G^\vee,T^\vee). 
\end{equation}
From \cite[Proposition 3.1 and its proof]{ABPS3} we know that \eqref{eq:4.3} restricts to an isomorphism
\begin{equation}\label{eq:4.1}
N_G (T) / T \cong N_{G^\vee} (T^\vee \rtimes \mb W_F) / T^\vee = W(G^\vee,T^\vee )^{\mb W_F} .
\end{equation}
The LLC for tori is natural, so compatible with isomorphisms of 
\[
(X^* (T),X_* (T)) = (X_* (T^\vee), X^* (T^\vee)).
\]
In particular it is equivariant for the action of \eqref{eq:4.3}, and hence for the action of 
\eqref{eq:4.1}. Thus the action of $N_G (T) / T$ on $\Irr (T)$ is turned into
the conjugation action of $W(G^\vee,T^\vee )^{\mb W_F}$ by the LLC. In particular \eqref{eq:4.1} 
matches $\mr{Stab}_{N_G (T)/T}(\mf s_T)$ with the $W(G^\vee,T^\vee )^{\mb W_F}$-stabilizer 
of the image $\mf s_T^\vee$ of $\mf s_T$ in $\Phi_e (T)$.
\end{proof}

Similarly we can compare root systems on both sides of the LLC.

\begin{lem}\label{lem:4.1}
There exists a natural bijection between $R_{\mf s}^\vee$ and $R_{\mf s^\vee}$, which preserves
positivity.
\end{lem}
\begin{proof}
Pick any $\chi \in \mf s_T$.

By construction $R_{\mf s}^\vee$ consists of positive multiples of the $\alpha^\vee \in 
R(\mc G,\mc S)^\vee$ for which $\alpha \in R_{\mf s,\mu}$. Similarly $R_{\mf s^\vee}$ consists
of positive multiples of the $\alpha^\vee$ in
\[
R \big( Z_{G^\vee} (\hat \chi (\mb I_F)), T^{\vee, \mb W_F,\circ} \big)_\red
\subset R \big( G^\vee, T^{\vee, \mb W_F,\circ} \big)_\red \cong
R(G^\vee, S^\vee )_\red \cong R (\mc  G, \mc S )_\red^\vee 
\]
for which $\hat \chi (\mb I_F)$ fixes $U_{\alpha^\vee}$ or $U_{2 \alpha^\vee}$ in
$Z_{G^\vee}( \hat \chi (\mb I_F) )$. Since both $R_{\mf s}^\vee$ and $R_{\mf s^\vee}$ are reduced
root systems, this means that there exists at most one bijection $R_{\mf s}^\vee \to R_{\mf s^\vee}$
which scales each root by a positive real number. Positivity in $R_{\mf s}^\vee$ is determined by
$\mc B$ and positivity in $R_{\mf s^\vee}$ is determined by $B^\vee$, so such a bijection would
automatically preserve positivity of roots.

It remains to check that for $R_{\mf s}^\vee$ and $R_{\mf s^\vee}$ the same elements of 
$R(\mc G, \mc S)_\red^\vee$ are relevant. For the non-exceptional roots we know from 
\eqref{eq:1.1}--\eqref{eq:1.5} that $\alpha \in \Sigma_{\mf s,\mu}$ if and only if $\chi \circ 
\alpha^\vee : F_\alpha^\times \to \C^\times$ is unramified. Via the LLC for tori that becomes: 
\[
\alpha^\vee \circ \hat \chi : \mb W_{F_\alpha} \to \C \rtimes \mb W_{F_\alpha} 
\text{ restricts to the identity on } \mb I_{F_\alpha} . 
\]
In this setting the roots in the associated $\mb W_F$-orbit in
$R(G^\vee,T^\vee)$ are mutually orthogonal, permuted by $\mb W_F$ and fixed by $\mb W_{F_\alpha}$.
Hence $\alpha^\vee \circ \hat \chi (\mb I_{F_\alpha})$ fixes $U_{\alpha^\vee}$ pointwise, which 
means that $\alpha$ belongs to $R \big( Z_{G^\vee} (\hat \chi (\mb I_F)), T^{\vee, \mb W_F,\circ} 
\big)_\red$. This argument also works in the opposite direction, so $\alpha^\vee \in R_{\mf s^\vee}$
and only if $\alpha \in R_{\mf s,\mu}$.

For the exceptional roots $\alpha^\vee$ with $s_\alpha \in W_{\mf s} \cong W_{\mf s^\vee}$,
we saw on page \pageref{eq:1.4} and after \eqref{eq:2.7} that on both sides the issue can
be reduced to a unitary group $U_{2n+1}$. From the list of cases at the end of Section \ref{sec:1}
it is clear that if $U_{2n+1}$ is unramified, $\alpha^\vee$ is relevant for $R_{\mf s}^\vee$ if
and only if it is relevant for $R_{\mf s^\vee}$. 

When the involved group $U_{2n+1}$ only splits over a ramified extension, we need to check one more
detail to arrive at the same conclusion. Namely, if $\alpha^\vee \circ \hat \chi : \mb W_{E_\alpha} 
\to \C^\times$ is conjugate-orthogonal (respectively conjugate-symplectic) then $\chi \circ 
\alpha^\vee : \mf o_{E_\alpha}^\vee \to \C^\times$ must be trivial (respectively of order two).
This is exactly \cite[Lemma 3.4]{GGP}.
\end{proof}

Lemma \ref{lem:4.1} implies that the isomorphism \eqref{eq:4.1} restricts to $W_{\mf s}^\circ
\cong W_{\mf s^\vee}^\circ$.
We choose a $W_{\mf s^\vee}^\circ$-invariant base point $\hat \chi_0$ of $\mf s_T^\vee$ as in 
Section \ref{sec:2}. We use the image $\chi_0$ of $\hat \chi_0$ under the LLC as basepoint of
$\mf s_T$. By the aforementioned equivariance of the LLC for tori, $\chi_0$ is invariant under
$W_{\mf s}^\circ$.

Recall that $h_\alpha^\vee \in R_{\mf s}^\vee$ generates $\Q \alpha^\vee \cap T / T_\cpt$. The
element $m_\alpha \alpha^\vee$ does not necessarily generate $\Q \alpha^\vee \cap T / T_\cpt$.
However, since $R_{\mf s^\vee}$ is part of the root datum $\mc R_{\mf s^\vee}$, $m_\alpha 
\alpha^\vee$ is at most divisible by 2 in $T / T_\cpt$ (namely when it is a long root in a type
$C$ root system). For a better comparison, we replace $m_\alpha \alpha^\vee$ by
$m_\alpha \alpha^\vee / 2$ whenever that is possible. That option was already taken into account
in Section \ref{sec:2}. We denote the new multiple of $\alpha^\vee$ by 
$\tilde m_\alpha \alpha^\vee$ and we write
\[
\tilde R_{\mf s^\vee} = \big\{ \tilde m_\alpha \alpha^\vee : \alpha^\vee \in R (J^\circ,
T^{\vee,\mb W_F, \circ}) \big\} .
\]
Now Lemma \ref{lem:4.1} entails that the isomorphism 
\begin{equation}\label{eq:4.2}
X^* (T_{\mf s}) \cong X^* (X_\nr (T)) \cong T / T_\cpt \cong 
X^* \big( (T^{\vee,\mb I_F})^\circ_{\mb W_F} \big),
\end{equation}
induced by the LLC for tori, sends $R_{\mf s}^\vee$ bijectively to $R_{\mf s^\vee}$.

\begin{lem}\label{lem:4.2}
For any $\alpha \in R_{\mf s,\mu}$: $\lambda (h_\alpha^\vee) = \lambda (\tilde m_\alpha \alpha^\vee)$
and $\lambda^* (h_\alpha^\vee) = \lambda^* (\tilde m_\alpha \alpha^\vee)$.
\end{lem}
\begin{proof}
For the non-exceptional roots, this was checked in \eqref{eq:1.5}, \eqref{eq:1.4}, Lemma 
\ref{lem:2.1} and \eqref{eq:2.3}. For exceptional roots (i.e.~those for which the issue can be 
reduced to a unitary group $U_3$), it is verified case-by-case in the lists at the end of 
Section \ref{sec:1} and just before \eqref{eq:2.4}.
\end{proof}

We are ready to prove that the desired isomorphism between Hecke algebras on two sides of the LLC.

\begin{thm}\label{thm:4.3}
There is a canonical algebra isomorphism $\psi_{\mf s} : \mc H (\mf s)^{op} \to 
\mc H (\mf s^\vee,q_F^{1/2})$, given by
\begin{itemize}
\item on $\mc O (T_{\mf s})$, $\psi_{\mf s}$ is induced by the bijection 
$T_{\mf s} \cong T_{\mf s^\vee}$ from the LLC for tori,
\item $\psi_{\mf s} (N_w) = N_{w^{-1}}$ for all $w \in W_{\mf s} \cong W_{\mf s^\vee}$.
\end{itemize}
\end{thm}
\begin{proof}
By Theorem \ref{thm:3.4} there is a unique isomorphism of $\mc O (T_{\mf s^\vee})$-modules 
with these properties. By the $W_{\mf s}$-equivariance of the LLC for tori via \eqref{eq:4.1},
$\mc O (T_{\mf s}) \isom \mc O (T_{\mf s^\vee})$ is $W_{\mf s}$-equivariant.
Combine that with Lemma \ref{lem:4.2} and the multiplication rules in extended affine Hecke
algebras \cite[Proposition 2.2]{AMS3}.
\end{proof}

We note that Theorem \ref{thm:4.3} is compatible with parabolic induction from standard
parabolic and standard Levi subgroups of $G$. Indeed, for a standard Levi subgroup $M$ of
$G$ one obtains the same isomorphism as in Theorem \ref{thm:4.3}, on the subalgebra generated
by $\mc O (T_{\mf s})$ and the $N_w$ with $w \in N_M (T) / T$.

\begin{rem}\label{rem:4.5}
It is also possible to construct a canonical algebra isomorphism $\mc H (\mf s) \cong 
\mc H (\mf s^\vee, q_F^{1/2})$. To that end, we only have to change the condition 
$\psi_{\mf s} (N_w) = N_{w^{-1}}$ in Theorem \ref{thm:4.3} to $\psi_{\mf s} (N_w) = N_w$. 
The two options are related by the canonical isomorphism 
\[
\begin{array}{cccc}
\mc H (\mf s^\vee, q_F^{1/2}) & \isom & \mc H (\mf s^\vee, q_F^{1/2})^{op} \\
f N_w & \to & N_{w^{-1}} f & \quad f \in \mc O (T_{\mf s^\vee}), w \in W_{\mf s^\vee}.
\end{array}
\]
\end{rem}

\section{Parameters of generic representations}
\label{sec:red}

With Theorem \ref{thm:4.3} and \eqref{eq:4.2} we can reformulate Theorem \ref{thm:6.4}
in terms of\\ $\mc H (\mf s^\vee, q_F^{1/2})$-modules. Then it says:
$\pi$ is $(U,\xi)$-generic if and only if 
\begin{equation}\label{eq:6.4}
\Hom_{\mc H ( W_{\mf s^\vee},q_F^\lambda)} \big( \Hom_G (\Pi_{\mf s},\pi), \mr{St} \big) 
\quad \text{is nonzero.}
\end{equation}
We want to investigate which Langlands parameters should correspond to generic representations
in Theorem \ref{thm:2.3}. With the reduction theorems from \cite[\S 8--9]{Lus-Gr} we translate 
the study of (irreducible) representations of $\mc H (\mf s)^{op} \cong \mc H (\mf s^\vee, 
q_F^{1/2})$ to representations of graded Hecke algebras. Subsequently we take a closer look
at the geometric construction of the representations of such algebras. We need to revisit
the methods from \cite{Lus-Gr} and \cite{AMS2,AMS3}, because the aspects we are interested in 
were not considered previously and require quite some details.

\subsection{Reduction to graded Hecke algebras} \

To ease the notation, from now on the elements of $R_{\mf s^\vee}$ will be called just 
$\alpha^\vee$, instead of $m_\alpha \alpha^\vee$ as previously.
For a $\mc H (\mf s^\vee, q_F^{1/2})$-module $V$ and $t \in T_{\mf s^\vee}$ we write
\[
V_t = \{ v \in V : \text{ there exists } n \in \N \text{ such that }
(\theta_x - x(t))^n v = 0 \text{ for all } x \in X \} .
\]
If $V_t$ is nonzero, then we call $t$ a weight of $V$. For a $W_{\mf s^\vee}$-stable subset 
$U \subset T_{\mf s^\vee}$, let $\Mod (\mc H_{\mf s^\vee} )_U$ be the category of finite 
length $\mc H_{\mf s^\vee}$-modules all whose $\mc O (T_{\mf s^\vee})$-weights belong to $U$.
There is a natural equivalence of categories
\[
\begin{array}{ccl}
\Mod \big( \mc H (\mf s^\vee, q_F^{1/2}) \big)_U & \to & \bigoplus_{t \in U / W_{\mf s^\vee}}   
\Mod \big( \mc H (\mf s^\vee, q_F^{1/2}) \big)_{W_{\mf s^\vee} t} \\
V & \mapsto & \bigoplus_{t \in U / W_{\mf s^\vee}} 
\Big( \sum_{w \in W_{\mf s^\vee}} V_{w t} \big) 
\end{array}.
\]
Let $T_{\mf s^\vee,\mr{un}} \subset T_{\mf s^\vee}$ be the maximal compact real subtorus. 
It is homeomorphic to the set of unitary characters in $T_{\mf s} = X_\nr (T) \chi_0$. 
For $u \in T_{\mf s^\vee,\mr{un}}$ we put
\[
R_{\mf s^\vee,u} = \{ \alpha^\vee \in R_{\mf s^\vee} : s_\alpha (u) = u \} .
\]
This is a root system and its Weyl group is contained in $W_{\mf s^\vee,u}$. Recall that
we fixed a Borel subgroup $B^\vee \subset G^\vee$, which provides $R_{\mf s^\vee,u}$
with a notion of positive roots. Let $\Gamma_{\mf s^\vee,u}$ be the stabilizer of
$R_{\mf s^\vee,u}^+ = R_{\mf s^\vee}^+ \cap R_{\mf s^\vee,u}$ in $W_{\mf s^\vee,u}$, then
\[
W_{\mf s^\vee,u} = W(R_{\mf s^\vee,u}) \rtimes \Gamma_{\mf s^\vee,u} .
\]
From these objects we build a new root datum 
\[
\mc R_{\mf s^\vee,u} = \big( R_{\mf s^\vee,u}, X^* (T_{\mf s^\vee}), R_{\mf s^\vee,u}^\vee, 
X_* (T_{\mf s^\vee}) \big) ,
\]
which is endowed with an action of $\Gamma_{\mf s^\vee,u}$. That gives rise to an extended
affine Hecke algebra
\[
\mc H_{\mf s^\vee,u} = \mc H ( \mc R_{\mf s^\vee,u}, \lambda, \lambda^*, q_F^{1/2})
\rtimes \Gamma_{\mf s^\vee,u}. 
\]
We denote the standard generators of this algebra (as $\mc O (T_{\mf s^\vee})$-module) by
$N_{w,u}$, where $w \in W_{\mf s^\vee,u}$. 

The positive part of $X_\nr (T)$ is $X_\nr^+ (T) = \Hom_\Z (T,\R_{>0})$. Via the isomorphism 
\eqref{eq:2.8}, $X_\nr^+ (T)$ can be regarded as a subgroup of $(T^\vee,\mb I_F)_{\mb W_F}$, 
and as such it acts on $T_{\mf s^\vee}$. In particular that yields a subset $W_{\mf s^\vee,u} 
X_\nr^+ (T) u$ of $T_{\mf s^\vee}$. Notice that every element of $T_{\mf s^\vee}$ lies in a 
subset of the form $X_\nr^+ (T) u$ with $u \in T_{\mf s^\vee,\mr{un}}$.

\begin{thm}\label{thm:6.5}
There exists a canonical equivalence of categories 
\[
\begin{array}{cccc}
\ind_u : & \Mod (\mc H_{\mf s^\vee,u} )_{X_\nr^+ (T) u} & \to &
\Mod \big( \mc H (\mf s^\vee, q_F^{1/2}) \big)_{W_{\mf s^\vee} X_\nr^+ (T) u} \\
 & V_{X_\nr^+ (T) u} := \sum_{t \in X_\nr^+ (T)} V_{tu} & \text{\reflectbox{$\mapsto$}} & V
\end{array}
\]
such that:
\enuma{
\item $\ind_u$ is given by localization of the centres on both sides, followed by induction.
\item $\ind_u$ and $\ind_u^{-1}$ preserve central characters.
\item For $V \in \Mod (\mc H_{\mf s^\vee,u} )_{X_\nr^+ (T) u}$ there is an isomorphism}
\[
\Hom_{\mc H (W_{\mf s^\vee},q_F^\lambda)} ( \ind_u V, \mr{St} ) \cong \Hom_{\mc H 
(W_{\mf s^\vee,u},q_F^\lambda)} ( V, \mr{St} ).\
\] 
\end{thm}
\begin{proof}
The original version of this equivalence is \cite[Theorem 8.6]{Lus-Gr}, but the setup is
slightly different there. The version we need, including the canonicity and the group
$\Gamma_{\mf s^\vee}$, is shown in \cite[Theorem 2.1.2]{SolAHA}. Strictly speaking 
$\Gamma_{\mf s_\vee}$ must fix a point of $T_{\mf s^\vee}$ in \cite{SolAHA}. Fortunately, that 
does not play a role in the proof, it works in the generality of our setting because 
we consider $u$ that need not be fixed by $W(R_{\mf s^\vee})$. The properties (a) and (b) are
checked in \cite[Theorem 2.5]{AMS3}.

By \cite[Theorem 2.5]{AMS3} the effect of the thus obtained functor $\ind_u$ on \\
$\mc H (W_{\mf s^\vee},q_F^\lambda)$-modules is
\begin{equation}\label{eq:6.5}
V \; \mapsto \; 
\ind_{\mc H (W_{\mf s^\vee,u},q_F^\lambda)}^{\mc H (W_{\mf s^\vee},q_F^\lambda)} V.
\end{equation}
In this expression $\mc H (W (R_{\mf s^\vee}),q_F^\lambda)$ and
$\C [\Gamma_{\mf s^\vee} \cap \Gamma_{\mf s^\vee, u}]$ are naturally subalgebras of
$\mc H (W_{\mf s^\vee},q_F^\lambda)$, but we have to be careful with
the $\tilde{N}_{w,u}$ for which $w \in \Gamma_{\mf s^\vee, u}$ but $w \not\in
\Gamma_{\mf s^\vee}$. From \cite[\S 8]{Lus-Gr} and \cite[\S 2.1]{SolAHA} one sees that
$\tilde N_{w, u}$ is sent to
\[
\tilde{\mc T}_w 1_{X_\nr^+ (T) u} = \mc T_w 1_{X_\nr^+ (T) u}
\]
in a suitable completion of $\mc H (\mf s^\vee, q_F^{1/2})$. Here $\mc T_w$ is as in
Section \ref{sec:Whit}, transferred to completions of $\mc H (\mf s^\vee, q_F^{1/2})$
via Theorem \ref{thm:4.3}. From \eqref{eq:6.5} and
Frobenius reciprocity (in a suitably completed algebra) we obtain (c).
\end{proof}

With Theorem \ref{thm:6.5} we can reduce the study of $\mc H (\mf s^\vee,q_F^{1/2})$-modules that
admit a central character to modules of another affine Hecke algebra, $\mc H_{\mf s^\vee,u}$,
such that for the new modules the compact part of the central character is fixed by the new
extended Weyl group. In this process all relevant properties of modules are preserved.

Let $\mf T_u (T_{\mf s^\vee})$ be the tangent space of $T_{\mf s^\vee}$ at $u$. It can be identified
with\\ $\C \otimes_\Z X_* (T_{\mf s^\vee})$, so $R_{\mf s^\vee,u}$ can be regarded as a subset of
the cotangent space $\mf T_u^* (T_{\mf s^\vee})$. For $\alpha^\vee \in R_{\mf s^\vee}$ we define
a parameter 
\[
k^u_{\alpha^\vee} = (\lambda (h_\alpha^\vee) + \alpha (u) \lambda^* (h_\alpha^\vee)) \log (q_F)/2 
\in \R_{\geq 0} .
\]
The graded Hecke algebra $\mh H (W(R_{\mf s^\vee,u}), \mf T_u (T_{\mf s^\vee}), k^u)$ is the vector
space\\ $\mc O (\mf T_u (T_{\mf s^\vee})) \otimes \C [W (R_{\mf s^\vee,u})]$ with multiplication
defined by
\begin{itemize}
\item $\mc O (\mf T_u (T_{\mf s^\vee}))$ and $\C [W (R_{\mf s^\vee,u})]$ are embedded as unital subalgebras,
\item for $\alpha^\vee \in R_{\mf s^\vee,u}$ simple and $f \in \mc O (\mf T_u (T_{\mf s^\vee}))$:
\[
s_\alpha f - s_\alpha (f) s_\alpha = k^u_{\alpha^\vee} (f - s_\alpha (f)) / \alpha^\vee .
\]
\end{itemize}
The group $\Gamma_{\mf s^\vee,u}$ acts naturally on this algebra, by
\[
\gamma (w f) = (\gamma w \gamma^{-1}) \, f \circ \gamma^{-1} \qquad w \in W (R_{\mf s^\vee,u}), 
f \in \mc O (\mf T_u (T_{\mf s^\vee})) .
\]
We define the extended graded Hecke algebra 
\[
\mh H_{\mf s,u} = \mh H (W(R_{\mf s^\vee,u}), \mf T_u (T_{\mf s^\vee}), k^u) \rtimes \Gamma_{\mf s^\vee,u} .
\]
Its centre is $\mc O (\mf T_u (T_{\mf s^\vee}))^{W_{\mf s^\vee,u}}$ and weights of 
$\mh H_{\mf s^\vee,u}$-modules are by default considered with respect to the maximal commutative 
subalgebra $\mc O (\mf T_u (T_{\mf s^\vee}))$. Like for affine Hecke algebras, for a 
$W_{\mf s^\vee,u}$-stable subset $U \subset \mf T_u (T_{\mf s^\vee})$ we have the category 
$\Mod (\mh H_{\mf s^\vee,u})_U$ of finite
length modules all whose $\mc O (\mf T_u (T_{\mf s^\vee}))$-weights belong to $U$.

Recall the exponential map for $T_{\mf s^\vee}$ based at $u$:
\[
\begin{array}{cccc}
\exp_u : & \mf T_u (T_{\mf s^\vee}) & \to & T_{\mf s^\vee} \\
& y & \mapsto & u \exp (y) 
\end{array}.
\]

\begin{thm}\label{thm:6.6}
The map $\exp_u$ induces a canonical equivalence of categories
\[
\exp_{u*} : \Mod (\mh H_{\mf s^\vee,u})_{\R \otimes X_* (T_{\mf s^\vee})} \to
\Mod (\mc H_{\mf s^\vee,u})_{X_\nr^+ (T) u} , 
\]
such that:
\enuma{
\item $\exp_{u*}$ comes from an isomorphism (induced by $\exp_u$) between localized versions of
$\mh H_{\mf s^\vee,u}$ and of $\mc H_{\mf s^\vee,u}$.
\item $\exp_{u*}$ does not change the vector spaces underlying the modules.
\item The effect of $\exp_{u*}$ on $\mc O (\mf T_u (T_{\mf s^\vee}))$-weights is $\exp_u$.
\item For any $V \in \Mod (\mh H_{\mf s^\vee,u})_{\R \otimes X_* (T_{\mf s^\vee})}$ there is an
isomorphism}
\[
\Hom_{W_{\mf s^\vee,u}} (V,\det) \cong 
\Hom_{\mc H (W_{\mf s^\vee,u},q_F^\lambda)} (\exp_{u*} V, \mr{St}) .
\]
\end{thm}
\begin{proof}
The original version of this equivalence is \cite[Theorem 9.3]{Lus-Gr}. We use the version from
\cite[Theorem 2.1.4 and Corollary 2.1.5]{SolAHA}. This includes the canonicity and the properties
(a),(b) and (c).

One way to see (d) is via deformations of the parameters. We can scale the parameters $k^u$ linearly
to 0. That gives a family of extended graded Hecke algebras
\[
\mh H_{\mf s^\vee,u,\epsilon} = \mh H (W(R_{\mf s^\vee,u}), \mf T_u (T_{\mf s^\vee}), \epsilon k^u) 
\rtimes \Gamma_{\mf s^\vee,u} \hspace{15mm} \epsilon \in \R_{\geq 0}.
\]
A module $V$ can be ``scaled" to modules $V_\epsilon$, via the scaling
homomorphisms $\mh H_{\mf s^\vee,u,\epsilon} \to \mh H_{\mf s^\vee}$ for $\epsilon \geq 0$
\cite[(1.11)]{SolAHA}. For $\epsilon = 0$ we obtain a module $V_0$ of 
\[
\mh H_{\mf s^\vee,u,0} = \mc O (\mf T_u (T_{\mf s^\vee})) \rtimes W_{\mf s^\vee,u},
\]
which equals $V$ as $\C [W_{\mf s^\vee,u}]$-module and on which $\mc O (\mf T_u (T_{\mf s^\vee}))$
acts by evaluation at $0 \in \mf T_u (T_{\mf s^\vee})$. 

For the affine Hecke algebra $\mc H_{\mf s^\vee,u}$, the parameters can be scaled via $q_F \mapsto
q_F^\epsilon$ with $\epsilon \in [0,1]$. That yields a family of algebras
\[
\mc H_{\mf s^\vee,u,\epsilon} = \mc H (\mc R_u, \lambda, \lambda^*, q_F^\epsilon ) \rtimes 
\Gamma_{\mf s^\vee,u} \hspace{15mm} \epsilon \in \R_{\geq 0}.
\]
The module $\exp_{u*} V$ can be ``scaled" accordingly via a functor
\[
\tilde \sigma_\epsilon : \Mod (\mc H_{\mf s^\vee,u} )_{\chi_+ u} \to
\Mod (\mc H_{\mf s^\vee,u} )_{\chi_+^\epsilon u} \hspace{15mm} \epsilon \in [0,1],
\]
see \cite[Corollary 4.2.2]{SolAHA}. In this process $\mc H (W_{\mf s^\vee,u},q_F^\lambda)$ is
replaced by the isomorphic semisimple algebra $\mc H (W_{\mf s^\vee,u},q_F^{\epsilon \lambda})$.
The multiplicities 
\[
\dim \Hom_{\mc H (W_{\mf s^\vee,u},q_F^{\epsilon \lambda})} \big( \tilde \sigma_\epsilon (\exp_{u*} V),
\mr{St} \big)
\] 
depend continuously on $\epsilon \in [0,1]$ and they are integers, so in fact they are constant as
functions of $\epsilon$. It is known from \cite[(4.6)--(4.7)]{SolAHA} that 
\[
\exp_{u*} (V_\epsilon) = \tilde \sigma_\epsilon (\exp_{u*} V) 
\qquad \text{for all } \epsilon \in [0,1].
\]
We conclude that
\begin{multline*}
\Hom_{\mc H (W_{\mf s^\vee,u},q_F^\lambda)} ( \exp_{u*} V), \mr{St} ) \cong
\Hom_{\mc H (W_{\mf s^\vee,u},q_F^0)} \big( \tilde \sigma_0 (\exp_{u*} V), \mr{St}) \\
\cong \Hom_{W_{\mf s^\vee,u}} (V_0,\det) = \Hom_{W_{\mf s^\vee,u}} (V,\det) . \qedhere
\end{multline*}
\end{proof}

In view of Theorems \ref{thm:6.4}, \ref{thm:6.5} and \ref{thm:6.6}, the role of genericity for
$\mh H_{\mf s^\vee,u}$ is played by modules that the contain the character det of $\C [W_{\mf s^\vee,u}]$.
To analyse those, we bring the algebra in an easier form.

Let $R_{u>0}$ be the subset of $R_{\mf s^\vee,u}$ consisting of the roots $\alpha^\vee$ for which
$k^u_{\alpha^\vee} > 0$. Let $\Gamma_{u>0}$ be the stabilizer of 
$R_{u>0}^+ = R_{u>0} \cap R_{\mf s^\vee}^+$ in $W_{\mf s^\vee,u}$.

\begin{lem}\label{lem:6.2}
The set $R_{u>0}$ is a root system and
\[
\mh H_{\mf s^\vee,u} = \mh H (W(R_{u>0}), \mf T_u (T_{\mf s^\vee}), k^u) \rtimes \Gamma_{u>0}.
\]
\end{lem}
\begin{proof}
The set $R_{u>0}$ is $W_{\mf s^\vee,u}$-stable by the invariance properties of the labels. In particular
it is stable under the reflections with respect to its roots, so it is a root system. In every 
irreducible component of $R_{\mf s^\vee,u}$, $R_{u>0}$ is either everything or empty or the roots of 
one given length. By the simple transitivity of the action of $W(R_{u>0})$ on the collection of positive
systems in $R_{u>0}$:
\[
W_{\mf s^\vee,u} = W(R_{u>0}) \rtimes \Gamma_{u>0} .
\]
We note that
\begin{equation}\label{eq:6.6}
R_{\mf s^\vee,u} \setminus R_{u>0} = \{ \alpha^\vee \in R_{\mf s^\vee,u} :
\lambda (h_\alpha^\vee) = \lambda^* (h_\alpha^\vee), \alpha^\vee (u) = -1 \} .
\end{equation}
By reduction to irreducible root systems, and the classification thereof, one checks that 
$W (R_{\mf s^\vee,u})$ is the semidirect product of $W(R_{u>0})$ and the subgroup generated by the 
reflections with respect to the simple roots in $R_{\mf s^\vee,u} \setminus R_{u>0}$. 
For such reflections the multiplication relations in $\mh H_{\mf s^\vee,u}$ simplify to 
$s_\alpha f = s_\alpha (f) s_\alpha$. That implies 
\[
\mh H (W(R_{\mf s^\vee,u}), \mf T_u (T_{\mf s^\vee}), k^u) =
\mh H (W(R_{u>0}), \mf T_u (T_{\mf s^\vee}), k^u) \rtimes (\Gamma_{u>0} \cap W(R_{\mf s^\vee,u})) ,
\]
which in turn implies the lemma.
\end{proof}

The advantage of Lemma \ref{lem:6.2} is that via the new presentation the algebra becomes isomorphic to 
a graded Hecke algebras with equal parameters.

\begin{lem}\label{lem:6.7}
$\mh H_{\mf s^\vee,u}$ is isomorphic to a graded Hecke algebra (extended by $\Gamma_{u>0}$) 
such that every root from $R_{u>0}$ has the parameter $\log (q_F)$.
\end{lem}
\begin{proof}
By \cite[Proposition 3.14]{AMS3}, $\mh H_{\mf s^\vee,u}$ is isomorphic to the graded Hecke algebra 
associated to a certain complex reductive group $\tilde G$ and a cuspidal L-parameter with values in a 
quasi-Levi subgroup $\tilde M$ of $\tilde G$. In our specific setting $\tilde M^\circ$ is a torus, because 
we only work with principal series L-parameters. In particular the cuspidal L-parameter is trivial on
$SL_2 (\C)$. Thus $\mh H_{\mf s^\vee,u}$ is a graded Hecke algebra associated to $\tilde G$ and a
cuspidal support whose unipotent (or nilpotent) element is trivial. By construction \cite[\S 1]{AMS3}
all the nonzero parameters are of the form $k^u_{\alpha^\vee} = c(\alpha^\vee) r_i$, where $r_i \in \C$ 
depends only on the connected component of the root system that contains $\alpha^\vee$. Further 
$c(\alpha^\vee) = 2$ by \cite[\S 2]{Lus-Cusp1} and our earlier specialization of $\mb z$ to
$q_F^{1/2}$ entails $k_i = \log (q_F^{1/2}) = \log (q_F) /2$. Combine that with Lemma \ref{lem:6.2}.
\end{proof}

\subsection{Geometric representations of graded Hecke algebras} \

Recall that $u$ corresponds to a unitary character of $T$, so it is a bounded L-parameter for $T$. 
By \cite[Proposition 3.14]{AMS3}, the algebra $\mh H_{\mf s^\vee,u}$ is of the form 
$\mh H (u,0,\mr{triv},\log (q_F)/2)$, where triv means the trival local system on the trivial nilpotent
orbit 0. The meaning of this statement is explained somewhat further in \cite[(3.9) and below]{AMS3}.
It can be formulated as 
\[
\mh H (u,0,\mr{triv},\log (q_F)/2) \cong \mh H (G_u^\vee, M^\vee, \mr{triv}, \log (q_F)/2) .
\]
In \cite{AMS3} the group $G_u^\vee$ is defined as $Z^1_{G^\vee_{sc}}(u) \times X_\nr (G)$, but in 
our current setting we have just $G_u^\vee = Z_{G^\vee} (u)$. The reason is that at the start of Section
\ref{sec:2} we refrained from involving the simply connected cover  of $G^\vee_\der$, that would be
superfluous for quasi-split groups. Similarly the group $M^\vee$, which is a quasi-Levi subgroup of
$Z^1_{G^\vee_{sc}}(u) \times X_\nr (G)$ in \cite{AMS3}, becomes simply $T^\vee$ in our setup. 
 
Notice that $G^\vee_u$ need not be connected. In fact the isomorphism 
\begin{equation}\label{eq:6.9}
\mh H (G_u^\vee, T^\vee, \mr{triv}, \log (q_F)/2) \cong \mh H_{\mf s^\vee,u} =
\mh H (W(R_{u>0}), \mf T_u (T_{\mf s^\vee}), k^u) \rtimes \Gamma_{u>0}
\end{equation} 
and Lemma \ref{lem:6.2} imply that $\pi_0 (G_u^\vee) \cong \Gamma_{u>0}$.
When we replace $G_u^\vee$ by its identity component, we obtain the subalgebra
\[
\mh H_{\mf s^\vee,u}^\circ := \mh H (G_u^{\vee,\circ}, T^\vee, \mr{triv}, \log (q_F)/2) \cong 
\mh H (W(R_{u>0}), \mf T_u (T_{\mf s^\vee}), k^u) .
\]
The irreducible representations of such graded Hecke algebras were parametrized and constructed 
geometrically in \cite{Lus-Cusp1,Lus-Cusp2}. The parameters are triples $(\sigma,y,\rho^\circ)$ where:
\begin{enumerate}[(i)]
\item $\sigma \in \mr{Lie}(G_u^{\vee,\circ})$ is semisimple,
\item $y \in \mr{Lie}(G_u^{\vee,\circ})$ is nilpotent,
\item $[\sigma,y] = \log (q_F) y$,
\item $\rho^\circ$ is an irreducible representation of $\pi_0 \big( Z_{G_u^{\vee,\circ}} / Z(G_u^{\vee,\circ}) 
\big)$ satisfying the analogue of (ii) on page \pageref{eq:2.8}. 
\end{enumerate}
By \cite{Lus-Cusp2} $G_u^{\vee,\circ}$-conjugacy classes of such triples are naturally in bijection with 
$\Irr \big( \mh H (G_u^{\vee,\circ}, T^\vee, \mr{triv}, \log (q_F)/2) \big)$. Let us write that as
\[
(\sigma,y,\rho^\circ) \mapsto M^\circ_{y,\sigma,\rho^\circ} .
\] 
In \cite{Lus-Cusp1,Lus-Cusp2} there is an extra parameter $r \in \C$, but we suppress that because in this
paper it will always be equal to $\log (q_F)/2$. From these parameters $\sigma$ can always be chosen in
Lie$(T^\vee)$, and then $W(R_{u>0}) \sigma$ is the central character of $M^\circ_{y,\sigma,\rho^\circ}$.

Lusztig's parametrization was slightly modified in \cite[\S 3.5]{AMS2}, essentially by composing it with the 
Iwahori--Matsumoto involution IM of $\mh H_{\mf s^\vee,u}^\circ$. To make that consistent, the above condition 
(iii) must be replaced by 
\begin{itemize}
\item[(iii')] $[\sigma,y] = -\log (q_F) y$.
\end{itemize} 
We denote the resulting parametrization of $\Irr \big( \mh H (G_u^{\vee,\circ}, T^\vee, \mr{triv}, 
\log (q_F)/2) \big)$, which is the one used in \cite{AMS3}, by
\begin{equation}\label{eq:6.10}
(\sigma,y,\rho^\circ) \mapsto \bar M^\circ_{y,\sigma,\rho^\circ} := \mr{IM}^* M_{y,-\sigma,\rho^\circ} .
\end{equation}

\begin{prop}\label{prop:6.8}
The irreducible $\mh H (G_u^{\vee,\circ}, T^\vee, \mr{triv}, \log (q_F)/2)$-representation 
$\bar M^\circ_{y,\sigma,\rho^\circ}$ contains the sign representation of $\C [W(R_{u>0})]$ if and only if 
$\rho^\circ$ is the trivial representation and the $Z_{G_u^{\vee,\circ}}(\sigma)$-orbit of $y$ is dense in
\[
\{ Y \in \mr{Lie}(G_u^{\vee,\circ}) : [\sigma,Y] = -\log (q_F) Y \} .
\]
\end{prop}
\begin{proof}
We may replace $G_u^{\vee,\circ}$ by any finite covering group, that does not change the associated graded
Hecke algebra. In particular we may assume that the derived group of $G_u^{\vee,\circ}$ is simply connected.

Via \cite[Theorems 2.5 and 2.11]{AMS3}, analogous to Theorems \ref{thm:6.5} and \ref{thm:6.6},
$\bar M^\circ_{y,\sigma,\rho^\circ}$ becomes an irreducible representation of the affine Hecke algebra
associated to $(G_u^{\vee,\circ},T^\vee,\mr{triv})$, with parameter $q_F$. By \cite[Proposition 2.18]{AMS3},
$\bar M^\circ_{y,\sigma,\rho^\circ}$ is turned into the module $\tilde M_{\exp (\sigma), \exp (y), 
\rho^\circ}$ associated by Kazhdan--Lusztig \cite{KaLu} to\\ 
$(\exp (\sigma), \exp (y), \rho^\circ)$ and $q_F$. We note that the paper \cite{KaLu} assumed that the 
derived group of the involved complex reductive group is simply connected.
It was shown in \cite[\S 7.2--7.3]{Ree} that $\tilde M_{\exp (\sigma), \exp (y), \rho^\circ}$
contains the Steinberg representation of $\mc H (W(R_{u>0}),q_F)$ if and only if $\rho^\circ$ is trivial
and the $Z_{G_u^{\vee,\circ}}^\circ$-orbit of $y$ is dense in
\[
\{ Y \in \mr{Lie}(G_u^{\vee,\circ}) : \mr{Ad}(\exp \sigma) Y = q_F^{-1} Y \}
\]
Now we go back to $\mh H (G_u^{\vee,\circ}, T^\vee, \mr{triv}, \log (q_F)/2)$-modules, and we conclude 
with a version of Theorem \ref{thm:6.6}.d.
\end{proof}

The parametrization of $\Irr (\mh H_{\mf s^\vee,u}^\circ)$ from \eqref{eq:6.10} has been generalized
to $\mh H_{\mf s^\vee,u}$ in \cite[Theorem 3.20]{AMS2} and \cite[Theorem 3.8]{AMS3}. The parameters are 
$G_u^\vee$-conjugacy classes of triples $(\sigma,y,\rho)$ as above, with as only difference that $\rho$ 
is now an irreducible representation of $\pi_0 \big( Z_{G_u^\vee}(\sigma,y) / Z (G_u^{\vee,\circ}) \big)$.

The two constructions are related as follows. To $(\sigma,y)$ one associates \cite{Lus-Cusp1} a
$\mh H_{\mf s^\vee,u}^\circ \times \pi_0 (Z_{G_u^{\vee,\circ}})$-representation $E^\circ_{y,-\sigma}$.
Then 
\[
E^\circ_{y,-\sigma,\rho^\circ} = \Hom_{\pi_0 (Z_{G_u^{\vee,\circ}}(\sigma,y))}
(\rho^\circ, E^\circ_{y,-\sigma})
\]
and $M^\circ_{y,-\sigma,\rho^\circ}$ is the unique irreducible quotient of that module.

Similarly a $\mh H_{\mf s^\vee,u} \times \pi_0 (Z_{G_u^\vee})$-representation $E_{y,-\sigma}$ can
be constructed \cite{AMS2}, and by \cite[Lemma 3.3]{AMS2} there is a canonical isomorphism
\begin{equation}\label{eq:6.14}
E_{y,-\sigma} \cong \ind_{\mh H_{\mf s^\vee,u}^\circ}^{\mh H_{\mf s^\vee,u}} E^\circ_{y,\sigma} .
\end{equation}
One defines 
\begin{equation}\label{eq:6.12}
E_{y,-\sigma,\rho} = \Hom_{\pi_0 (Z_{G_u^\vee}(\sigma,y))}(\rho, E_{y,-\sigma}) ,
\end{equation}
and then $M_{y,-\sigma,\rho}$ is the unique irreducible quotient of $E_{y,-\sigma,\rho}$.

\begin{lem}\label{lem:6.13}
Every semisimple $\sigma \in G_u^{\vee,\circ}$ can be extended to a triple as used in \eqref{eq:6.12},
such that $M_{y,-\sigma,\rho}$ contains the trivial representation of $\C [W(R_{u>0}) \rtimes 
\Gamma_{u>0}]$. Moreover $(y,\rho)$ is unique up to $Z_{G_u^\vee}(\sigma)$-conjugacy, $y$ lies in the
dense $Z_{G_u^{\vee,\circ}}(\sigma)$-orbit in 
\[
\{ Y \in \mr{Lie}(G_u^{\vee,\circ}) : [\sigma,Y] = -\log (q_F) Y \} 
\]
and the restriction of $\rho$ to $\pi_0 (Z_{G_u^{\vee,\circ}} (\sigma,y))$ is a multiple of the trivial
representation.
\end{lem}
\begin{proof}
Let $M_{y,-\sigma}$ be the maximal semisimple quotient $\mh H_{\mf s^\vee,u}$-module of $E_{y,-\sigma}$.
Then 
\begin{equation}\label{eq:6.13}
M_{y,-\sigma,\rho} = \Hom_{\pi_0 (Z_{G_u^\vee}(y,\sigma))}(\rho, M_{y,-\sigma}) ,
\end{equation}
for any eligible $\rho$. The same can be done for the analogous $\mh H_{\mf s^\vee,u}^\circ$-modules. 
It follows from \eqref{eq:6.14}, \eqref{eq:6.12} and \eqref{eq:6.13} that
\begin{equation}\label{eq:6.15}
M_{y,-\sigma} \cong \ind_{\mh H_{\mf s^\vee,u}^\circ}^{\mh H_{\mf s^\vee,u}} M^\circ_{y,-\sigma} . 
\end{equation}
Recall that $\mh H_{\mf s^\vee,u} = \mh H_{\mf s^\vee,u}^\circ \rtimes \Gamma_{u>0}$.
By Frobenius reciprocity and \eqref{eq:6.15} the multiplicity of $\mr{triv}_{W(R_{u>0}) \rtimes 
\Gamma_{u>0}}$ in $M_{y,-\sigma}$ equals the multiplicity of $\mr{triv}_{W(R_{u>0})}$ in 
$M^\circ_{y,-\sigma}$.

For any given $\sigma$, Proposition \ref{prop:6.8} and \eqref{eq:6.13} for $\mh H_{\mf s^\vee,u}^\circ$
entail that $\mr{triv}_{W(R_{u>0})}$ appears with multiplicity one in $M^\circ_{y,-\sigma}$ if $y$ 
satisfies the density condition, and otherwise that multiplicity is zero. Hence $M_{y,-\sigma}$
contains $\mr{triv}_{W(R_{u>0}) \rtimes \Gamma_{u>0}}$ if and only if $y$ satisfies the condition
from the statement, and then the multiplicity is one. 

For such $(\sigma,y)$, multiplicity one ensures that there exists a unique $\rho$ such that 
$M_{y,\sigma,\rho}$ contains $\mr{triv}_{W(R_{u>0}) \rtimes \Gamma_{u>0}}$. Let $\rho^\circ$ be an
irreducible subrepresentation of $\rho$ restricted to the normal subgroup 
$\pi_0 (Z_{G_u^{\vee,\circ}} (\sigma,y))$. By Clifford theory the restriction of $\rho$ to
$\pi_0 (Z_{G_u^{\vee,\circ}} (\sigma,y))$ is a multiple of $\bigoplus_g g \cdot \rho^\circ$, where
$g$ runs over $\pi_0 (Z_{G_u^\vee}(\sigma,y))$ modulo the stabilizer of $\rho^\circ$. 

\hspace{-4mm} Suppose that $\rho^\circ$ is nontrivial. Then $g \cdot \rho^\circ$ is nontrivial for any 
$g \in \pi_0 (Z_{G_u^\vee}(\sigma,y))$, and $M_{y,-\sigma,\rho}$ cannot contain any 
$\mh H_{\mf s^\vee,u}^\circ$-submodule of the form $M^\circ_{y',\sigma',\mr{triv}}$. In this case 
$M_{y,-\sigma,\rho}$ does not contain $\mr{triv}_{W(R_{u>0})}$. 
\end{proof}

To see that the parametrization of $\Irr (\mh H^{\mf s^\vee,u})$ from \cite[Theorem 3.8]{AMS3} has 
a property like Proposition \ref{prop:6.8}, it remains to analyse the $\rho$ determined by
Lemma \ref{lem:6.13}. To that end we have to delve more deeply into the underlying constructions.

By the naturality of the parametrization \eqref{eq:6.10}, the $\Gamma_{u>0}$-stabilizer of 
$\bar M^\circ_{y,\sigma,\rho^\circ}$ (or equivalently of $M^\circ_{y,-\sigma,\rho^\circ}$) equals the 
$\Gamma_{u>0}$-stabilizer of the $G_u^{\vee,\circ}$-orbit of $(\sigma,y,\rho^\circ)$. When 
$\rho^\circ = \mr{triv}$, that group depends only on $(\sigma,y)$. When furthermore $y$ satisfies 
the density condition from Proposition \ref{prop:6.8}, the $\Gamma_{u>0}$-stabilizer of 
$\bar M^\circ_{y,\sigma,\mr{triv}}$ equals the $\Gamma_{u>0}$-stabilizer of the $G_u^{\vee,\circ}$-orbit
of $\sigma$, which we denote by $\Gamma_{[\sigma]}$.

\begin{lem}\label{lem:6.9}
Let $(\sigma,y,\mr{triv})$ be as in Proposition \ref{prop:6.8}. The action of 
$\mh H_{\mf s^\vee,u}^\circ$ on $M^\circ_{y,-\sigma,\mr{triv}}$ extends a unique way to an action of 
$\mh H_{\mf s^\vee,u}^\circ \rtimes \Gamma_{[\sigma]}$ that contains the trivial representation of 
$\C [W(R_{u>0}) \rtimes \Gamma_{[\sigma]}]$.
\end{lem}
\begin{proof}
By Proposition \ref{prop:6.8} $M^\circ_{y,-\sigma,\rho^\circ}$ contains the trivial representation of\\
$\C [W(R_{u>0})]$, and by Lemma \ref{lem:6.12} it does so with multiplicity one. Any $\gamma \in 
\Gamma_{[\sigma]}$ stabilizes $M^\circ_{y,-\sigma,\mr{triv}}$, so there exists a linear bijection 
$I_\gamma$ such that
\[
I_\gamma \circ h = \gamma (h) \circ I_\gamma : M^\circ_{y,-\sigma,\rho^\circ} \to 
M^\circ_{y,-\sigma,\rho^\circ} \quad \text{for all } h \in \mh H_{\mf s^\vee,u}^\circ .
\] 
Schur's lemma says that $I_\gamma$ is unique up to scalars. Since $\mr{triv}_{W(R_{u>0})}$ is 
$\Gamma_{[\sigma]}$-stable and appears with multiplicity on in $M^\circ_{y,-\sigma,\rho^\circ}$,
$I_\gamma$ stabilizes the one-dimensional subspace which affords $\mr{triv}_{W(R_{u>0})}$.
We normalize $I_\gamma$ by requiring that it fixes $\mr{triv}_{W(R_{u>0})} \subset 
M^\circ_{y,-\sigma,\rho^\circ}$ pointwise, that is the only possibility if we want to end up
with the trivial representation of $W(R_{u>0}) \rtimes \Gamma_{[\sigma]}$. 

For any $\gamma, \gamma' \in \Gamma_{[\sigma]}$, $I_\gamma \circ I_{\gamma'}$ satisfies the same 
condition as $I_{\gamma \gamma'}$, so equals $I_{\gamma \gamma'}$. These $I_\gamma$ provide 
the desired extension. 
\end{proof}

The module $M^\circ_{y,\sigma,\rho^\circ}$ comes as the unique irreducible quotient of a standard
module $E^\circ_{y,\sigma,\rho^\circ}$ \cite[Theorem 1.15.a]{Lus-Cusp3}. The latter is a subspace of
the homology of the variety $\mc B^y$ of Borel subgroups of $G_u^{\vee,\circ}$ that contain $\exp (y)$,
with coefficients in a certain local system $\dot{\mc L}$. In our setting $\dot{\mc L}$ is trivial
because it comes from the trivial local system on $\{0\}$. In terms of the $\Gamma_{u>0}$-stable
Borel subgroup $B^\vee \cap G_u^{\vee,\circ}$ we have
\[
E^\circ_{y,\sigma,\rho^\circ} \subset H_* (\mc B^y) = H_* \big(
\big\{ g \in G_u^{\vee,\circ} / B^\vee \cap G_u^{\vee,\circ} : \mr{Ad}(g^{-1}) y \in 
\mr{Lie}(B^\vee \cap G_u^{\vee,\circ}) \big\} \big) .
\]
From that and \eqref{eq:6.12} we see that
\begin{equation}\label{eq:6.23}
E^\circ_{y,-\sigma,\mr{triv}} = H_* (\mc B^y )^{Z_{G_u^{\vee,\circ}} (y,\sigma)} .
\end{equation}

\begin{lem}\label{lem:6.10}
Let $(\sigma,y,\mr{triv})$ be as in Proposition \ref{prop:6.8}. Then 
\[
M^\circ_{y,-\sigma,\mr{triv}} = E^\circ_{y,-\sigma,\mr{triv}} =
H_* (\mc B^y )^{Z_{G_u^{\vee,\circ}} (y,\sigma)}
\]
as vector spaces. The subspace $H_0 (\mc B^y )^{Z_{G_u^{\vee,\circ}} (y,\sigma)}$ has dimension one
and\\ $\C [W(R_{u>0})]$ acts on it as the trivial representation. 
\end{lem}
\begin{proof}
By \cite[\S 10.4--10.8]{Lus-Cusp2}, every irreducible subquotient of $E^\circ_{y,-\sigma,\mr{triv}}$
different from $M^\circ_{y,-\sigma,\mr{triv}}$ is of the form $M^\circ_{y',-\sigma,\rho^\circ}$ with
\[
\mr{Ad}(Z_{G_u^{\vee,\circ}}) y \; \subset \; \overline{\mr{Ad}(Z_{G_u^{\vee,\circ}}) y'} .
\]
By the density condition on $y$, such a $y'$ does not exist. Therefore $E^\circ_{y,-\sigma,\mr{triv}}$
is irreducible and equal to $M^\circ_{y,-\sigma,\mr{triv}}$.

Again by \cite[\S 10.4--10.8]{Lus-Cusp2}, $M^\circ_{y,-\sigma,\mr{triv}}$ is a subquotient of
$E^\circ_{0,-\sigma,\mr{triv}}$. As 
\begin{equation}\label{eq:6.11}
\overline{\mr{Ad}(Z_{G_u^{\vee,\circ}}) y} = \{ Y \in \mr{Lie}(G_u^{\vee,\circ}) : [\sigma,Y]
= -\log (q_F) Y \}
\end{equation}
is a vector space, the intersection cohomology complex from the constant sheaf on 
$\mr{Ad}(Z_{G_u^{\vee,\circ}}) y$ is simply the constant sheaf on \eqref{eq:6.11}. In view of 
\cite[\S 10]{Lus-Cusp2}, restricting that sheaf to $\{0\}$ provides a natural nonzero
$\mh H_{\mf s^\vee,u}^\circ$-module homomorphism
\[
E^\circ_{y,-\sigma,\mr{triv}} \to E^\circ_{0,-\sigma,\mr{triv}} .
\]
This realizes $M^\circ_{y,-\sigma,\mr{triv}}$ as subrepresentation of $E^\circ_{0,-\sigma,\mr{triv}}$.

Consider the algebra $A = \mc O (\mr{Lie}(G_u^{\vee,\circ}) / \mr{Ad}(G) \times \C)$ of conjugation
invariant functions on the Lie algebra of $G_u^{\vee,\circ} \times \C^\times$. 
We recall from \cite{Lus-Cusp1} that
\[
E^\circ_{0,-\sigma,\rho^\circ} = \Hom \big( \rho^\circ, E^\circ_{y,\sigma} \big) = 
\Hom \big( \rho^\circ, \C_{-\sigma,\log (q_F)/2} \otimes_A H^A_* (\mc B^0) \big) .
\]
If we replace $\log (q_F)/2$ by an arbitrary $r \in \C$, we still obtain a module for a graded
Hecke algebra, namely $\mh H (G_u^{\vee,\circ}, T^\vee, \mr{triv},r)$.
It is known from \cite[Proposition 7.2]{Lus-Cusp1} that $H_*^A (\mc B^0)$ is a free $A$-module. 
That implies that the modules 
$\C_{-\sigma,r} \otimes_A H^A_* (\mc B^0)$ form an algebraic family parametrized by $r \in \C$ and
a semisimple $\sigma \in \mr{Lie}(G_u^{\vee,\circ})$. In particular, as modules for the finite 
dimensional semisimple subalgebra $\C[W(R_{u>0})]$ they do not depend on $(\sigma,r)$. 

For $r = 0, \sigma = 0$ the group $Z_{G_u^{\vee,\circ}} (\sigma,0) = G_u^{\vee,\circ}$ is connected,
and we obtain $E^\circ_{0,0} = H_* (\mc B^0)$. This is a $\C[W(R_{u>0})]$-representation 
with which the classical Springer correspondence can be constructed. Here we must use the version of 
the Springer correspondence from \cite{Lus-Int}, which by \cite[Theorem 9.2]{Lus-Int} means that the
trivial $W(R_{u>0})$-representation appears as 
\[
H_0 (\mr{pt}) = H_0 ( \mc B^x ) \cong H_0 (\mc B^0)
\]
for a regular unipotent element $x \in G_u^{\vee,\circ}$. As $\dim H_0 (\mc B^0) = 1$, the parts of
\[
M^\circ_{y,-\sigma,\mr{triv}} \; \subset \; E^\circ_{0,-\sigma,\mr{triv}} \; \subset \; 
\C_{-\sigma,\log (q_F)/2} \otimes_A H^A_* (\mc B^0) 
\]
in homological degree zero also have dimension one and carry the trivial representation of $W(R_{u>0})$.
\end{proof}

We are ready to prove the desired generalization of Proposition \ref{prop:6.8}.

\begin{thm}\label{thm:6.11}
There exists a canonical bijection between $\Irr (\mh H_{\mf s^\vee,u})$ and the set of
$G_u^\vee$-conjugacy classes of triples $(\sigma,y,\rho)$, where
\begin{itemize}
\item $\sigma, y \in \mr{Lie}(G_u^\vee)$ with $\sigma$ semisimple, $y$ nilpotent and
$[\sigma,y] = -\log (q_F) y$,
\item $\rho$ is an irreducible representation of $\pi_0 \big( Z_{G_u^\vee}(\sigma,y) /
Z (G_u^{\vee,\circ}) \big)$, such that any irreducible $\pi_0 \big( Z_{G_u^{\vee,\circ}}(\sigma,y) 
/ Z (G_u^{\vee,\circ}) \big)$-subrepresentation appears in the homo\-logy of the
variety of Borel subgroups of $G_u^{\vee,\circ}$ that contain $\exp (\sigma)$ and $\exp (y)$.
\end{itemize}
The module $\bar M_{y,\sigma,\rho}$ associated to $(\sigma,y,\rho)$ contains the character $\det$ of
$\C [W_{\mf s^\vee,u}]$ if and only if $\rho$ is trivial and the Ad$(Z_{G_u^{\vee,\circ}}(\sigma))
$-orbit of $y$ is dense in 
\[
\{ Y \in \mr{Lie}(G_u^{\vee,\circ}) : [\sigma,Y] = -\log (q_F) Y \} .
\]
\end{thm}
\begin{proof}
We start with the parametrization of $\Irr (\mh H (G_u^\vee, T^\vee,\mr{triv}, \log (q_F)/2)$ 
provided by \cite[\S 3.5]{AMS2} and \cite[Theorem 3.8]{AMS3}. This has almost all the required 
properties, only the action of $\C [\Gamma_{u>0}]$ on the thus constructed modules can still be 
normalized in several ways. 

Fix a nilpotent $y \in \mr{Lie}(G_u^{\vee,\circ})$ and consider the variety
\[
\mc P_y := \big\{ g \in G_u^\vee / B^\vee \cap G_u^{\vee,\circ} : \mr{Ad}(g^{-1}) y \in 
\mr{Lie}(B^\vee \cap G_u^{\vee,\circ}) \big\} .
\]
The $\mh H_{\mf s^\vee,u} \times \pi_0 (Z_{G_u^\vee}(\sigma,y))$-representation $E_{y,-\sigma}$
equals $H_* (\mc P^y)$ as vector space. The action of $\pi_0 (Z_{G_u^\vee}(\sigma,y))$ on
$H_* (\mc P^y)$ is induced by the natural left action of $Z_{G_u^\vee}(\sigma,y)$ on $\mc P_y$. 
An element $\gamma \in \Gamma_{u>0}$ acts on $\mc P^y$ by
\begin{equation}\label{eq:6.16}
r_\gamma^{-1} : g (B^\vee \cap G_u^{\vee,\circ}) \mapsto g \gamma^{-1} B^\vee \cap G_u^{\vee,\circ} ,
\end{equation}
which in fact makes $\mc P_y$ isomorphic to $\mc B^y \times \Gamma_{u>0}$. We normalize the
action of $\C [\Gamma_{u>0}]$ on $E_{y,-\sigma}$, by defining it as $H_* (r_\gamma^{-1})$. 
(This normalization was not possible in \cite{AMS2}, because there the homology of $\mc P_y$ had 
coefficients in a local system that could be nontrivial.)

From now on we assume that $y$ satisfies the density condition from the statement.
In view of Lemma \ref{lem:6.13}, it remains to analyse the $\pi_0 (Z_{G_u^\vee}(\sigma,y))
$-invariants in $E_{y,-\sigma}$. We recall from \cite[Lemma 3.12]{AMS2} that
\[
\Gamma_{[\sigma]} \cong \pi_0 (Z_{G_u^\vee}(\sigma,y)) / \pi_0 (Z_{G_u^{\vee,\circ}}(\sigma,y)) .
\]
Let $\gamma_{\sigma,y} \in Z_{G_u^\vee}(\sigma,y)$ be a representative of $\gamma \in 
\Gamma_{[\sigma]}$. Then $H_* (\mc P_y)^{\pi_0 (Z_{G_u^\vee}(\sigma,y))}$ consists of the invariants
for $\{ \gamma_{\sigma,y} : \gamma \in \Gamma_{[\sigma]} \}$ in
\begin{equation}\label{eq:6.17}
H_* (\mc P_y)^{\pi_0 (Z_{G_u^{\vee,\circ}}(\sigma,y))} = \bigoplus\nolimits_{w \in \Gamma_{u>0}} 
H_* (w \cdot \mc B^y)^{\pi_0 (Z_{G_u^{\vee,\circ}}(\sigma,y))} .
\end{equation}
Fix $\gamma \in \Gamma_{[\sigma]}$ and consider the map
\begin{equation}\label{eq:6.18}
\begin{array}{cccc}
f_\gamma : & \mc B^y & \to & \mc B^y \\
& g (B^\vee \cap G_u^{\vee,\circ}) & \mapsto &
\gamma_{\sigma,y} g \gamma^{-1} (B^\vee \cap G_u^{\vee,\circ})
\end{array}.
\end{equation}
It can be decomposed as 
\[
f_\gamma = l_{\gamma_{y,\sigma}} \circ \rho_{\gamma}^{-1} = 
r_{\gamma}^{-1} \circ l_{\gamma_{y,\sigma}} 
\]
The induced map on $H_* (\mc B^y)^{\pi_0 (Z_{G_u^{\vee,\circ}}(\sigma,y))}$ is the composition of 
the action of an $\mh H_{\mf s^\vee,u}$-intertwiner $H_* (l_{\gamma_{y,\sigma}})$  from 
$\pi_0 (Z_{G_u^{\vee,\circ}}(\sigma,y))$ and the action $H_* (\rho_\gamma^{-1})$ of 
$\gamma \in \mh H_{\mf s^\vee,u}$, so it is an $\mh H_{\mf s^\vee,u}^\circ$-intertwiner
\[
H_* (f_\gamma) : E^\circ_{y,-\sigma,\mr{triv}} \to \gamma \cdot E^\circ_{y,-\sigma,\mr{triv}} .
\]
Let $\pi_{y,-\sigma}$ be the extension of $E^\circ_{y,-\sigma,\mr{triv}} = 
M^\circ_{y,-\sigma,\mr{triv}}$ to an $\mh H_{\mf s^\vee,u} \rtimes \Gamma_{[\sigma]}$-repre\-sen\-ta\-tion
from Lemma \ref{lem:6.10}. Consider the composition
\[
\pi_{y,-\sigma} (\gamma^{-1}) \circ H_* (f_\gamma) \in \End_{\mh H_{\mf s^\vee,u}^\circ} 
(E^\circ_{y,-\sigma,\mr{triv}}) .
\]
By Schur's lemma this is a scalar, say $\lambda \in \C$. We know from Lemma \ref{lem:6.10} that
$H_0 (\mc B^y)^{\pi_0 (Z_{G_u^{\vee,\circ}}(\sigma,y))}$ has dimension one, so in terms of
simplicial homology it is spanned by an element $v$ that is the sum of one point from every 
connected component of $\mc B^y$. That $v$ is fixed by $H_0 (f_\gamma)$ because \eqref{eq:6.18}
is a homeomorphism. By Lemmas \ref{lem:6.9} and \ref{lem:6.10} also $\pi_{y,-\sigma}(\gamma^{-1}) v
= v$. Hence $\lambda = 1$ and $\pi_{y,-\sigma} (\gamma^{-1}) \circ H_* (f_\gamma)$ is the identity.
Equivalently,
\[
H_* (l_{\gamma_{y,\sigma}}) = H_* (r_\gamma) \circ \pi_{y,-\sigma} (\gamma) :
H_* (\mc B^y)^{\pi_0 (Z_{G_u^{\vee,\circ}}(\sigma,y))} \to 
H_* (\gamma \cdot \mc B^y)^{\pi_0 (Z_{G_u^{\vee,\circ}}(\sigma,y))} .
\]
Specializing to $H_0 (\mc B^y)^{\pi_0 (Z_{G_u^{\vee,\circ}}(\sigma,y))} = \C v$ we obtain
\[
H_0 (l_{\gamma_{y,\sigma}}) = H_0 (r_\gamma) : H_0 (\mc B^y)^{\pi_0 (Z_{G_u^{\vee,\circ}}
(\sigma,y))} \to H_0 (\gamma \cdot \mc B^y)^{\pi_0 (Z_{G_u^{\vee,\circ}}(\sigma,y))} .
\]
It follows that $H_0 (\mc P^y)^{\pi_0 (Z_{G_u^\vee}(\sigma,y))}$ contains the nonzero vector
\[
\sum_{\gamma \in \Gamma_{u>0}} H_0 (r_\gamma^{-1}) v = 
\sum_{w \in \Gamma_{u>0} / \Gamma_{[\sigma]}} H_0 (r_w^{-1}) 
\sum_{\gamma \in \Gamma_{[\sigma]}} H_0 (l_{\gamma_{y,\sigma}}) v_0.
\]
Lemma \ref{lem:6.10} shows that this an element of $E_{y,-\sigma,\mr{triv}}$ fixed by 
$W(R_{u>0}) \rtimes \Gamma_{u>0}$. In other words,
\begin{equation}\label{eq:6.19}
\mr{IM}^* M_{y,-\sigma,\mr{triv}} \text{ contains the }
\C [W(R_{u>0}) \rtimes \Gamma_{u>0}]\text{-representation } \mr{sign} \rtimes \mr{triv} .
\end{equation}
Lemma \ref{lem:6.13} says that only the triples $(y,\sigma,\rho)$ of the kind indicated in
the statement have that property.

Finally, we slightly modify the construction from \cite[\S 3.5]{AMS2}. Instead of extending the
Iwahori--Matsumoto involution from $\mh H_{\mf s^\vee,u}^\circ$ to $\mh H_{\mf s^\vee,u}$ by 
making it the identity on $\C [\Gamma_{u>0}]$,  
\begin{equation}\label{eq:6.20}
\text{we extend } \IM \text{ to } \mh H_{\mf s^\vee,u} 
\text{ as multiplication by } \det \text{ on } \C[\Gamma_{u>0}].
\end{equation} 
Then \eqref{eq:6.19} becomes: $\IM^* M_{y,-\sigma,\rho}$ contains 
$\det_{W_{\mf s^\vee,u}}$ if and only if $(\sigma,y,\rho)$ is as stated in the theorem.
\end{proof}

We note that \eqref{eq:6.20} only differs from the usual Iwahori--Matsumoto involution on
the extended graded Hecke algebra $\mh H_{\mf s^\vee,u}$ by the automorphism
\begin{equation}\label{eq:6.21}
\det\nolimits_{\Gamma_{u>0}} : \gamma w f \mapsto \det (\gamma) \gamma w f \qquad
\gamma \in \Gamma_{u>0}, w \in W(R_{u>0}), f \in \mc O (\mf T_u (T_{\mf s^\vee})) .
\end{equation}
Since $\det_{\Gamma_{u>0}}$ is the identity on $\mc O (\mf T_u (T_{\mf s^\vee}))$, it preserves
all the properties (e.g. temperedness) that we need later on. The advantage of \eqref{eq:6.20}
over IM from \cite{AMS2} is that all reflections in $W(R_{u>0}) \rtimes \Gamma_{u>0}$ are 
treated in the same way, irrespective of their $q$-parameters in some Hecke algebra. This is 
needed to express Theorem \ref{thm:6.11} with $\det |_{W_{\mf s^\vee,u}}$.

We wrap up this section by combining the main results.

\begin{lem}\label{lem:6.14}
We modify \cite[Theorem 3.18]{AMS3} (see Theorem \ref{thm:2.3}) by using \eqref{eq:6.20} instead
of the involution $\IM$ from \cite[\S 3.5]{AMS2}. That yields a canonical bijection
\[
\begin{array}{ccc}
\Phi_e (G)^{\mf s^\vee} & \to & \Irr \big( \mc H (\mf s^\vee, q_F^{1/2}) \big) \\
(\phi,\rho) & \mapsto & \bar M (\phi, \rho, q_F^{1/2}) 
\end{array}
\]
such that:
\begin{itemize}
\item It has all the properties listed in \cite[Theorem 3.18]{AMS3}.
\item $\bar M (\phi, \rho, q_F^{1/2})$ contains the Steinberg representation of 
$\mc H (W_{\mf s^\vee},q_F^\lambda)$ if and only if $\rho$ is trivial and 
$\log \phi (1, \matje{1}{1}{0}{1})$ lies in the dense 
$Z_{G^\vee}\big(\tilde \phi (\mb W_F)\big)$-orbit in
\[
\big\{ Y \in \mr{Lie}\big( Z_{G^\vee}(\phi (\mb I_F)) \big) : 
\mr{Ad}(\tilde \phi (\Fr_F)) Y = q_F^{-1} Y \big\} .
\]
Here $\tilde \phi |_{\mb I_F} = \phi |_{\mb I_F}$ and 
$\tilde \phi (\Fr_F) = \phi \big( \Fr_F, \matje{q_F^{-1/2}}{0}{0}{q_F^{1/2}} \big)$.
\end{itemize}
\end{lem}
\begin{proof}
By design \cite[Theorem 3.18]{AMS3} for $\mc H (\mf s^\vee,q_F^{1/2})$ is the composition of
Theorems \ref{thm:6.5}, \ref{thm:6.6} and \ref{thm:6.11}, with \eqref{eq:6.20} as only 
difference. Since each of the three involved bijections is canonical, so is our version of 
\cite[Theorem 3.18]{AMS3}. 
As explained after \eqref{eq:6.21}, the automorphism $\det_{\Gamma_{u>0}}$ does not destroy any of the
properties from \cite[Theorem 3.18]{AMS3}, so our bijection still satisfies all those properties.

In Theorem \ref{thm:6.11} we found a necessary and sufficient condition so that  \\
$\bar M_{y,\sigma,\rho} \in \Irr (\mh H_{\mf s^\vee})$ contains $\det_{W_{\mf s^\vee,u}}$. With Theorem
\ref{thm:6.6} we can translate that to $\exp_{u*} \bar M_{y,\sigma,\rho}$. Thus the latter module 
contains the representation St of ${\mc H (W_{\mf s^\vee,u},q_F^\lambda)}$ if and only if $\rho$ 
is trivial and the orbit of $y$ is dense in
\[
\{ Y \in \mr{Lie}(Z_{G^\vee}(\phi (\mb I_F)) : \mr{Ad}(u \exp (\sigma)) Y = q_F^{-1} Y \} .
\]
With Theorem \ref{thm:6.5} we transfer that to a property of
\[
\ind_u^{-1} \exp_{u*} \bar M_{y,\sigma,\rho} \in \Irr \big( \mc H (\mf s^\vee, q_F^{1/2}) \big) . 
\]
The translation to L-parameters from \cite{AMS3} is such that $\Sc (\phi,\rho) = u \exp (\sigma)$ and 
$y = \log \phi (1, \matje{1}{1}{0}{1})$. Thus we recover the characterization of genericity stated in
the lemma.
\end{proof}

\section{A canonical local Langlands correspondence}
\label{sec:5}

Recall that we fixed a quasi-split group $G = \mc G (F)$, a maximal split torus $S$ of $G$,
a Borel subgroup $B \subset G$ containing $T = Z_G (S)$ and a Whittaker datum $(U,\xi)$.
Given $G$, only $\xi$ is really a choice, the other objects are unique up to $G$-conjugacy.

We denote the space of irreducible $G$-representations in the principal series by $\Irr (G,T)$,
and we write $\Phi_e (G,T)$ for the set of principal series enhanced L-parameters in $\Phi_e (G)$. 

\begin{thm}\label{thm:5.1}
The Whittaker datum $(U,\xi)$ determines a canonical bijection
\[
\begin{array}{ccc}
\Irr (G,T) & \longleftrightarrow & \Phi_e (G,T) \\
\pi & \mapsto & (\phi_\pi, \rho_\pi) \\
\pi (\phi,\rho) & \text{\reflectbox{$\mapsto$}} & (\phi,\rho) 
\end{array} .
\]
\end{thm}
\begin{proof}
Recall from \eqref{eq:4.1} that the LLC for tori provides a $N_G (T)/T$-equivariant bijection 
between the Bernstein components of $\Irr (T)$ and the Bernstein components of $\Phi_e (T)$, say
$\mf s_T \mapsto \mf s_T^\vee$. 

Every principal series Bernstein component $\Irr (G)^{\mf s}$ of $\Irr (G)$ determines a unique
$N_G (T)/T$-orbit of Bernstein components $\Irr (T)^{\mf s_T}$. Similarly every principal series
Bernstein component $\Phi_e (G)^{\mf s^\vee}$ determines a unique $N_G (T)/T$-orbit of Bernstein
components $\Phi_e (G)^{\mf s_T^\vee}$. Thus the LLC for tori induces a natural bijection between
the Bernstein components of $\Irr (G,T)$ and those of $\Phi_e (G,T)$. We denote it by
$\Irr (G)^{\mf s} \mapsto \Phi_e (G)^{\mf s^\vee}$, where typically $\mf s = [T,\chi_0 ]_G$
and $\mf s^\vee  = (T,\hat \chi_0 X_\nr (T) )$. From respectively \eqref{eq:1.6} and Theorem
\ref{thm:3.4}, Theorem \ref{thm:4.3} and Theorem \ref{thm:2.3} in the form of Lemma 
\ref{lem:6.14}, we obtain canonical bijections
\begin{equation}\label{eq:5.1}
\Irr (G)^{\mf s} \leftrightarrow \Irr (\End_G (\Pi_{\mf s})^{op}) \leftrightarrow 
\Irr (\mc H (\mf s)^{op}) \leftrightarrow \Irr \big( \mc H 
(\mf s^\vee, q_F^{1/2}) \big) \leftrightarrow \Phi_e (G)^{\mf s^\vee}.
\end{equation}
Suppose we represent $\mf s$ instead by $w \mf s_T = [T,w \chi_0]_T$ with $w \in W(\mc G,\mc S)
= N_G (T) /T$. Clearly we may assume that $w$ has minimal length in $w W_{\mf s}$. Start
with any $\pi \in \Irr (G)^{\mf s}$ and follow \eqref{eq:5.1} to obtain 
\[
\pi_{\mf s} \in \Irr (\mc H (\mf s)^{op}),\; \pi_{\mf s^\vee} \in 
\Irr \big( \mc H (\mf s, q_F^{1/2}) \big) \text{ and } (\phi_\pi ,\rho_\pi) \in \Phi_e (G,T).
\] 
We use the same notations with $w \mf s$ instead of $\mf s$. Proposition \ref{prop:3.9} 
implies that $\pi_{w \mf s} = \pi_{\mf s} \circ \mr{Ad}(\phi_w)$ where
\[
\phi_w (f N_v) = (f \circ w) N_{w^{-1} v w} \quad \text{for } f \in \mc O (T_{w \mf s}),
v \in W_{w \mf s} .
\]
Now we consider $w$ as element of $N_{G^\vee}(T^\vee \rtimes \mb W_F) / T^\vee$ via \eqref{eq:4.1},
and we define an algebra isomorphism
\[
\begin{array}{llll}
\mr{Ad}(\phi_w)^\vee : & \mc H (w \mf s^\vee, q_F^{1/2}) & \to & \mc H (\mf s^\vee,q_F^{1/2}) \\
 & f N_v & \mapsto & (f \circ w) N_{w^{-1} v w} \qquad 
 f \in \mc O (T_{w \mf s^\vee}),v \in W_{w \mf s^\vee} .
\end{array} 
\]
With Theorem \ref{thm:4.3} we find that $\pi_{w \mf s}$ is matched with 
$\pi_{w \mf s^\vee} = \pi_{\mf s^\vee} \circ \mr{Ad}(\phi_w)^\vee$. All the constructions 
behind Theorem \ref{thm:2.3} and Lemma \ref{lem:6.14} are equivariant for automorphisms of 
$(G^\vee \rtimes \mb W_F,T^\vee \rtimes \mb W_F)$ which preserve the projections to $\mb W_F$ 
and are algebraic on $G^\vee$. This means that $\pi_{\mf s^\vee} \circ \mr{Ad}(\phi_w)^\vee$ 
is parametrized by $(w \phi_\pi w^{-1}, w \cdot \rho_\pi)$, for any 
representative of $w$ in $N_{G^\vee}(T^\vee \rtimes \mb W_F)$. As $(w \phi_\pi w^{-1}, 
w \cdot \rho_\pi)$ equals $(\phi_\pi, \rho_\pi)$ in $\Phi_e (G)$, we deduce that the bijection 
between $\Irr (G)^{\mf s}$ and $\Phi_e (G)^{\mf s^\vee}$ from \eqref{eq:5.1} does not 
depend on the choice of an inertial equivalence class for $T$ underlying $\mf s$. 

Knowing that, we can unambiguously take the union of the bijections \eqref{eq:5.1}
over all Bernstein components of $\Irr (G,T)$.
\end{proof}

\begin{rem}\label{rem:5.13}
If we had used the isomorphism $\mc H (\mf s) \cong \mc H (\mf s^\vee, q_F^{1/2})$ from
Remark \ref{rem:4.5} instead of Theorem \ref{thm:4.3}, then \eqref{eq:5.1} would provide 
a canonical bijection between $\Irr (G)^{\mf s}$ and $\Irr \big( \mc H (\mf s^\vee, 
q_F^{1/2})^{op} \big)$. That could be more natural, depending on the point of view.
\end{rem}

In the remainder of this section we will show that the bijection from Theorem \ref{thm:5.1}
has many desirable properties. 

The definition of $\tilde \phi$ in Lemma \ref{lem:6.14} applies to any Langlands parameter 
$\phi \in \Phi (G)$. The group $Z_{G^\vee}(\tilde \phi (\mb W_F))$ acts by conjugation on the
variety
\[
V_{\tilde \phi} = \big\{ v \in Z_{G^\vee}(\phi (\mb I_F)) : v \text{ is unipotent and }
\tilde \phi (\Fr_F)^{-1} v \tilde \phi (\Fr_F) = v^{q_F} \big\} .
\]
It is known from \cite[Proposition 5.6.1]{CFMMX} that $V_{\tilde \phi}$ is an affine space over
$\C$ on which $Z_{G^\vee}(\tilde \phi (\mb W_F))$ acts with finitely many orbits, of which
exactly one is open. Following \cite[\S 0.6]{CFZ}, we call $\phi \in \Phi (G)$ open if
$u_\phi \in V_{\tilde \phi}$ lies in the open $Z_{G^\vee}(\tilde \phi (\mb W_F))$-orbit.

\begin{lem}\label{lem:5.11}
The representation $\pi (\phi,\rho) \in \Irr (G,T)$ is $(U,\xi)$-generic if and only if
$\phi$ is open and $\rho$ is trivial.
\end{lem}
\begin{proof}
By Theorem \ref{thm:6.4}, $\pi (\phi,\rho)$ is $(U,\xi)$-generic if and only if the
$\End_G (\Pi_{\mf s})^{op}$-module $\Hom_G (\Pi_{\mf s}, \pi (\phi,\rho))$ contains St.
Via Theorems \ref{thm:4.3} and \ref{thm:2.3} that becomes the analogous statement for
$\mc H (\mf s^\vee,q_F^{1/2})$-representations. In Lemma \ref{lem:6.14} we showed the equivalence 
with the stated conditions on $\phi$ and $\rho$, except unipotency. The conditions in Lemma 
\ref{lem:6.14} imply that $\log u_\phi$ must be nilpotent. Hence $u_\phi$ must be unipotent
(as is any case required for Langlands parameters).
\end{proof}

We note that Lemma \ref{lem:5.11} agrees with the Reeder's findings \cite{Ree1,Ree} for generic
unipotent representations and generic principal series representations, in both cases for
split reductive $p$-adic groups with connected centre.

For the next properties of our LLC, the setup will be similar to \cite[\S 5]{SolLLCunip}.

\begin{lem}\label{thm:5.2}
Theorem \ref{thm:5.1} is compatible with direct products of quasi-split $F$-groups.
\end{lem}
\begin{proof}
If $\mc G = \mc G_1 \times \mc G_2$, then all involved objects for $\mc G$ are naturally
products of the analogous objects for $G_1$ and $G_2$.
\end{proof}

Recall that the group of (smooth) characters $\Hom (G,\C^\times)$ is naturally isomorphic to
$H^1 (\mb W_F, Z(G^\vee))$. The former group acts on $\Irr (G)$ by tensoring, and that action
commutes with the supercuspidal support map so stabilizes $\Irr (G,T)$.

On the other hand, $H^1 (\mb W_F, Z(G^\vee))$ acts on $\Phi (G)$ by multiplication of maps 
$\mb W_F \times SL_2 (\C) \to G^\vee$, where $H^1 (\mb W_F,Z(G^\vee))$ gives maps that do not use
$SL_2 (\C))$. That action does not change $R_\phi$, so it induces an action of 
$H^1 (\mb W_F, Z(G^\vee))$ on $\Phi_e (G)$ which does not change the enhancements. This last
action commutes with the cuspidal support maps, so it stabilizes $\Phi_e (G,T)$.

\begin{lem}\label{lem:5.3}
The bijection in Theorem \ref{thm:5.1} is $H^1 (\mb W_F,Z(G^\vee))$-equivariant.
\end{lem}
\begin{proof}
For the fourth bijection in \eqref{eq:5.1}, such equivariance was shown in 
\cite[Lemma 2.2.a]{SolLLCunip}. Here $z \in H^1 (\mb W_F,Z(G^\vee))$ acts via the algebra 
isomorphism 
\[
\begin{array}{ccccl}
\mc H (z) : & \mc H (\mf s^\vee, q_F^{1/2}) & \to & \mc H (z \mf s^\vee, q_F^{1/2}) \\
& f N_w & \mapsto & (f \circ z^{-1}) N_w & \qquad f \in \mc O (T_{\mf s^\vee}), w \in W_{\mf s^\vee} .
\end{array}
\]
In view of Theorem \ref{thm:4.3}, the same formula also defines an algebra isomorphism
\[
\mc H (z) : \mc H (\mf s)^{op} \to \mc H (z \mf s)^{op}. 
\]
We define an action of $H^1 (\mb W_F,Z(G^\vee))$ on the union of the spaces 
$\Irr (\mc H (\mf s)^{op})$ by $z \cdot \tau = \tau \circ \mc H (z)^{-1}$. That renders the third
bijection in \eqref{eq:5.1} equivariant. Using Theorem \ref{thm:3.4} and the same argument we also 
make the second bijection in \eqref{eq:5.1} equivariant for $H^1 (\mb W_F,Z(G^\vee))$.

Finally, consider $\pi \in \Irr (G)^{\mf s}$ and $\Hom_G (\Pi_{\mf s}, \pi) \in 
\Irr (\End_G (\Pi_{\mf s})^{op})$. Then $z \otimes \pi \in \Irr (G)^{z \mf s}$ and 
\begin{multline*}
\Hom_G (\Pi_{z\mf s},z \otimes \pi) = \Hom_G \big( I_B^G \mr{ind}_{T_\cpt}^T (z \otimes \chi),
z \otimes \pi \big) \cong \\ 
\Hom_G \big( z \otimes I_B^G \mr{ind}_{T_\cpt}^T (\chi), z \otimes \pi \big) =
\Hom_G \big( I_B^G \mr{ind}_{T_\cpt}^T (\chi), \pi \big) = \Hom_G (\Pi_{\mf s}, \pi) . 
\end{multline*}
The isomorphism (from bottom to top) is given by translation by $z$ on $\Irr (T)$. As 
modules over $\mc H (\mf s)$ and $\mc H (z \mf s)$, that isomorphism is implemented by composition
with $\mc H (z)^{-1}$. Hence the first bijection in \eqref{eq:5.1} is equivariant as well.
\end{proof}

It is clear that a principal series $G$-representation is supercuspidal if and only if $G$ is a
torus. Similarly, the discussion at the start of Section \ref{sec:2} entails that a principal
series enhanced L-parameter for $G$ is cuspidal if and only if $G$ is a torus. The next result
relates the cuspidal support maps on both sides, when $G$ is not a torus.

\begin{lem}\label{lem:5.4}
Theorem \ref{thm:5.1} and the cuspidal support maps make a commutative diagram
\[
\begin{array}{ccc}
\Irr (G,T) & \longleftrightarrow & \Phi_e (G,T) \\
\downarrow \mr{Sc}  & & \downarrow \mr{Sc} \\
\Irr (T) / N_G (T) & \xrightarrow{LLC} & \Phi (T) / N_{G^\vee} (T \rtimes \mb W_F) 
\end{array}.
\]
\end{lem}
\begin{proof}
From the formula for the cuspidal support \eqref{eq:2.1} and Theorem \ref{thm:2.2}.a, we see
that the central character of $\overline{M} (\phi,\rho,q_F^{1/2})$ is given by
$\mr{Sc} (\phi,\rho) / W_{\mf s^\vee} \in \Phi_e (T) / W_{\mf s^\vee}$. Hence the central 
character of $\Hom_G (\Pi_{\mf s}, \pi (\phi,\rho))$ is the image $W_{\mf s} \chi_\phi$ of 
$\mr{Sc} (\phi,\rho) / W_{\mf s^\vee}$ in $\Irr (T) / W_{\mf s}$.

More explicitly, $\mc O (T_{\mf s})^{W_{\mf s}}$ acts on $\Hom_G (\Pi_{\mf s}, \pi (\phi,\rho))$
via $W_{\mf s} \chi_\phi$. Then a glance at the construction of $\Pi_{\mf s}$ reveals that 
$W_{\mf s} \chi_\phi$ represents the supercuspidal support of $\pi (\phi,\rho)$.
\end{proof}

We turn to more analytic properties of $G$-representations.

\begin{lem}\label{lem:5.5}
$\pi \in \Irr (G,T)$ is tempered if and only if $\phi_\pi \in \Phi (G)$ is bounded.
\end{lem}
\begin{proof}
Theorem \ref{thm:2.2}.b says that the fourth bijection in \eqref{eq:5.1} has the desired 
property. By Lemma \ref{lem:4.2} and Theorem \ref{thm:4.3}, the third bijection in \eqref{eq:5.1}
preserves temperedness. By \cite[Theorem 9.6.a]{SolEnd}, so does the composition of the first and
the second bijections in \eqref{eq:5.1}.
\end{proof}

\begin{lem}\label{lem:5.6}
$\pi \in \Irr (G,T)$ is essentially square-integrable if and only if $\phi$ is discrete.
\end{lem}
\begin{proof}
Suppose first that $R_{\mf s,\mu}$ has smaller rank that $R(\mc G,\mc S)$. By 
\cite[Theorem 9.6.b]{SolEnd}, $\Rep (G)^{\mf s}$ contains no essentially square-integrable
representations. As rk$(R(\mc G,\mc S))$ equals the $F$-split rank of $\mc G$ and
\[
\mr{rk} (R_{\mf s,\mu}) = \mr{rk} (R_{\mf s}^\vee) = \mr{rk} (R_{\mf s^\vee}) ,
\]
Theorem \ref{thm:2.2}.c says that $\Phi_e (G )^{\mf s^\vee}$ contains no discrete enhanced
L-parameters.

Now we suppose that rk$(R_{\mf s,\mu}) = \mr{rk}(R(\mc G,\mc S))$. Then \cite[Theorem 9.6.c]{SolEnd}
says that \eqref{eq:1.6} restricts to a bijection between essentially square-integrable 
representations in $\Irr (G)^{\mf s}$ and essentially discrete series representations in
$\Irr (\mc H (\mf s)^{op})$. By Lemma \ref{lem:4.2} and Theorem \ref{thm:4.3}, the latter set is
naturally in bijection with the set of essentially discrete series representations in 
$\Irr \big( \mc H (\mf s^\vee, q_F^{1/2}) \big)$. Combine that with Theorem \ref{thm:2.2}.c.
\end{proof}

Recall from \cite[p. 20--23]{Lan1} and \cite[\S 10]{Bor2} that every $\phi \in \Phi (G)$ determines
in a canonical way a character $\chi_\phi$ of $Z(G)$.

\begin{lem}\label{lem:5.7}
For any $(\phi, \rho) \in \Phi_e (G,T)$, the central character of $\pi (\phi,\rho)$ equals
$\chi_\phi$. 
\end{lem}
\begin{proof}
For any subquotient $\pi$ of $I_B^G (\chi) = \mr{ind}_B^G (\chi \otimes \delta_B^{1/2})$, the
central character of $\pi$ equals $(\chi \otimes \delta_B^{1/2}) |_{Z(G)} = \chi |_{Z(G)}$.
In particular the central character of $\pi (\phi,\rho)$ equals $\mr{Sc} (\pi (\phi,\rho)) |_{Z(G)}$.
By Lemma \ref{lem:5.4} that is $\pi (\Sc (\phi,\rho)) |_{Z(G)}$. With \eqref{eq:2.1} we write it as 
\[
\chi |_{Z(G)} \quad \text{where} \quad \hat \chi |_{\mb I_F} = \phi |_{\mb I_F} \quad \text{and} \quad
\hat \chi (\Fr_F) = \phi \big( \Fr_F, \matje{q_F^{-1/2}}{0}{0}{q_F^{1/2}} \big) .  
\]
It remains to show that $\chi_\phi$ equals $\chi |_{Z(G)}$, and to that end we revisit the construction
from \cite{Bor2,Lan1}. Let $\overline G$ be a quasi-split reductive $F$-group with connected centre,
such that $\mc G_\der = \overline{\mc G}_\der$. Let $\overline \phi \in \Phi (\overline G)$ be a lift
of $\phi \in \Phi (G)$. With the canonical map $\overline p : {}^L \overline G \to {}^L Z(\overline G)$
we obtain $\overline p (\overline \phi) \in \Phi (Z(\overline G))$.
Via the LLC for tori that gives $\chi_{\overline p (\overline \phi)} \in \Irr (Z(\overline G))$, and
by definition $\chi_\phi = \chi_{\overline p (\overline \phi)} |_{Z(G)}$. 

Let $\overline T = Z_{\overline G}(S) = Z_{\overline G} (\overline T)$. From \eqref{eq:2.1} we see that,
for any enhancement $\overline \rho$ of $\overline \phi$ such that $(\overline \phi, \overline \rho)
\in \Phi_e (\overline G, \overline T)$, we have $\Sc (\overline \phi, \overline \rho) = (\overline \psi,
\overline \epsilon)$, where $\overline \psi \in \Phi (\overline T)$ is a lift of $\hat \chi \in  
\Phi (T)$. As $\overline \phi$ and $\overline \psi$ differ only by elements of ${\overline G^\vee}_\der
\subset \ker (\overline p)$, we have $\overline p \overline \phi = \overline p \overline \psi$.
By the naturality of the LLC for tori, $\chi_{\overline \psi}$ extends both $\chi \in \Irr (T)$ and
$\chi_{\overline p \overline \psi} = \chi_{\overline p \overline \phi} \in \Irr (Z(\overline G))$. 
Hence $\chi |_{Z(G)} =  \chi_{\overline p \overline \phi} |_{Z(G)} = \chi_\phi$.
\end{proof}

Suppose that $P = M R_u (P)$ is a parabolic subgroup of $G$, where $M$ is a Levi factor of $P$ and
$T \subset M$. We can use the normalized parabolic induction functor $I_P^G$ to relate
representations of $M$ and of $G$. 

The restriction of $\xi$ to $U \cap M$ is a nondegenerate character $\xi_M$. We use $(U \cap M,\xi_M)$
to define genericity of $M$-representations and to normalize the LLC for $\Irr (M,T)$.

Suppose furthermore that $\phi \in \Phi (G)$ factors via $\Phi (M)$.  By \cite[Theorem 7.10.a]{AMS1}
the group $R_\phi^M = \pi_0 (Z_{M^\vee} (\phi) / Z(M^\vee))$ injects naturally into $R_\phi$. 
Hence any enhancement of $\phi \in \Phi (G)$ can be considered as (possibly reducible) representation 
of $R_\phi^M$.

\begin{lem}\label{lem:5.8}
Let $(\phi, \rho^M) \in \Phi_e (M,T)$ be bounded. Then
\[
I_P^G \pi (\phi, \rho^M) \cong 
\bigoplus\nolimits_\rho \Hom_{R_\phi^M} (\rho^M, \rho) \otimes \pi (\phi,\rho),
\]
where the sum runs over all $\rho \in \Irr (R_\phi)$ with $\Sc (\phi,\rho) = \Sc (\phi,\rho^M)$.
\end{lem}
\begin{proof}
By \cite[Theorem 3.18.f and Lemma 3.19.a]{AMS3}, the analogous statement holds for $\mc H (\mf s^\vee,
q_F^{1/2})$-modules. Theorem \ref{thm:4.3} (which is compatible with parabolic induction) entails
it also holds for $\mc H (\mf s)^{op}$-modules. Then \eqref{eq:1.7} enables us to transfer the 
desired statement from $\mc H (\mf s)^{op}$ to $\Rep (G)^{\mf s}$. 
\end{proof}

\noindent
Recall that the Langlands classification for irreducible $G$-representations \cite{Lan1,Ren}
associates to any $\pi \in \Irr (G)$ a unique standard parabolic subgroup $P = M R_u (P)$, a unique
tempered $\tau \in \Irr (M)$ and a unique strictly positive $z \in \Hom (M,\R_{>0})$, such that
$\pi$ is the unique irreducible quotient of the standard module $I_P^G (\tau \otimes z)$.
It has a counterpart for (enhanced) L-parameters \cite{SiZi}. Let $(\phi, \rho) \in \Phi_e (G,T)$
and let $(P = M R_u (P), \phi_b, z)$ be the triple associated to $\phi$ by \cite[Theorem 4.6]{SiZi}.
Here $\phi_b \in \Phi (M)$ is bounded and $z \in X_\nr (M) \cong (Z (M^\vee)^{\mb I_F} )^\circ_{\mb W_F}$
is ``strictly positive with respect to $P$". By \cite[Theorem 7.10.b]{AMS1} there are natural
isomorphisms
\[
R_{\phi_b}^M \cong R_{z \phi_b}^M = R_\phi^M \cong R_\phi .
\]
Hence $\rho$ can also be regarded as enhancement of $\phi \in \Phi (M)$ or $\phi_b \in \Phi (M)$.

\begin{lem}\label{lem:5.9}
In the above setting:
\enuma{
\item $\pi (\phi,\rho)$ is the unique irreducible quotient of $I_P^G \pi^M (\phi, \rho)$.
\item $\pi^M (\phi,\rho) = \pi^M (z \phi_b, \rho) = z \otimes \pi^M (\phi_b,\rho)$ with
$\pi^M (\phi_b,\rho) \in \Irr (M)$ tempered. 
\item The triple associated to $\pi (\phi,\rho)$ by the Langlands classification for $\Irr (G)$ is
$(P,\pi^M (\phi_b,\rho),z)$.
}
\end{lem} 
\begin{proof}
(a) By \cite[Proposition 2.3]{SolLLCunip}, the analogue in $\Rep (\mc H (\mf s^\vee, q_F^{1/2}))$
holds. As in the proof of Lemma \ref{lem:5.8}, that can be transferred to $\Rep (G)^{\mf s}$
via \eqref{eq:1.7}.\\
(b) This is a direct consequence of Lemmas \ref{lem:5.3} and \ref{lem:5.5}.\\
(c) This follows from parts (a) and (b) and the uniqueness in the Langlands classification.
\end{proof}

Suppose that $F' / F$ is a finite extension inside the fixed separable closure $F_s$. Let 
$\mc G'$ be a quasi-split $F'$-group and put $\mc G = \mr{Res}_{F'/F} (\mc G')$. Then
$\mc G (F) = \mc G' (F')$, so there is a natural bijection $\Irr (\mc G (F)) \to \Irr (\mc G' (F'))$.
On the other hand, Shapiro's lemma provides a natural isomorphism
\[
\mr{Sh} : \Phi_e (\mc G (F)) \to \Phi_e (\mc G' (F')) ,
\]
see \cite[Lemma A.3]{FOS1}.

\begin{lem}\label{lem:5.10}
The bijection in Theorem \ref{thm:5.1} is compatible with restriction of scalars, in the sense that
the following diagram commutes:
\[
\begin{array}{ccc}
\Irr (\mc G (F), \mc T (F)) & \to & \Phi_e (\mc G (F), \mc T (F)) \\
\downarrow \mr{Sh} & & \downarrow \mr{Sh} \\
\Irr (\mc G' (F'), \mc T' (F')) & \to & \Phi_e (\mc G' (F'), \mc T' (F')) 
\end{array}
\]
Here $\Res_{F'/F} \mc T' = \mc T$.
\end{lem} 
\begin{proof}
By \cite[(26)]{SolLLCunip}, Sh induces a bijection from the set of Bernstein components of 
$\Phi_e (\mc G (F))$ to the analogous set for $\mc G' (F')$. This bijection commutes with the 
cuspidal support maps, so it also applies to $\Phi_e (\mc G (F), \mc T (F))$ and 
$\Phi_e (\mc G' (F'), \mc T' (F'))$. Whenever $\mf s^\vee$ corresponds to $\mf s'^\vee$, there is
a natural algebra isomorphism $\mc H (\mf s^\vee, q_F^{1/2}) \cong \mc H (\mf s'^\vee, q_{F'}^{1/2})$
\cite[Lemma 2.4]{SolLLCunip}. Combine that with \eqref{eq:5.1}.
\end{proof}

Finally we investigate in what sense our (enhanced) L-parameters are unique.

\begin{lem}\label{lem:5.12}
Let $\pi \in \Irr (G,T)$. Then the $\phi_\pi$ from Theorem \ref{thm:5.1} is uniquely determined by
Lemmas \ref{lem:5.3}, \ref{lem:5.4}, \ref{lem:5.5} and \ref{lem:5.9}.
\end{lem}
\begin{proof}
Suppose that $\pi$ is tempered. Lemma \ref{lem:5.4} determines Sc$(\phi_\pi,\rho_\pi) = \tilde \phi$
up to $N_{G^\vee}(T^\vee \rtimes \mb W_F)$. Lemma \ref{lem:5.6} says that $\phi_\pi$ must be bounded,
so according to \cite[\S 0.6]{CFZ} $\phi_\pi$ is an open Langlands parameter. In other words,
$u_{\phi_\pi}$ is uniquely determines (up to $Z_{G^\vee}(\tilde{\phi_\pi}(\mb W_F))$-conjugacy) 
as an element of the open orbit in $V_{\tilde{\phi_\pi}}$. Thus $\phi_\pi$ is unique up to 
$G^\vee$-conjugacy.

Suppose now that $\pi$ is not tempered. Let $(P,\tau,z)$ be the triple associated to $\pi$ by the 
Langlands classification. Here $\tau$ is tempered, so the above determines $\phi_\tau \in \Phi 
(P/R_u (P),T)$ uniquely. Then Lemma \ref{lem:5.3} forces $\phi_{\tau \otimes z} =z \cdot \phi_\tau$ 
and Lemma \ref{lem:5.9} says that $\phi_\pi$ equals $z \phi_\tau$ up to $G^\vee$-conjugacy. 
\end{proof}

It is less clear to what extent the enhancement $\rho_\pi$ of $\phi_\pi$ is uniquely specified.
Lemma \ref{lem:5.9} reduces this issue to tempered $\pi \in \Irr (G,T)$. Then $\phi_\pi$ is bounded,
so open. By Lemma \ref{lem:5.11} the L-packet $\Pi_{\phi_\pi}(G)$ contains a unique generic member,
namely $\pi (\phi_\pi,\mr{triv})$. That fixes the normalization of the interwining operators from
elements of $R_{\phi_\pi}$, which then determines $\pi (\phi_\pi,\rho)$ for any $\rho \in \Irr
(R_{\phi_\pi})$ such that $(\phi_\pi,\rho) \in \Phi_e (G,T)$. However, to make that precise one
has to say on which module these intertwining operators acts. That involves the constructions with
Hecke algebras in Sections \ref{sec:4} and \ref{sec:red}. Those are canonical, but they may not be
unique, see Remarks \ref{rem:4.5} and \ref{rem:5.13}.\\

\noindent\textbf{Acknowledgements.}
The author thanks Jessica Fintzen and Tasho Kaletha for interesting discussions about this paper.
He also thanks the referee for his or her work.

\end{document}